\documentclass[11pt]{article}
\usepackage[utf8]{inputenc}
\usepackage[T1]{fontenc}
\usepackage{comment}
\usepackage{authblk}
\usepackage{centernot}
\usepackage{relsize}
\usepackage{changepage}
\usepackage{mathtools}
\mathtoolsset{showonlyrefs}
\usepackage[english]{babel}
\usepackage{amsmath}
\usepackage{amsthm}
\usepackage{amsfonts}
\usepackage{esvect}
\usepackage{amssymb}
\usepackage{graphicx}
\usepackage{bbm}
\usepackage[usenames,dvipsnames]{xcolor}
\theoremstyle{plain}
\newtheorem{theorem}{Theorem}[section]
\newtheorem{corollary}[theorem]{Corollary}
\newtheorem{proposition}[theorem]{Proposition}
\newtheorem{lemma}[theorem]{Lemma}

\theoremstyle{definition}
\newtheorem{definition}[theorem]{Definition}
\newtheorem{remark}[theorem]{Remark}

\usepackage{empheq}

\makeatother
\definecolor{lyucol}{rgb}{0.5, 0.1, 0.1}
\definecolor{gray}{rgb}{0.5, 0.5, 0.5}

\newcommand{\lyu}{}
\newcommand{\bas}{}

\newcommand{\barv}{v}

\renewcommand{\P}[1]{\mathbb P\left(#1\right)}
\newcommand{\E}[1]{\mathbb E\left[#1\right]}
\newcommand{\e}{\mathrm e}
\newcommand{\N}{\mathbb N}
\newcommand{\R}{\mathbb R}
\newcommand{\Z}{\mathbb Z}
\newcommand{\wt}{\widetilde}
\newcommand{\luh}{<_{\mathrm{UH}}}
\newcommand{\lequh}{\leq_{\mathrm{UH}}}

\newcommand{\be}{\begin{equation}}
\newcommand{\ee}{\end{equation}}
\newcommand{\ba}{\begin{aligned}}
\newcommand{\ea}{\end{aligned}}
\newcommand{\dd}{\mathrm{d}}
\newcommand{\cA}{\mathcal A}
\newcommand{\cB}{\mathcal B}
\newcommand{\cC}{\mathcal C}
\newcommand{\cD}{\mathcal D}
\newcommand{\cE}{\mathcal E}
\newcommand{\cF}{\mathcal F}
\newcommand{\cG}{\mathcal G}
\newcommand{\cH}{\mathcal H}
\newcommand{\cI}{\mathcal I}
\newcommand{\cK}{\mathcal K}
\newcommand{\cL}{\mathcal L}

\newcommand{\cN}{\mathcal N}
\newcommand{\cO}{\mathcal O}
\newcommand{\cP}{\mathcal P}
\newcommand{\cQ}{\mathcal Q}
\newcommand{\cR}{\mathcal R}
\newcommand{\cS}{\mathcal S}
\newcommand{\cT}{\mathcal T}
\newcommand{\cV}{\mathcal V}
\newcommand{\cW}{\mathcal W}

\newcommand{\toindis}{\overset d\longrightarrow}
\newcommand{\toinp}{\overset{\mathbb P}{\longrightarrow}}
\newcommand{\toas}{\overset{\mathrm{a.s.}}{\longrightarrow}}
\newcommand{\ind}{\mathbbm 1}

\usepackage[left=1.9cm,right=1.9cm,top=1.9cm,bottom=1.9cm]{geometry}
\usepackage{amssymb}
\usepackage[colorlinks=true, linkcolor=black, urlcolor=black]{hyperref}
\makeatletter

\newlength{\dhatheight}

\usepackage[normalem]{ulem}
\usepackage{dsfont}
\usepackage{accents}
\usepackage{verbatim}
\usepackage{ulem}
\usepackage{pgf,tikz,pgfplots}
\pgfplotsset{compat = 1.17}
\usepackage[all]{xy} 

\newcommand{\eps}{\varepsilon}

\usepackage{subfig}

\usepackage{enumerate}

\renewcommand{\P}[1]{\mathbb P\left(#1\right)}

\renewcommand{\P}[1]{\mathbb P\left(#1\right)}

\newcommand{\invisible}[2]{%
\ifthenelse{\isempty{#1}}
{}
{#2}
}

\newcommand{\Pk}[1]{\mathbb P_{k_0}\left(#1\right)}
\newcommand{\Ek}[1]{\mathbb E_{k_0}\left[#1\right]} 
\newcommand{\Pq}[1]{\mathbb P_q\left(#1\right)}
\newcommand{\Eq}[1]{\mathbb E_q\left[#1\right]}

\title{\scshape
  Dynamic random graphs with vertex removal}

\author[1]{Josep D\'iaz\footnote{Supported by grant MOTION, PID2020-112581GB-C21 from MCIN/ AEI /10.13039/501100011033}}
\author[2,3]{Lyuben Lichev}
\author[2,3]{Bas Lodewijks\footnote{Supported by grant GrHyDy ANR-20-CE40-0002}}
\affil[1]{Universitat Polit\`ecnica de Catalunya, Barcelona, Spain}
\affil[2]{Univ. Jean Monnet, Saint-Etienne, France}
\affil[3]{Institut Camille Jordan, Lyon, France}

\begin{document}

\maketitle

\begin{abstract}
We introduce and analyse a Dynamic Random Graph with Vertex Removal (DRGVR) defined as follows. At every step, with probability $p > 1/2$, a new vertex is introduced and, with probability $1-p$, a vertex chosen uniformly at random among the present ones (if any) is removed from the graph together with all edges adjacent to it. In the former case, the new vertex connects by an edge to every other vertex with probability \lyu{inversely} proportional to the number of vertices already present.

We prove that the DRGVR converges to a local limit and determine this limit. Moreover, we analyse its component structure and distinguish a subcritical and a supercritical regime with respect to the existence of a giant component. As a byproduct of this analysis, we obtain upper and lower bounds for the critical parameter. Furthermore, we provide precise expression of the maximum degree (as well as in- and out-degree for a natural orientation of the DRGVR). Several concentration and stability results complete the study.
\end{abstract}

\section{Introduction}\label{sec:intro}

Since the appearance of the first random graph models in the late 50's~\cite{ER59,Gilbert59}, the study of random graphs has attracted great interest both from theoretical and applied point of view. 

In the last twenty years, apart from the study of classical random graph theory, much effort has been made to provide and analyse accurate models for large-scale real-world networks like the internet or different P2P social networks. One key characteristic of such networks is that for most of them, the degree of the vertices follows an inverse power law, that is, the proportion of vertices with degree $d$ is approximately $cd^{-\alpha}$ for some constants $c,\alpha > 0$. 
As a consequence, the fact that the classical random graphs as $G(n,m)$ and $G(n,p)$ typically have Poisson degree distribution makes them unsuitable as models for real-world networks. An important contribution was made by Albert and Barab\'asi~\cite{Barabasi99} who proposed a scale-free dynamic network model with a preferential attachment rule. 
In their model, a new vertex is introduced at every step. At its arrival, a vertex chooses to attach to existing vertices with probability proportional to the degree of the existing vertices. Since Albert and Barab\'asi's breakthrough, the analysis of their scale-free model and its variations has been object of intensive study, see for example~\cite{Mitzenmacher01,van2017}.

Although the power-law behaviour is an important characteristic of real-world networks, it is not the only one. One natural observation is that besides the creation of vertices in a dynamic network, sometimes vertices could also disappear: 
for instance, in a social network, a user could delete their account or, in a web network, a server could break down due to failure. 
However, most dynamic network models did not consider the possibility of vertex removal. Among the few exceptions are~\cite{CL04,Cooper04,Dei10}, which both consider settings based on Albert and Barab\'asi's scale-free model. 
In the model of Chung and Lu~\cite{CL04}, each of vertex creation, vertex deletion, edge creation and edge deletion may happen at any given step with probabilities that sum up to 1. 
When a vertex is created, it connects to an existing vertex by a single edge according to a preferential attachment rule. The authors provide bounds for the diameter, the graph distance between two typical vertices, the connected components and the spectrum of the adjacency matrix. 
In the similar model of Cooper, Frieze and Vera~\cite{Cooper04}, every vertex is connected to $m\ge 1$ neighbours at its arrival. The authors analyse the degree sequence of the obtained dynamic graph and prove its scale-freeness. \bas{Finally, the tree model of Deijfen~\cite{Dei10} considers vertex creation and vertex ``death'', where dead vertices no longer make new connections but are still considered a part of the tree. A precise and general result for the limiting degree distribution is obtained.
Furthermore, the recent works of Bellin, Blanc-Renaudie, Kammerer, and Kortchemski~\cite{BelBlaKamKort23.1,BelBlaKamKort23.2} consider the height and the scaling limit of the uniform attachment tree with freezing.
There, at every step, a vertex is either added or frozen (equivalent to vertex death in the model of Deijfen).}

The notion of a dynamic network on a set of aging individuals has also been considered. Instead of deleting a vertex at a given step, which forbids later vertices to attach to it, several works~\cite{BSW13, DHKS17, GvdHW17} considered models in which new vertices connect to any fixed vertex with smaller probability as time goes by. 
The main motivation of such models comes from real-world networks like the citation network where empirical observations show that papers usually get less popular with time. 
Contrary to~\cite{CL04, Cooper04} and similarly to the model we introduce below, the authors of~\cite{GvdHW17} consider a setting where every new vertex connects to the present ones by a random number of edges. This allows new vertices to occasionally have large degrees even upon their arrival, which is another common feature of a number of real-world networks.

In this paper, we introduce and study a new model of a dynamically growing network with vertex removal based on uniform attachment.

\begin{definition}[Dynamic Random Graph with Vertex Removal (DRGVR)]\label{def:drgvr}
	\hspace{0.2em} Fix constants $\beta>0$, $\eps \in (0,1/2]$ and a sequence $(\xi_n)_{n\ge 0}$ of i.i.d.\ random variables with Bernoulli distribution with parameter $p:=1/2+\eps$. 
	We define a sequence of graphs $(G_n)_{n\ge 0} = ((V_n, E_n))_{n\ge 0}$ by initialising $G_0$ to be the empty graph and, for every $n\ge 1$, we construct $G_n$ from $G_{n-1}$ as follows. 
	If $\xi_n=1$, set $V_n = V_{n-1}\cup \{n\}$ and, conditionally on $V_{n-1}$, \bas{if $V_{n-1}\neq\emptyset$,} add each of the edges $(\{i,n\})_{i\in V_{n-1}}$ to $E_{n-1}$ independently with probability $\min\big\{\tfrac{\beta}{|V_{n-1}|},1\big\}$ to construct $E_n$ \bas{and, if $V_{n-1} = \emptyset$, set $E_n=E_{n-1}=\emptyset$.}
	If $\xi_n=0$, select a vertex uniformly at random from $V_{n-1}$ (if any) and remove it together with all edges incident to~it. \lyu{Finally, we define the \emph{mark function} $M_n:v\in V(G_n)\mapsto v/n\in [0,1]$.}
\end{definition}

\bas{The graph $G_n$ being constructed recursively, its edges may naturally be equipped with an orientation from the vertex arriving at a later step to the vertex arriving earlier.}  
In this case, the edge set is denoted by $\vv{E}_n$, and the directed graph itself by $\vv{G}_n$.

\bas{Though similar to the model of Bellin et al.~\cite{BelBlaKamKort23.1,BelBlaKamKort23.2}, te main differences with the model studied here are $(1)$ vertices and edges are not removed, frozen vertices simply cannot make new connections $(2)$ addition and freezing of vertices can also take place according to a predetermined deterministic rule, whereas we only consider addition or removal at every step according to the outcome of i.i.d.\ Bernoulli random variables $(3)$ the model of Bellin et al.\ is a tree, whereas the model studied here construct a graph (though the analysis can be adapted to a tree model, where newly added vertices connect to one other vertex, too).}

Finally, we remark that the well-studied \emph{Dubin's model}~\cite{Dur03, KW88, Shepp89} is a particular case of the DRGVR for $\eps=1/2$ (or equivalently $p=1$). In this particular setting Shepp~\cite{Shepp89} showed the existence of a sharp threshold at $\beta = 1/4$ for the appearance of a giant component, and Dereich and M\"orters later recovered the result in a more general framework, see Proposition~1.3 in~\cite{DerMor13}. Similar results were obtained in~\cite{DorMenSam01, Dur03} and for the closely related CHKNS model.

\paragraph{Notation.} Throughout the paper, we write $\N:=\{1,2,\ldots\}$ for the set of natural numbers, $\N_0:=\mathbb N\cup \{0\}$ and, for every $t\geq 1$,  $[t]:=\{1,2,\ldots,t\}$. For $x\in\R$, we let $\lceil x\rceil:=\min\{n\in\Z: n\geq x\}$ and $\lfloor x\rfloor:=\max\{n\in\Z: n\leq x\}$. For positive real sequences $(a_n,b_n)_{n\ge 1}$, we say that $a_n=o(b_n)$ if $\lim_{n\to\infty} a_n/b_n=0$, $a_n=\omega(b_n)$ if $b_n = o(a_n)$, $a_n\sim b_n$ if $\lim_{n\to\infty} a_n/b_n=1$, $a_n=\mathcal{O}(b_n)$ if there exists a constant $C>0$ such that $a_n\leq Cb_n$ for all $n\ge 1$, and $a_n=\Omega(b_n)$ if $b_n = \cO(a_n)$. 
By default, we allow the constant $C$ to depend on the parameters of the problem that are \emph{fixed}, that is, do not depend on the choice of a sufficiently large $n$; when this is not the case, we explicitly mention that the constant is absolute. 

For a graph $G$, the \emph{size} of $G$, denoted $|G|$, is the number of vertices in $G$. 
For two vertices $u,v$ in $G$, we let $u\to v$ denote the fact that $u$ is connected to $v$ by a (directed) edge.
Moreover, for two finite rooted graphs $(G,o_G)$ and $(H,o_H)$ (i.e.\  graphs with distinguished vertices, called the root), we let $(G,o_G)\cong (H,o_H)$ denote the fact that the two rooted graphs are isomorphic. \lyu{Note that, when roots are clear from the context, they are sometimes omitted to improve readability.}
\lyu{The space of finite rooted graphs is denoted by $\cG$ and is equipped with the metric 
	\[d: ((G,o_G),(H,o_H))\mapsto \inf\{2^{-r}: B_r(G,o_G)\cong B_r(H,o_H)\}.\]}
For random variables $(X_n)_{n\ge 1}$ and $X$, we let $X_n\toindis X$, $X_n\toinp X$ and $X_n\toas X$ denote convergence in distribution, in probability and almost sure convergence of $(X_n)_{n\ge 1}$ to $X$, respectively. 
Furthermore, we let $\mathfrak{L}(X)$ denote the distribution of $X$. 
Finally, for an event $\cE$ in a given probability space, we let $\cE^c$ denote the \emph{complement} of $\cE$.

\lyu{
\subsection{A brief introduction to local convergence}
	To introduce our results, we first briefly explain several variants of the notion of \emph{local convergence} of a sequence of graphs; a more detailed account on the topic can be found in Sections~2.3 and~2.4 in~\cite{Hof24} and in~\cite{GHL20,Sal11}. 
	For a graph $H$, a vertex $v$ in $H$ and an integer $r\ge 0$, the \emph{ball with radius $r$ around $v$ in $H$}, denoted $B_r(H, v)$, is the graph whose vertices are the ones at distance at most $r$ from $v$ in $H$ and whose edges are the ones in $H$ containing an endpoint at distance at most $r-1$ from $v$. (Note that the ball $B_r(H, v)$ is often defined as the graph induced by the vertices at distance at most $r$ from $v$. In our work, it turns out to be more convenient to exclude the edges between the vertices in the last layer.) We often assume that $B_r(H,v)$ is a graph rooted at $v$. 
	
	Fix a sequence of (deterministic or random) finite rooted graphs $(H_n,o_n)_{n\ge 1}$, where $o_n$ is a vertex of $H_n$ chosen uniformly at random.
	Given a probability measure $\mu$ on the set of finite rooted graphs $\cG$, we say that 
	\begin{itemize}
		\item $(H_n,o_n)_{n\ge 1}$ \emph{converges locally in distribution} to $\mu$ if, for every integer $r\ge 1$ and every finite rooted graph $(\bar{H},\bar{o})$,
		\[\mathbb P(B_r(H_n,o_n) = (\bar{H},\bar{o}))\to \mu(\{(H,o)\in \cG: B_r(H,o) = (\bar{H}, \bar{o})\}).\]
		Equivalently, $(H_n,o_n)_{n\ge 1}$ \emph{converges locally weakly} to $\mu$ if
		\[\mathbb E[h(H_n,o_n)]\to \mu(h)\]
		for every bounded and continuous function $h: \cG\to \mathbb R$, where $\mu(h)$ denotes the expectation of $h$ with respect to $\mu$.
		\item $(H_n,o_n)_{n\ge 1}$ \emph{converges locally in probability} to $\mu$ if the sequence of random variables $\mathbb E[h(H_n,o_n)\mid H_n]$ converges in probability to $\mu(h)$ for every bounded and continuous function $h: \cG\to \mathbb R$.
		\item $(H_n,o_n)_{n\ge 1}$ \emph{converges locally almost surely} to $\mu$ if the sequence of random variables $\mathbb E[h(H_n,o_n)\mid H_n]$ converges almost surely to $\mu(h)$ for every bounded and continuous function $h: \cG\to \mathbb R$.
	\end{itemize}
	In each case, the limiting measure $\mu$ is often identified with a (possibly random) rooted graph.

    \bas{\paragraph{Local convergence in distribution of marked rooted graphs.} 
    We now extend the concept of local convergence in distribution to marked local convergence in distribution.
    Marked local convergence considers marked rooted graphs $(G,o,M(G))$ where $M$ is a map from $V(G)$ to a Polish space $\Xi$ endowed with a metric $\dd_\Xi$. \lyu{The space of marked finite rooted graphs is denoted by $\cG_m$ (where $m$ is not an index and stands for ``marked'') and is naturally equipped with the metric
    \begin{align*}
    d_m: ((G,o_G,M_G),(H,o_H,M_H))\mapsto \inf\{2^{-r}:\, &B_r(G,o_G)\overset{\phi}{\cong} B_r(H,o_H)\quad\text{and}\\
    &\forall v\in B_r(G,o_G), |M_G(v)-M_H(\phi(v))|\le 1/r\}.
    \end{align*}}
    }} \lyu{A sequence of marked rooted graphs $(H_n,o_n,M_n)_{n\geq 1}$ then converges locally in distribution to a measure $\mu$ on $\cG_m$ if, for every integer $r\geq 1$, every finite rooted graph $(\bar H, \bar o)$ and every family of measurable sets $(A_v)_{v\in V(\bar H)}\subseteq \Xi$, 
 \be \ba 
 \mathbb P({}&\exists\, \phi_n: B_r(H_n,o_n)\overset{\phi_n}{\cong}(\bar H,\bar o), \ M_{n}(\phi_n^{-1}(v))\in A_{v} \text{ for all } v\in V(\bar H)))\\ 
 &\to \mu(\{(H,o,\lyu{M_H})\in \cG_m: B_r(H,o)\overset{\phi}{\cong}(\bar H,\bar o),\ M_H(\phi^{-1}(v))\in A_{v}\text{ for all }v\in V(\bar H)\}).
 \ea\ee}
 In this paper, we concentrate on local convergence in distribution \bas{of marked rooted random graphs, where we work with the mark space $\Xi:=[0,1]$ endowed with the Euclidean metric}.
	In fact, we provide a stronger quantitative version of that convergence for the DRGVR in terms of the \emph{total variation distance} defined as follows:
	For two probability distributions $\mu_1$ and $\mu_2$ defined on a common probability space $\Omega$, the total variation distance between $\mu_1$ and $\mu_2$ is defined as 
	\begin{equation}\label{eq:dTV}
 \begin{split}
     d_{\mathrm{TV}}(\mu_1,\mu_2)
     &= \sup_{A} |\mu_1(A)-\mu_2(A)|\\
     &=\inf\{\lyu{2\sup_{A}\P{\{X\in A\}\setminus \{Y\in A\}}}: (X,Y)\text{ is a coupling of }(\mu_1,\mu_2)\},
 \end{split}
	\end{equation} 
	\lyu{where the supremum is taken over a collection of measurable events generating the underlying $\sigma$-algebra (in many cases, the cylindric events).}
 By taking $\Omega=\cG_{\lyu{m}}$, we consider the total variation distance between the distributions of \bas{$B_r(G_n,k_0,M_\lyu{n})$} and \bas{$B_r(\cT,0,M_\infty)$}, where $(\cT,0)$ is a rooted random graph defined below and $k_0$ is a vertex chosen uniformly at random from $G_n$. \bas{Here, we abuse notation to let $B_r(G_n,k_0,M_\lyu{n})$} denote the $r$-neighbourhood of $k_0$ in $G_n$, including the marks of the vertices in the $r$-neighbourhood. 
 
    Throughout the paper, the definition of the total variation distance in terms of couplings is particularly useful.

\subsection{Results}

The first and main result of this paper deals with the local convergence \lyu{in distribution} of the DRGVR model. To this end, we define the following multi-type branching process.

\begin{definition}[Binomial birth-death tree]\label{def:wll}
	Fix $\beta>0$, $\eps\in (0,1/2]$ and set $p := 1/2 +\eps$.
	We define a multi-type branching process $\cT$ with type space $(0,1]$ as follows. The root $0$ of $\cT$ has type $a_0\sim \mathrm{Beta}(\tfrac{p}{2\eps},1)$. Then, any vertex $v$ in $\cT$ with type $a_v$ produces an offspring independently of all other vertices according to a Poisson point process on $(0,1]$. Conditionally on $a_v$, the \lyu{intensity} of this Poisson process is given by
	\begin{equation}\begin{aligned}\label{eq:lambda}
			& \lambda_v^-(\dd x):=\frac{\beta p}{2\eps a_v}x^{(1-p)/2\eps}\,\dd x \quad&&\text{ for } x\in(0,a_v],\quad \text{and}\\
			& \lambda^+(\dd x):=\frac{\beta p}{2\eps}x^{(1-p)/(2\eps)-1}\,\dd x \quad &&\text{ for } x\in(a_v,1].
	\end{aligned}\end{equation}
	The types of the vertices in the offspring are identified by their position in the Poisson point process. \bas{Finally, define the mark function \lyu{$M_{\infty}:v\in V(\cT)\mapsto a_v\in [0,1]$.}}
\end{definition}

\begin{theorem}\label{thrm:LWC intro}
	Fix $\beta>0$ and $\eps\in(0,1/2]$, and consider the DRGVR model and the Binomial birth-death tree given in Definitions~\ref{def:drgvr} and~\ref{def:wll}, respectively. Let $k_0$ be a vertex selected uniformly at random from $V_n$. Then, for any $r\in\N$,\bas{
	\begin{equation}
		 d_{\mathrm{TV}}(\mathfrak{L}(B_r(G_n,k_0,M_\lyu{n})),\mathfrak{L}(B_r(\cT,0,M_{\infty})))\leq (\log\log n)^{-(1-p)/(2\eps)-1/r}. 
	\end{equation} 
	In particular, $(G_n,k_0,M_\lyu{n})_{n\ge 1}$ converges locally \lyu{in distribution} as a marked rooted graph to the random marked rooted tree $(\cT, 0, M_{\infty})$. }
\end{theorem}

\lyu{We remark that local convergence in probability and almost sure local convergence of the randomly rooted graphs $(G_n)_{n\ge 1}$ cannot be deduced from Theorem~\ref{thrm:LWC intro}.}
\lyu{At the same time, the existence of a local limit is novel even in the setting of Dubin's model (that is, the case $p=1$ of the DRGVR).}

The following corollary is obtained by applying Theorem~\ref{thrm:LWC intro} for $r=1$ and integrating the \lyu{intensity} of the Poisson process associated to the offspring of the origin over the interval $(0,1]$. \lyu{For every $\lambda > 0$, let $\mathrm{Po}(\lambda)$ denote the Poisson distribution with parameter~$\lambda$.}

\begin{corollary}\label{cor:deg}
Let $D_n^0$ denote the degree of the vertex $k_0$ and, conditionally on $a_0$, let 
$$X_0\sim \lyu{\mathrm{Po}}\Big(\frac{\beta p}{1-p}\Big(1-\frac{2p-1}{p}a_0^{(1-p)/(2\eps)}\Big)\Big).$$
Then, $\dd_\mathrm{TV}(\mathfrak{L}(D^0_n),\mathfrak{L}(X_0))\leq (\log\log n)^{-p/(2\eps)}$.
\end{corollary}

\lyu{We note that Corollary~\ref{cor:deg} was already obtained in~\cite{BolJanRio05} without explicit bound on the error probability for a related inhomogeneous random graph model.}




\lyu{While the DRGVR is defined via a natural stochastic dynamics, its Poisson degree sequence derived in~Corollary~\ref{cor:deg} suggests a possible connection with some inhomogeneous Erd\H{o}s-R\'enyi model. 
In Section~\ref{sec inhom E-R}, we show that such a connection exists in a strong sense.
More precisely, we construct a coupling between the graph $G_n$ and two different (but closely related) inhomogeneous Erd{\H o}s-R\'enyi graphs $G^{(1)}_n$ and $G^{(2)}_n$ such that $G^{(1)}_n\subseteq G_n \subseteq G^{(2)}_n$. It allows us to ignore the randomness coming from the arrival times of the vertices in $V_n$ and fit this model in a common framework developped by Bollob\'as and Riordan~\cite{BolJanRio05}.
This comparison has several more global important consequences.}
First, we discuss the conditions for the emergence of a giant component and the relation with the local limit. Define
\begin{equation}\label{eq:survprob} 
\gamma = \gamma(\beta):=\P{|\cT|=\infty}
\end{equation} 
to be the survival probability of the branching process $\cT$. 

\begin{theorem}\label{thm components intro}
Fix $\beta > 0$ and $\eps\in (0,1/2]$, and consider the DRGVR model given in Definition~\ref{def:drgvr}. Moreover, consider the survival probability $\gamma$ of its local limit given in~\eqref{eq:survprob}. Let $\cC_1 = \cC_1(G_n)$ and $\cC_2 = \cC_2(G_n)$ denote the first and the second largest component in $G_n$. Then, there exists $\beta_c = \beta_c(p)$ such that:
\begin{enumerate}
	\item if $\beta < \beta_c$, then $|\cC_1|/n$ converges in probability to $0$ as $n\to \infty$;
	\item if $\beta > \beta_c$, then $|\cC_1|/n$ converges in probability to $2\eps \gamma\in(0,1]$ and, moreover, $|\cC_2| = \cO(\log n)$ with high probability.
\end{enumerate}
\end{theorem}

\begin{remark} 
$(i)$ 
\lyu{Let us give an intuitive explanation for the constant $2\eps \gamma$ in the second part of Theorem~\ref{thm components intro}. On the one hand, the local convergence in distribution of $(G_n,k_0)$ to $(\cT,0)$ shows that the number of vertices in components of size $k\ge 1$ in $G_n$ converges to the probability that $|\cT| = k$.
	Thus, up to showing that only $o(n)$ vertices typically participate in components $\cC$ of size $|\cC| = o(n)$ with $|\cC|\to \infty$, the proportion of vertices in the giant component $\cC_1$ has to converge to the survival probability $\gamma$ of $\cT$. 
	On the other hand, simple concentration arguments show that $|G_n|/n$ converges almost surely to $2\eps$. Combining the two observations justifies the result.}

\noindent
$(ii)$ It follows directly from the definition of the DRGVR model that \emph{bond percolation} on $G_n$ with retention probability $q$ yields a graph with a giant component only when $\beta>\beta_c$ and $q>\beta_c/\beta$.  
\end{remark} 

\bas{
\begin{remark}\label{rem:brw}
    One can formulate the local weak limit $(\cT,o)$ as a branching random walk on $(-\infty,0]$ with a killing barrier at zero. 
    To this end, start with an initial vertex located at position $y_0=\log a_0$. 
    Then, each vertex produces children conditionally on its location $y<0$ according to a Poisson point process $\Pi_y$ on $\R$ with intensity
    \be 
    \pi_y(\dd x):= \frac{\beta p}{2\eps}\e^{\frac{1-p}{2\eps}y}\big(\e^{\frac{p}{2\eps} x}\ind_{\{x\leq 0\}}+\e^{\frac{1-p}{2\eps}x}\ind_{\{x>0\}}\big)\dd x.
    \ee 
    The children are then positioned at locations $y+\Pi_y$, and only children with a location smaller than $0$ are retained. 
    \lyu{The above point of view was taken by Dereich and M\"orters~\cite{DerMor13} in their analysis of preferential attachment models without vertex deletion. In their setting, the Poisson process $\Pi_y$ is independent of the position of $y < 0$.
    Despite this simplification, while they manage to provide precise qualitative information for the critical parameters of the model, the survival probability $\gamma$ was only analysed via Monte Carlo simulations even in this simpler setting.}
\end{remark}}

Theorem~\ref{thm components intro} raises the natural question of how $\beta_c$ behaves as a function of $p$. The following proposition partially answers this question by providing bounds for $\beta_c$ in terms of $p$.

\begin{proposition}\label{prop beta_c intro}
In the setup of Theorem~\ref{thm components intro}, the threshold function $\beta_c=\beta_c(p)$ is a non-increasing continuous function over the interval $(1/2, 1]$. Moreover, for every $p\in (1/2, 1]$,
$$\max\left\{\sqrt{\tfrac{1-p}{p}}, \tfrac{1}{4}\right\} \le \beta_c(p)\le \inf_{t\in (-1/2, \infty)}\left(\tfrac{(1+2t)(2t^2+7t+4+1/p)}{(1+t)^2(t+1/p)(2t+2/p-1)}\right)^{-1/2}\le \sqrt{\tfrac{2-p}{p(1+4p)}}.$$
\end{proposition}

\begin{remark}
While the tighter upper bound for $\beta_c(p)$, as provided by the infimum in Proposition~\ref{prop beta_c intro} is strictly smaller for all $p\in(1/2,1]$, we still provide a weaker but also simpler upper bound in terms of $p$. This bound is realised by taking $t=0$ in the infimum. Note that $\beta_c(1)=1/4$ and $\lim_{p\downarrow 1/2}\beta_c(p)=1$, as can be observed in Figure~\ref{fig:betac}.
\end{remark}

\begin{figure}[ht]
\centering
\includegraphics[width=0.8\textwidth]{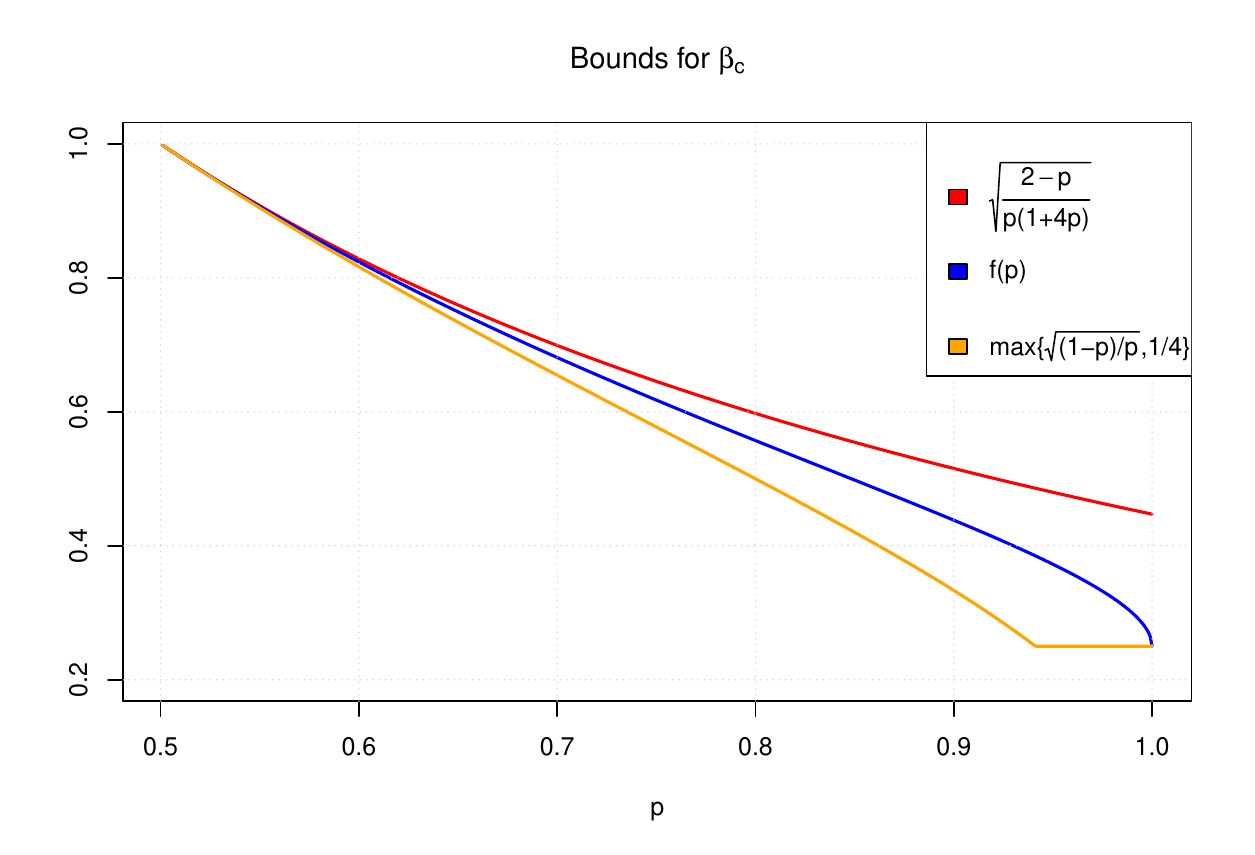}
\caption{A plot of the lower bound and the two upper bounds for $\beta_c$ from Proposition~\ref{prop beta_c intro}. The sharper upper bound given by
	$f(p) := \inf_{t\in(-1/2,\infty)} \Big(\frac{(1+2t)(2t^2+7t+4+1/p)}{(1+t)^2(t+1/p)(2t+2/p-1)}\Big)^{-1/2}$ is approximated numerically.}
\label{fig:betac}
\end{figure}


\lyu{Our next result exhibits a stability property of the giant component in the supercritical regime. 
More precisely, we show that, when $\beta > \beta_c$, given a sufficiently small constant $\lambda = \lambda(\beta) > 0$, removing the set $V_{\lambda n}$ containing the vertices born before time $\lambda n$ from $G_n$ modifies the size of the giant component only slightly.
We further prove some properties of the second largest component and the metric structure of the giant in the perturbed graph.}

\begin{theorem}\label{thm approximate intro}
Consider the setup of Theorem~\ref{thm components intro} with $\beta>\beta_c(p)$ and recall the constant $\gamma = \gamma(p)$ from~\eqref{eq:survprob}. Then, for every $\delta \in (0,\gamma)$ there exists $\lambda > 0$ such that the graph $G^\lambda_n$, induced from $G_n$ by the vertices in $V_n\setminus V_{\lambda n}$, satisfies w.h.p.\ each of the following statements: 
\begin{enumerate}
	\item $(2\eps \gamma - \delta) n\le |\cC_1(G^\lambda_n)|\le (2\eps \gamma+\delta)n$,
	\item $|\cC_2(G^\lambda_n)| = \cO(\log n)$,
	\item there is a constant $\zeta = \zeta(\beta, p, \lyu{\lambda})>0$ such that two vertices of $G^\lambda_n$ chosen uniformly at random are at graph distance in the interval $[(1-\delta)\zeta \log n, (1+\delta) \zeta \log n]$ (in $G^\lambda_n$).
\end{enumerate}
\end{theorem} 

The graph $G_n$ being constructed recursively, its edges may naturally be equipped with an orientation from the vertex arriving at a later step to the vertex arriving earlier. 
Therefore, by setting $V_n = \{v_1, \ldots, v_{|V_n|}\}$ with $v_i<v_j$ if $i<j$, for every $j\in [|V_n|]$, we let $d^-_n(v_j)$ (respectively $d^+_n(v_j)$) denote the number of neighbors of $v_j$ in the set $\{v_1, \ldots, v_{j-1}\}$ (respectively in the set $\{v_{j+1}, \ldots, v_{|V_n|}\}$). 
Also, we let $d^s_n(v_j) := d^+_n(v_j)+d^-_n(v_j)$ denote the total degree of $v_j$ in $G_n$. 
We remark that, in general, we identify the vertices in the process in two different ways: as elements of $V_n = \{v_1, \ldots, v_{|V_n|}\}$, and by their labels which indicate the step at which they were born. The latter notation is more general as it is applied to all vertices in the process and not only the ones that survive until step $n$.

The following theorem precisely characterises the maximum degree, the maximum in-degree and the maximum out-degree, as well as the the labels of vertices that attain these maximal degrees.

\begin{theorem}\label{thrm:max intro}
Consider the DRGVR model as in Definition~\ref{def:drgvr} with $\beta>0$ and $\eps \in(0,1/2)$.
Then, for $\square\in\{s,+\}$,
\begin{equation}\label{eq:2ndords}
	\frac{(\log\log n)^2}{\log n}\max_{v\in V_n} d_n^\square(v) -\log\log n-\log\log\log n \toinp 1+\log \Big(\frac{\beta p}{1-p}\Big),
\end{equation}  
and 
\begin{equation}\label{eq:2ndordmin}
	\frac{(\log\log n)^2}{\log n}\max_{v\in V_n} d_n^-(v) -\log\log n-\log\log\log n\toinp 1+\log\beta.
\end{equation}
Finally, let $\cI^{\square}_n:=\{u\in V_n \text{ and }d_n^\square(u)=\max_{v\in V_n}d_n^\square(v)\}$ be the set of vertices that attain the maximum (in-/out-)degree. Then, for $\square\in\{s,+\}$ and for any fixed $\delta\in(0,1)$, 
\begin{equation}\label{eq:loc}
	\frac{\log \min\cI^\square_n}{\log n}\toinp 1\quad \text{and}\quad \P{\max\cI^\square_n\leq \delta n}=1-o(1),\quad \text{and} \quad \frac{\min\cI^-_n}{n}\toinp 1.
\end{equation} 
\end{theorem}

Note that Theorem~\ref{thrm:max intro} does not include the case $p=1$ (or, equivalently, $\eps=1/2$) in which no vertices are removed. Indeed, as can be observed from the results by Lodewijks~\cite{Lod21,Lod22} and Banerjee and Bhamidi~\cite{BanBha21} (in the case of deterministic out-degree equal to $1$), the behaviour of large-degree vertices and their labels in the case $p=1$ is rather different. Most notably, the vertices with maximal (in-)degree have labels of the order $n^{(1-1/(2\log 2))+o(1)}$ instead of $n^{1-o(1)}$.

The last result in this section concerns a large family of Lipschitz-type \lyu{network statistics}. Fix a function $f$ from the set of finite directed graphs to the real numbers. We say that $f$ is \emph{$L$-Lipschitz} if, for any two directed graphs $H_1$ and $H_2$ that differ in only one edge (that is, $|E(H_1)\setminus E(H_2)|+|E(H_2)\setminus E(H_1)|\le 1$), we have $|f(H_1) - f(H_2)|\le L$. We remark that the definition does not make any reference to the vertex sets of $H_1$ and $H_2$; in particular, the family of functions we are interested in is insensitive to isolated vertices. 

\begin{theorem}\label{thm concentration intro}
Fix $\beta>0$ and $\eps\in(0,1/2]$, and consider the DRVGR model as in Definition~\ref{def:drgvr}. 
Fix an integer $L\in \mathbb N$ and an $L$-Lipschitz function $f$ defined on the set of directed graphs. 
Then, for every $t\ge 0$,
\begin{align*}
	(i)\,&\P{\Big|f(\vv{G}_n) - \mathbb E\left[f(\vv{G}_n)\mid |V_n|, (d_n^+(v_i))_{i=1}^{|V_n|}\right]\Big|\ge t\;\Big|\; |V_n|, (d_n^+(v_i))_{i=1}^{|V_n|}}\le 2\exp\left(-\tfrac{t^2}{8 |E_n| L^2}\right)\hspace{-2.7pt},\\
	(ii)\,&\P{\Big|f(\vv{G}_n) - \mathbb E\left[f(\vv{G}_n)\mid (V_i)_{i=1}^n, |E_n|\right]\Big|\ge t\;\Big|\; (V_i)_{i=1}^n, |E_n|}\le 2\exp\left(-\tfrac{t^2}{8 |E_n| L^2}\right).
\end{align*}
\end{theorem}

Note that, despite the fact that both statements show conditional concentration of $f(\vv{G}_n)$, the two are quite different in nature. 
In the first case, the conditioning is done \lyu{over $|V_n|$ and} the out-degrees of all vertices in $\vv{G}_n$ \lyu{but the vertex process $(V_i)_{i=1}^{n-1}$ is irrelevant.}
In the second case, the conditioning is done over the entire vertex process $(V_i)_{i=1}^n$ but the only structural information about $\vv{G}_n$ is its number of edges.

\paragraph{Main ideas of the proofs.} The proofs of the main results use a variety of techniques. First, Theorem~\ref{thrm:LWC intro} is obtained by providing a coupling between the Breadth-First Search (BFS) exploration of the neighbourhood of a vertex $k_0$ selected uniformly at random in $G_n$, and the recursive construction of the multi-type branching process $\cT$ with root $0$, as defined in Definition~\ref{def:wll}. 
This coupling is mainly inspired by techniques used for proving the local limit of affine preferential attachment models (see~\cite{BerBorChaSab14,Lo21}).
\lyu{Our main technical contribution is to incorporate the vertex deletion mechanism in the definition of the DRGVR by using an iterative two-step coupling scheme from~\cite{BerBorChaSab14,Lo21}.
The approach goes roughly as follows: given that the $r$-th neighbourhoods of $k_0$ in $G_n$ and $0$ in $\cT$ have been coupled, we consider two corresponding vertices at distance $r$ from the roots of $G_n$ and $\cT$, respectively, and couple the rescaled labels and positions of their children, respectively, with an intermediary discretised Poisson point process.}
\lyu{We note that an alternative approach towards a proof of a local limit goes through a sandwiching lemma (Lemma~\ref{lemma sandwich}) and by applying results from~\cite{BolJanRio05}, but remark that this has several disadvantages compared to our approach: 
\begin{itemize}
    \item Our results provide quantitative bounds on the total variation distance;
    \item We obtain a more general \emph{marked} local convergence;
    \item Our construction is explicit and the method provides a systematic approach for proving quantitative local limit theorems where the limiting object is naturally seen within an ambient geometric space (here $[0,1]$);
    \item We believe that an approach similar to ours can be applied to preferential attachment models with vertex removal where the results from~\cite{BolJanRio05}
    may not apply.
\end{itemize}}
The proof of Theorem~\ref{thm components intro} combines a comparison of the DRGVR model $G_n$ with a particular inhomogeneous Erd\H{o}s-R\'enyi graph model on the vertex set $V_n$ (of $G_n$) and results by Bollob\'as, Janson and Riordan~\cite{BolJanRio05}. 
More precisely, we use necessary and sufficient conditions for the existence of a giant component of a wide family of inhomogeneous random graph models from~\cite{BolJanRio05}, and a \lyu{stochastic} comparison between the DRGVR model and an instance of this family allows us to transfer their results to our setting. Proposition~\ref{prop beta_c intro} is a byproduct of this proof and is shown by analysing a related linear operator. 
Theorem~\ref{thm approximate intro} combines the fact that the restriction of the inhomogeneous Erd\H{o}s-R\'enyi graph to $V_n\setminus V_{\lambda n}$ satisfies that every two edges appear in $G^\lambda_n$ with probabilities which are a constant factor away from each other (this is not the case for $G_n$ itself) with results from~\cite{BolJanRio05}.

Theorem~\ref{thrm:max intro} combines the comparison between the DRGVR model and an inhomogeneous Erd{\H o}s-R\'enyi graph model with precise bounds on the tail distribution of the vertex-degrees for the latter model. These bounds are used to show that, when the expected number of vertices with degree at least $k_n$ either tends to zero (respectively to infinity), then the largest degree is at most (respectively at least) $k_n$ with high probability. With the growth rate of the maximum degree at hand, the tail distribution bounds can be used to argue that vertices with ``too small'' or ``too large'' label have degrees that are substantially smaller than the maximum degree in $G_n$. 

Finally, the proof of Theorem~\ref{thm concentration intro} is based on constructing martingales with bounded differences and the use of the classical Azuma inequality.

\paragraph{Organisation of the paper.}  
We provide some preliminary results in Section~\ref{sec:prelim} that are used in the proofs of the results presented in Section~\ref{sec:intro}. Sections~\ref{sec:bfs}, \ref{sec:lwl} and~\ref{sec:lwlcont} are dedicated to proving our main result (Theorem~\ref{thrm:LWC intro}) regarding the local convergence \lyu{in distribution} of the DRGVR model. In Section~\ref{sec inhom E-R}, we compare our model with a particular instance of the inhomogeneous Erd{\H o}s-R\'enyi graph, and use established results for inhomogeneous random graphs to prove Theorem~\ref{thm components intro}, Proposition~\ref{prop beta_c intro} and Theorem~\ref{thm approximate intro}. Section~\ref{sec:max} is then devoted to proving the size of the maximum degree and the labels of the vertices that attain the maximum degree (Theorem~\ref{thrm:max intro}). Finally, in Section~\ref{sec:lipschitz}, we prove Theorem~\ref{thm concentration intro}.

\section{Preliminaries}\label{sec:prelim}

\lyu{In this section, we introduce some preliminary lemmas. The reader eager to dive into the proofs of the main results is invited to consult the part of Section~\ref{sec:2.2} before Lemma~\ref{lemma:condconc} and skip the remaining preliminaries on first read; these may be safely consulted occasionally later on.}

\subsection{General probabilistic preliminaries}

Let $\Gamma$ be a probability distribution on $(0,1]\times \{0,1\}$ defined as
\begin{equation}\label{eq:gamma}
\Gamma(x,y):=px^{(1-p)/(2\eps)}\,\mathrm{d} x\, \delta_1( y)+\big(1-px^{(1-p)/(2\eps)}\big)\,\mathrm{d}x\,\delta_0(y),
\end{equation} 
where $\delta_0$ and $\delta_1$ are Dirac measures \lyu{at 0 and 1, respectively}.
As we shall see, the distribution $\Gamma$ is tightly connected to the distribution of the birth times of the vertices in $V_n$. For now, we state and prove a key property of this distribution.

\begin{lemma}\label{lemma:beta}
Let $(X,Y)$ be sampled from $\Gamma$. Then, for every $x\in(0,1]$,
\begin{equation}
	\P{X\leq x\,|\, Y=1}=x^{p/(2\eps)}.
\end{equation} 
In particular, conditionally on $Y=1$, $X \sim \mathrm{Beta}(\tfrac{p}{2\eps},1)$. Equivalently, let $X\sim \mathrm{Unif}(0,1)$ and, conditionally on $X$, $Y\sim \mathrm{Ber}(pX^{(1-p)/(2\eps)})$. Then, conditionally on $Y=1$, $X \sim \mathrm{Beta}(\tfrac{p}{2\eps},1)$.
\end{lemma}

\begin{proof}
Recall that $p = \tfrac{1}{2}+\eps$. Thus, both results readily follow from the fact that
\begin{equation}
	\P{X\leq x\,|\, Y=1}=\frac{\int_0^x ps^{(1-p)/(2\eps)}\,\mathrm ds}{\int_0^1 ps^{(1-p)/(2\eps)}\,\mathrm ds}=x^{p/(2\eps)},
\end{equation} 
as desired.
\end{proof}

The next lemma provides a useful coupling of the Bernoulli and the Poisson distributions.

\begin{lemma}[\cite{Lin02}, page 5, (1.11)]\label{lemma:poibercoupling}
Fix $\lambda \in (0,1)$, and let $X\sim \lyu{\mathrm{Po}}(\lambda)$ and $Y$ be a Bernoulli random variable with success probability $\lambda$. There exists a coupling $(\widehat X,\widehat Y)$ of $(X,Y)$ such that \lyu{$\widehat Y\ge \ind_{\{\widehat X\le 1\}} \widehat X$} almost surely and
\begin{equation}
	\mathbb P(\widehat X\neq \widehat Y)\leq \lyu{\lambda^2}.
\end{equation} 
\end{lemma}

\begin{lemma}\label{lemma:coupling poisson}
Fix $\lambda_1, \lambda_2 > 0$. There exists a coupling $(X_1, X_2)$ of the Poisson distributions with means $\lambda_1$ and $\lambda_2$ so that
\begin{equation}
	\mathbb P(X_1\neq X_2)\leq |\lambda_1 - \lambda_2|.
\end{equation} 
\end{lemma}
\begin{proof}
Assume without loss of generality that $\lambda_1 < \lambda_2$. Let $X_1\sim \mathrm{Po}(\lambda_1)$ and let $Y\sim \mathrm{Po}(\lambda_2-\lambda_1)$ be a Poisson variable independent from $X_1$. Set $X_2 = X_1 + Y$. Then, $X_2\sim \mathrm{Po}(\lambda_2)$ and 
\begin{equation*}
	\P{X_1\neq X_2} = \P{Y\ge \lyu{1}} = 1 - \e^{-(\lambda_2 - \lambda_1)}\le \lambda_2 - \lambda_1,
\end{equation*}
where the last inequality uses that, for all $x\in \mathbb R$, $1 - x\le \e^{-x}$.
\end{proof}

Finally, we state versions of the well-known \emph{Chernoff's inequality} and \emph{Azuma's inequality}. Recall that a (finite or infinite) stochastic process $(X_n)_{n\ge 0}$ is called a \emph{martingale} if for every $n\ge 0$, $\mathbb E[|X_n|] < \infty$ and $\mathbb E[X_{n+1}\mid (X_i)_{i=0}^n] = X_n$.

\begin{lemma}\label{lemma:chern}
\begin{enumerate}
	\item[(i)]\label{lem chern} \emph{(Theorem~2.1 and Corollary~2.3 in~\cite{JLR00})} Define the mapping $\phi: x\in (-1, \infty)\mapsto (1+x)\log(1+x)-x$. Given a binomial random variable $X$,
\end{enumerate}
\vspace{-1.5em}
\begin{align*}
	\mathbb P(X - \mathbb E[X] \ge t)&\le \exp{}\left( - \mathbb E[X]\cdot \phi\Big(\tfrac{t}{\mathbb E[X]}\Big) \right) \le \exp{}\left( - \tfrac {t^2}{2 (\mathbb E[X] + t/3)} \right) \text{ for all } t \ge 0,\\
	\mathbb P(X - \mathbb E[X] \le -t)&\le \exp{}\left( - \mathbb E[X]\cdot \phi\Big(\tfrac{-t}{\mathbb E[X]}\Big) \right) \le \exp{}\left( - \tfrac {t^2}{2 (\mathbb E[X] + t/3)} \right) \text{ for all } t\in [0,\mathbb E[X]).
\end{align*}
\begin{enumerate}
	\item[(ii)]\label{lem azuma}\emph{(Theorem~2.25 in~\cite{JLR00})} Fix a sequence of positive real numbers $(c_i)_{i=1}^n$ and a random variable $X$. Suppose that $(X_n)_{n\ge 0}$ is a martingale such that $X_n = X$, $X_0 = \mathbb E[X]$ and for every $i\in [n]$, $|X_{i-1}-X_i|\le c_i$. Then, for all $t \ge 0$,
	\begin{equation*}
		\mathbb P(|X - \mathbb E[X]| \ge t) \le 2\exp{}\left( - \frac {t^2}{2 \sum_{i=1}^n c_i^2} \right).
	\end{equation*}
\end{enumerate}
\end{lemma}

\subsection{Preliminaries for the DRGVR}\label{sec:2.2}

\paragraph{An alternative viewpoint on the vertex process $(V_n)_{n\ge 0}$.} In this section, we introduce a formal framework that help us to keep track of the vertex process $(V_n)_{n\ge 0}$. To begin with, we define recursively a set of \emph{marks} $(\nu_{i,n})_{1\le i\le n}$ as follows:

\begin{itemize}
\item if $n = 1$, set $\nu_{1,1} = 1$ if $\xi_1 = 1$ and $\nu_{1,1} = 0$ otherwise,
\item if $n\ge 2$ and $\xi_n = 1$, set $\nu_{i,n} = \nu_{i,n-1}$ for all $i\in [n-1]$ and $\nu_{n,n}=1$,
\item if $n\ge 2$ and $\xi_n = 0$, select a uniformly random element $i_n$ from $V_{n-1}$ and set $\nu_{i_n,n} = \nu_{n,n} = 0$ and $\nu_{i,n} = \nu_{i,n-1}$ for all $i\in [n-1]\setminus \{i_n\}$.
\end{itemize}

In particular, by definition $V_n = \{i\in [n], \nu_{i,n} = 1\}$, so the set of all marks contains the entire information for the vertex process $(V_n)_{n\ge 1}$ (and is, in a sense, equivalent to it). From this point, we say that the vertices in $V_n$ are the ones that are \emph{alive after $n$ steps} (of the vertex process), or equivalently \emph{survive after $n$ steps}. One advantage of the marks is that they allow to describe the distribution of the alive vertices in a clear way. More precisely, as we shall see in Lemma~\ref{lemma:rootcouple}, the empirical distribution
\begin{equation}
\Gamma^{(n)}:=\frac{1}{n} \sum_{i=1}^n \delta_{(i/n,\nu_{i,n})}
\end{equation} 
over the set $(0,1]\times \{0,1\}$, where $\delta_{(x,y)}$ denotes a Dirac mass at the point $(x,y)\in (0,1]\times \{0,1\}$, converges in distribution to the distribution $\Gamma$ defined in~\eqref{eq:gamma}.

\paragraph{Further preliminary results.} We define the event
\begin{equation}\label{eq:An}
\cQ_n := \{||V_n| - 2\eps n| \le n^{2/3}\}.
\end{equation} 
\begin{lemma}\label{lemma:condconc}
Fix $j\lyu{=j(n)}\in[n]$ and $S\subset [n]$ such that $j=\omega(1)$ and $|S|=o(j^{\lyu{1/3}})$. Then, there exists a constant $C = C(\eps)>0$ such that, \lyu{for all sufficiently large $n$,}
\begin{equation}
	\mathbb P(\cQ_j^c\,\mid\, \forall i\in S: \nu_{i,i}=1)\leq \e^{-Cj^{1/3}}.
\end{equation} 
\end{lemma}

\begin{proof}
The proof consists of two parts. \lyu{To estimate the lower tail of $|V_n|$}, note that the process $(|V_i|)_{i\ge 0}$ dominates a random walk $(S_i)_{i\ge 0}$ on $\mathbb Z$ that makes a step $+1$ with probability $p$ and $-1$ with probability $1-p$. 
Moreover, the variable $\widehat S_i := (S_i+i)/2$ has Binomial distribution with parameters $i$ and $p$. Using Chernoff's bound for $\widehat S_j$ shows that
\begin{align*}
	& \P{|V_j|-2\eps j \le \lyu{-}j^{2/3}\,\Big|\, \forall i\in S: \nu_{i,i}=1}
	\leq\; 
	\P{|V_j|-2\eps j \le \lyu{-}j^{2/3}}\\ 
	& \leq \P{S_j-2\eps j \le \lyu{-}j^{2/3}}\leq \P{\widehat S_j-(1/2+\eps) j \le \lyu{-}j^{2/3}/2} = \e^{-\Omega(j^{1/3})}.    
\end{align*}

For the upper \lyu{tail of $|V_n|$}, for all $i\ge 0$, we define $Z_i$ as the number of steps $r\in [i]$ at which \lyu{$S_r < \min_{t\in \{0,\ldots,r-1\}} S_t$}, that is, the last step of the random walk is \lyu{smaller than the minimum over all values up to that step}. 
We now show that, conditionally on the event $\{\forall i\in S: \nu_{i,i}=1\}$, $|V_j|$ is dominated by the process $\widetilde M_j := |S|+S_{j-|S|}+Z_{j-|S|}$. 
On the one hand, $(|V_i|)_{i\ge 0}$ and $(S_i+Z_i)_{i\ge 0}$ have the same distribution (which is the one of a discrete random walk \lyu{with steps $\{1,-1\}$, a fixed positive drift and} conditioned to stay non-negative). 
On the other hand, $|V_j|$ conditioned on births \lyu{at steps $j-|S|+1,\ldots,j$} stochastically dominates $|V_j|$ conditioned on $\{\forall i\in S: \nu_{i,i}=1\}$ (i.e.\ births at steps $i\in S$) for any other choice of $S\subseteq [j]$. Indeed, by introducing $|S|$ vertices \lyu{at steps $j-|S|+1,\ldots,j$,} we ensure that none of these vertices is removed until step $j$.

At the same time, the distribution of $\max_{i\ge 0} Z_i$ is dominated by a geometric random variable since at every step $i\ge 0$, the event $\{\forall t > i, S_t - S_i > 0\}$ has strictly positive probability depending only on $\eps$, see e.g. Example 21.2 in~\cite{LP17}. Therefore, together with Chernoff's inequality, we get
\begin{align*}
	& \P{|V_j|-2\eps j \ge j^{2/3}\mid \forall i\in S, \nu_{i,i}=1}\\
	\le\; 
	& \P{S_{j-|S|}+Z_{j-|S|} + |S| -2\eps j \ge j^{2/3}}\\
	\le\;
	& \P{S_{j-|S|}-2\eps (j-|S|) \ge j^{2/3}/3} + \P{Z_{j-|S|} \ge j^{2/3}/3}\\ 
	\le\; 
	& \e^{-\Omega(j^{1/3})} + \e^{-\Omega(j^{2/3})} = \e^{-\Omega(j^{1/3})},
\end{align*}
which finishes the proof.
\end{proof}

For $i\in [n]$ and a set $S\subseteq [\lyu{n}]$, define the event $\cW_i(S) = \{S\lyu{\cap [i]}\subseteq V_i\}$, that is, at each of the steps in $\lyu{S\cap [i]}$, a vertex was born, and each of these vertices survives until step $i$. Similarly to Lemma~\ref{lemma:condconc}, the following lemma allows us to control the number of vertices alive after $n$ steps conditionally on the introduction of vertices at certain steps.

\lyu{\begin{lemma}\label{lemma:condsurv}
	Fix $j_0 = j_0(n) = \omega(1)$, $j\in [j_0,n]$, a set $S\subseteq [j_0,n]$ and a non-empty set $R\subseteq [j_0,j]$ such that $|S\cup R|^{6} = o(j_0)$. Then,
	\begin{equation}\label{eq:2.6}
		\P{\cW_n(R) \mid \cW_n(S)\cap \cW_{j}(R)} = \left(1+\cO\left(\frac{|R|}{j^{1/6}}\right)\right)\left(\frac{j}{n} \right)^{|R|(1-p)/(2\eps)},
	\end{equation}
	where the constant in the $\cO$-term is independent of the parameters.
	\end{lemma}}
	
	In particular, the lemma implies that, if the set $S$ of vertices conditioned to survive after $n$ steps is not ``too large'', the probability that a vertex born at step $j$ survives after step $n$ remains the same as in the unconditional setting up to lower order terms.
	
	\lyu{\begin{proof}[Proof of Lemma~\ref{lemma:condsurv}]
	The event is trivial when $p=1$ or $j=n$, so we assume that $p < 1$ and $j\le n-1$.
	We also assume that $j_0 = \min(S\cup R)$: indeed, given a fixed $j\ge \min(S\cup R)$, proving the inequality for the largest possible value of $j_0$ implies it for smaller values satisfying the assumptions in the statement.
	By the definition of conditional probability and the fact that $\cW_n(R)\subseteq \cW_j(R)$,  
	we have
	\begin{equation}\label{eq:probeq}
		\P{\cW_n(R)\mid \cW_n(S)\cap \cW_j(R)} = \frac{\P{\cW_n(S)\cap \cW_n(R)}}{\P{\cW_n(S)\cap \cW_j(R)}}.
	\end{equation} 
	We first show that for every $i\in [j_0,n]$, $\P{\cW_i(S\cup R)}$ is equal to
	\begin{equation}\label{eq:induction}
		p^{|(S\cup R)\cap [i]|} \prod_{t=j_0+1}^{i} \left(1 - \ind_{t\notin S\cup R} (1+\cO(t^{-1/3})) \frac{(1-p)|(S\cup R)\cap [t-1]|}{2\eps t} \right),    
	\end{equation}
	where all constants in the $\cO$-terms here and below are all uniformly bounded in absolute value.
	We note that, for every $i\ge j_0$, $p^{|(S\cup R)\cap [i]|} \ge 2^{-|S\cup R|}$ and the product in~\eqref{eq:induction} is bounded from below by
	\[\exp\bigg(-2(1-p)\sum_{t=j_0}^i \frac{|S\cup R|}{2\eps t}\bigg) = \exp(-\cO(j_0^{1/6} \log i)).\]
	In particular, combining the fact that $|S\cup R|^6 = o(j_0)$ with Lemma~\ref{lemma:condconc} shows that~\eqref{eq:induction} is larger than $i^{4/3} \mathbb P(\cQ_i^c)$ for all $i\in [j_0,n]$, with $\cQ_i$ introduced in~\eqref{eq:An}.
	
	We turn to showing~\eqref{eq:induction} by induction. The base case $i=j_0$ is trivially satisfied. Fix an integer $i\in [j_0,n-1]$ and suppose that the induction hypothesis is satisfied for $i$. 
	If $i+1\in S\cup R$, then $\P{\cW_{i+1}(S\cup R)} = p\, \mathbb P(\cW_i(S\cup R))$ and the induction hypothesis for $i+1$ is satisfied.
	Suppose that $i+1\notin S\cup R$ and write $\P{\cW_{i+1}(S\cup R)}$ as
	\begin{align*} 
		&\P{\cW_{i+1}(S\cup R)\mid \cQ_i}\mathbb P(\cQ_i) + \P{\cW_{i+1}(S\cup R)\mid \cQ_i^c}\mathbb P(\cQ_i^c)\\
		=\; 
		&\P{\cW_{i+1}(S\cup R)\mid \cW_i(S\cup R)\cap \cQ_i}\P{\cW_i(S\cup R)\cap \cQ_i} + \cO(\mathbb P(\cQ_i^c))\\
		=\; 
		&\P{\cW_{i+1}(S\cup R)\mid \cW_i(S\cup R)\cap \cQ_i}(\P{\cW_i(S\cup R)} - \cO(\P{\cQ_i^c})) + \cO(\mathbb P(\cQ_i^c)).
	\end{align*}
	Moreover, since $\cW_i(S\cup R)\cap \cQ_i$ is independent of $\xi_{i+1}$ (introduced in Definition~\ref{def:drgvr}),
	\begin{align*}
		\mathbb P(\cW_{i+1}(S\cup R)\mid \cW_i(S\cup R)\cap \cQ_i) 
		&= 1-\frac{(1-p)|(S\cup R)\cap [i]|}{2\eps i + \cO(i^{2/3})}\\ 
		&= 1-(1+\cO(i^{-1/3}))\frac{(1-p)|(S\cup R)\cap [i]|}{2\eps i},
	\end{align*}
	which, together with the fact that $i^{4/3}\mathbb P(\cQ_i^c) = o(\mathbb P(\cW_i(S\cup R)))$, yields 
	\[\P{\cW_{i+1}(S\cup R)} = \bigg(1-(1+\cO(i^{-1/3}))\frac{(1-p)|(S\cup R)\cap [i]|}{2\eps i}\bigg)\P{\cW_i(S\cup R)}\]
	and finishes the induction step.
	Since $R\subseteq [j_0,j]$, the induction hypothesis for $i=n$ implies that $\P{\cW_n(S\cup R)}$ is equal to
	\begin{equation*}
		\mathbb P(\cW_{j}(S\cup R)) p^{|S\cap [j+1,n]|} \prod_{t=j+1}^{n} \left(1 - \ind_{t\notin S} (1+\cO(t^{-1/3})) \frac{(1-p)|(S\cup R)\cap [t-1]|}{2\eps t} \right).   
	\end{equation*}
	A similar proof by induction yields that, for every $i\in [j+1,n]$, $\P{\cW_i(S)\cap \cW_{j}(R))}$ is equal to
	\begin{equation*}
		\P{\cW_{j}(S)\cap \cW_{j}(R)} p^{|S\cap [j+1,i]|} \prod_{t=j+1}^{i} \left(1 - \ind_{t\notin S} (1+\cO(t^{-1/3})) \frac{(1-p)|S\cap [t-1]|}{2\eps t} \right).
	\end{equation*}
	Then, using that $\cW_{j}(S\cup R) = \cW_{j}(S)\cap \cW_{j}(R)$ together with the approximation $1-x = \exp(-x+o(x^{4/3}))$ when $x\to 0$, we conclude that~\eqref{eq:probeq} is equal to
	\begin{align*}
		\exp\bigg(-\sum_{t=j+1}^n \ind_{t\notin S} \bigg((1+\cO(t^{-1/3})) \frac{(1-p)|S\cup R|}{2\eps t} - (1+\cO(t^{-1/3})) \frac{(1-p)|S|}{2\eps t}\bigg)\bigg).
	\end{align*}
	However, since $t\ge j_0$ and $t^{-1/3} |S\cup R|\le t^{-1/6}$, we get that, for every $t\in [j+1,n]$,
	\[(1+\cO(t^{-1/3})) \frac{(1-p)|S\cup R|}{2\eps t} - (1+\cO(t^{-1/3})) \frac{(1-p)|S|}{2\eps t} = (1+\cO(t^{-1/6})) \frac{(1-p)|R|}{2\eps t}.\]
	Moreover, the terms corresponding to $t\in [j+1,n]\cap S$ that were excluded from the summation contribute at most $\cO(|S\cup R|/j) = \cO(j^{-1/2})$ to the total sum. 
	Hence, summing the above expression over $t\in [j+1,n]$ gives us that
	\begin{align*}
		\P{\cW_n(R)\mid \cW_n(S)\cap \cW_j(R)} 
		&=\, \exp\bigg(-\sum_{t=j+1}^n \ind_{t\notin S} (1+\cO(t^{-1/6})) \frac{(1-p)|R|}{2\eps t}\bigg)\\
		&=\,  \exp\bigg(-\sum_{t=j+1}^n (1+\cO(t^{-1/6})) \frac{(1-p)|R|}{2\eps t}+\cO(j^{-1/2})\bigg),
	\end{align*} 
	and the statement follows from the fact that
	\[\sum_{t=j+1}^n \frac{1}{t} = \log n - \log j + \cO(j^{-1}) \text{ and } \sum_{t=j+1}^n t^{-7/6} = \cO(j^{-1/6}),\]
	which concludes the proof.
	\end{proof}}
	
	\lyu{The previous lemma has the following useful corollary.
\begin{corollary}\label{cor:condsurv}
	Fix $j_0 = j_0(n) = \omega(1)$ and a set $U = \{j_1,\ldots,j_r\}\subseteq [j_0,n]$. For every vector $\mathbf{x}=(x_1,\ldots, x_r)\in \{0,1\}^r$, \bas{define $I_{\mathbf x}:=\{i\in [r]: x_i=1\}$ } and set
	\[\eta(\mathbf{x},U) := \prod_{i\in I_{\mathbf{x}}}  p\bigg(\frac{j_i}{n}\bigg)^{(1-p)/(2\eps)} \prod_{i\in [r]\setminus I_{\mathbf{x}}} \bigg(1-p\bigg(\frac{j_i}{n}\bigg)^{(1-p)/(2\eps)}\bigg).\]
	Then, for every $\mathbf{x}\in \{0,1\}^r$,
	\[\bigg|\mathbb P((\nu_{j,n})_{j\in U} = \mathbf{x}) - \eta(\mathbf{x},U)\bigg| =  \cO\bigg(\frac{2^r r}{j_0^{1/6}} \prod_{i\in I_{\mathbf{x}}} p\bigg(\frac{j_i}{n}\bigg)^{(1-p)/(2\eps)}\bigg),\]
	where the constant in the $\cO$-term is independent of $r$ and $\mathbf{x}\in \{0,1\}^r$.
\end{corollary}
}

\begin{remark}\label{rem:condsurv}
	Corollary~\ref{cor:condsurv} implies that 
	\be 
	\dd_{\mathrm{TV}}(\cL((\nu_{j,n})_{j\in U}),\eta(\cdot,U))=\cO\Big(\frac{4^rr}{j_0^{1/6}}\Big).
	\ee 
	We note that, in particular, this upper bound depends only on $j_0$ and $r=|U|$, not on the exact values of the indices $j_1,\ldots, j_r$.
\end{remark}

\begin{proof}[Proof of Corollary~\ref{cor:condsurv}]
	For all $i\in \{0,1,\ldots,r\}$, denote $U_i = \{j_\ell: \ell\in [1,i]\}$.
	First, by applying Lemma~\ref{lemma:condsurv} consecutively $r$ times (with $S = U_{i-1}$, $R = \{j_i\}$ and $j = j_i$ for all $i\in [r]$), we deduce that
	\begin{align*}
		\mathbb P(\cW_n(U)) 
		&= \prod_{i=1}^r p\,\mathbb P(\cW_n(U_i)\mid \cW_n(U_{i-1})\cap \cW_{j_i}(\{j_i\}))\\
		&= \prod_{i=1}^r \bigg(1 + \cO\bigg(\frac{1}{j_i^{1/6}}\bigg)\bigg) p\bigg(\frac{j_i}{n}\bigg)^{(1-p)/(2\eps)} = \bigg(1 + \cO\bigg(\frac{r}{j_0^{1/6}}\bigg)\bigg) \prod_{i=1}^r p\bigg(\frac{j_i}{n}\bigg)^{(1-p)/(2\eps)}\hspace{-0.2em}.
	\end{align*}
	More generally, the same argument shows that, for all subsets $I\subseteq [r]$,
	\[\mathbb P(\cW_n(\{\nu_{j_i}: i\in I\})) = \bigg(1 + \cO\bigg(\frac{r}{j_0^{1/6}}\bigg)\bigg) \prod_{i\in I} p\bigg(\frac{j_i}{n}\bigg)^{(1-p)/(2\eps)}.\]
	Then, for every vector $\mathbf{x}\in \{0,1\}^r$ with set of 1-bits $I_{\mathbf{x}}$, the inclusion-exclusion principle shows that
	\begin{align*}
		\mathbb P((\nu_{j_i})_{i=1}^r = \mathbf{x}) = \sum_{I_{\mathbf{x}}\subseteq I\subseteq [r]} \bigg(1 + \cO\bigg(\frac{r}{j_0^{1/6}}\bigg)\bigg) (-1)^{|I| - |I_{\mathbf{x}}|} \prod_{i\in I} p\bigg(\frac{j_i}{n}\bigg)^{(1-p)/(2\eps)},
	\end{align*}
	which further rewrites as the sum of $\eta(\mathbf{x},U)$ and
	\begin{align*} 
		\cO\bigg(\frac{r}{j_0^{1/6}} \prod_{i\in I_{\mathbf{x}}} p\bigg(\frac{j_i}{n}\bigg)^{\hspace{-0.5em}(1-p)/(2\eps)} \hspace{-1.2em}\sum_{I_{\mathbf{x}} \subseteq I\subseteq [r]} \prod_{i\in I\setminus I_{\mathbf{x}}} p\bigg(\frac{j_i}{n}\bigg)^{\hspace{-0.5em}(1-p)/(2\eps)}\bigg) = \cO\bigg(\frac{2^r r}{j_0^{1/6}} \prod_{i\in I_{\mathbf{x}}} p\bigg(\frac{j_i}{n}\bigg)^{\hspace{-0.5em}(1-p)/(2\eps)}\bigg).
	\end{align*}
	The corollary follows by observing that the constant in the $\cO$-term is uniform over the choice of the parameters, as this is the case in Lemma~\ref{lemma:condsurv} as well.
	\end{proof}
	
	\begin{lemma}\label{lemma:oldvertgen}
Fix $i = i(n) = \omega(1)$ and \lyu{recall that $V_n = \{v_1, \ldots, v_{|V_n|}\}$.}
\begin{enumerate}[(i)] 
	\item\label{part i} $\lyu{\mathbb P(V_n\neq \emptyset \text{ and } v_1\le i)}\le \frac{i^{p/(2\eps)}}{n^{(1-p)/(2\eps)}}$;
	\item\label{part ii} Fix $\ell = \lceil 4\eps (\log n)^{\lyu{6}p/\eps}\rceil$. Then, there is a sequence $(\delta_i)_{i=1}^{\infty}$ satisfying $\delta_n = o(1)$ such that w.h.p., for all $j\in [\ell, |V_n|]$, one has
	\begin{equation}
		v_j\in \left[(1-\delta_n)n^{(1-p)/p} \left(\frac{j}{2\eps}\right)^{(2\eps)/p}, (1+\delta_n) n^{(1-p)/p} \left(\frac{j}{2\eps}\right)^{(2\eps)/p}\right].
	\end{equation} 
\end{enumerate}
\end{lemma}

\begin{proof}
\lyu{If $p=1$, both parts hold trivially. Assume that $p\in (1/2,1)$.}
First, we conduct a first moment computation used in the proof of both Parts~\eqref{part i} and~\eqref{part ii}. For every $i\in \lyu{[n]}$, set $M_{i,n} := |V_i\cap V_n|$ and recall the event $\cQ_i = \{||V_i| - 2\eps i|\le i^{2/3}\}$. 
\lyu{Using that $(\P{\nu_{j,n} = 1})_{j=1}^i$ is increasing with respect to $j$} together with Lemma~\ref{lemma:condsurv} applied with \lyu{$j_0 = \lfloor i^{1/2}\rfloor$,} $i$ instead of $j$, $S = \emptyset$ and $R$ being a single vertex in $V_i$, this implies that
\lyu{
	\begin{align}
		\mathbb E[M_{i,n}] 
		&=\; \sum_{j=1}^i \mathbb P(\nu_{j,n} = 1) = \cO(j_0 \mathbb P(\nu_{j_0,n} = 1))+\sum_{j=j_0+1}^i \mathbb P(\nu_{j,i} = 1)\nonumber\\
		&=\; \cO\Bigg(j_0 p \left(\frac{j_0}{n}\right)^{(1-p)/2\eps}\Bigg)+\sum_{j=j_0+1}^i p\left(1+\cO\left(\frac{1}{i^{1/6}}\right)\right)\left(\frac{j}{n}\right)^{(1-p)/2\eps}\hspace{-0.3em}.
	\end{align}
	However, the sum on the right hand side is or order $\Theta(i\cdot (i/n)^{(1-p)/2\varepsilon}) = \omega(j_0\cdot (j_0/n)^{(1-p)/2\varepsilon})$, which allows us to rewrite the last expression as
	\begin{align}
		\left(1+\cO\left(\frac{1}{i^{1/6}}\right)\right)\sum_{j=1}^i p\left(\frac{j}{n}\right)^{(1-p)/2\eps}
		=\; &\left(1+\cO\left(\frac{1}{i^{1/6}}\right)\right) \frac{2\eps i^{1+(1-p)/(2\eps)}}{n^{(1-p)/(2\eps)}}\nonumber\\ 
		=\; &\left(1+\cO\left(\frac{1}{i^{1/6}}\right)\right) \frac{2\eps i^{p/(2\eps)}}{n^{(1-p)/(2\eps)}},\label{eq exact}
	\end{align}
	where the implicit constants in the asymptotic notation here and below are uniformly bounded in absolute value.}
Part~\eqref{part i} readily follows from~\eqref{eq exact} by Markov's inequality \lyu{and the fact that $2\eps < 2$}. 

\lyu{We turn to a proof} of Part~\eqref{part ii} by a second moment argument. From now on, we assume that $i = \omega(n^{(1-p)/p})$ so that $\mathbb E M_{i,n} = \omega(1)$. \lyu{Denote $\varphi(j_1,j_2) = \mathbb P(\nu_{j_1, n} = \nu_{j_2, n} = 1)$ for ease of writing and note that $\varphi(\cdot, \cdot)$ is increasing with respect to each of its coordinates.} 
\lyu{Then, by combining the said monotonicity with two consecutive applications of Lemma~\ref{lemma:condsurv}, the sum of $\varphi(j_1,j_2)$ over all pairs of distinct $j_1, j_2\in [i]$ is equal to
	\begin{align} 
		&\sum_{\substack{j_1,j_2\in [j_0];\\ j_1\neq j_2}} \varphi(j_1,j_2) + 2\sum_{\substack{j_1\in [j_0];\\ j_2\in [j_0+1,i]}} \varphi(j_1,j_2) +  \sum_{\substack{j_1, j_2\in [j_0+1,i];\\ j_1\neq j_2}} \varphi(j_1,j_2)\nonumber\\
		=\, 
		&\,\, \cO(j_0^2 \varphi(j_0-1,j_0)) + \cO(2j_0 i \varphi(j_0,i))\nonumber\\
		+\, 
		&\sum_{\substack{j_1,j_2\in [j_0+1,i];\\ j_1\neq j_2}}\bigg(1+\cO\bigg(\frac{1}{j_1^{1/6}}\bigg)\bigg) \bigg(1+\cO\bigg(\frac{1}{j_2^{1/6}}\bigg)\bigg) p^2 \left(\frac{j_1j_2}{n^2}\right)^{(1-p)/2\eps}\nonumber.
	\end{align}
	Again, using that the sum above is of order $\Theta(p^2 i^2 (i/n)^{2(1-p)/2\varepsilon})$ and thus dominates 
	\[i^{1/6}(\cO(j_0^2 \varphi(j_0-1,j_0)) + \cO(2j_0 i \varphi(j_0,i))),\]
	we obtain that the above expression rewrites as
	\begin{align} 
		\sum_{\substack{j_1,j_2\in [j_0+1,i];\\ j_1\neq j_2}} \bigg(1+\cO\bigg(\frac{1}{i^{1/6}}\bigg)\bigg) p^2 \left(\frac{j_1j_2}{n^2}\right)^{\hspace{-0.4em}(1-p)/2\eps}\hspace{-0.5em},\nonumber
\end{align}}
\lyu{which further simplifies to
	\begin{equation}\label{eq:snd_moment}
		\left(1+\cO\left(\frac{1}{i^{1/6}}\right)\right) \sum_{\substack{j_1,j_2\in [i];\\ j_1\neq j_2}} 
		p^2 \left(\frac{j_1j_2}{n^2}\right)^{(1-p)/2\eps} = \left(1+\cO\left(\frac{1}{i^{1/6}}\right)\right) \frac{2\eps i^{p/(2\eps)}}{n^{(1-p)/(2\eps)}}.
\end{equation}}

\noindent
Thus, by combining~\eqref{eq exact} and~\eqref{eq:snd_moment}, we conclude that
\begin{align*}
	\mathbb E[M_{i,n}^2] - \mathbb E[M_{i,n}]^2
	\le\; 
	&\Bigg(\mathbb E[M_{i,n}] + \hspace{-0.8em}\sum_{\substack{j_1, j_2\in [i];\\ j_1\neq j_2}}\hspace{-0.8em} \mathbb P(\nu_{j_1, n} = \nu_{j_2, n} = 1)\Bigg) - \lyu{\left(1+\cO\left(\frac{1}{i^{1/6}}\right)\right)} \frac{4\eps^2 i^{p/\eps}}{n^{(1-p)/\eps}}\\
	=\; 
	&\lyu{\left(1+\cO\left(\frac{1}{i^{1/6}}\right)\right)} \frac{2\eps i^{p/(2\eps)}}{n^{(1-p)/(2\eps)}} + \cO\left( \frac{\lyu{i^{p/\eps-1/6}}}{n^{(1-p)/\eps}}\right)\\
	=\; 
	&\cO\left(\lyu{\frac{1}{i^{1/6}}} + \frac{n^{(1-p)/(2\eps)}}{i^{p/(2\eps)}}\right)\mathbb E[M_{i,n}]^2.
\end{align*}
Define $\psi(i) := \lyu{i^{-1/6} + n^{(1-p)/(2\eps)}i^{-p/(2\eps)}}$. Then, Chebyshev's inequality implies
\begin{equation}\label{eq chebyshev ineq}
	\mathbb P(|M_{i,n}-\mathbb E[M_{i,n}]|\ge \psi(i)^{\lyu{1/3}} \mathbb E[M_{i,n}])\le \frac{\mathbb E[|M_{i,n} - \mathbb E [M_{i,n}]|^2]}{(\psi(i)^{\lyu{1/3}} \mathbb E[M_{i,n}])^2} = \cO\big(\psi(i)^{\lyu{1/3}}\big).
\end{equation}

Now, set $\alpha = \alpha(n) = ((n^{(1-p)/p}(\log n)^{\lyu{12}})^{-1} n)^{1/\lfloor \log n\rfloor^2} \lyu{= 1+o(1)}$, and for all integers $k$ between $0$ and $\lfloor \log n\rfloor^2$, set $i_k = \lfloor \alpha^k n^{(1-p)/p} (\log n)^\lyu{12}\rfloor$ \lyu{so that, in particular, $i_{\lfloor \log n\rfloor^2} = n$}. By~\eqref{eq chebyshev ineq}, we have that for all $k$ as above, $|M_{i_k, n} - \mathbb E[M_{i_k, n}]| \geq \psi(i_k)^{\lyu{1/3}} \mathbb E[M_{i_k, n}]$ holds with probability $\cO\big(\psi(i_k)^{\lyu{1/3}}\big)$. Moreover, since $i_k$ is increasing in $k$ (for $n$ sufficiently large) and $\psi(i)$ decreasing in $i$,
\begin{equation}\label{eq:unif_O}
	\sum_{k=\lyu{0}}^{\lfloor \log n\rfloor^2}\psi(i_k)^{\lyu{1/3}}\leq \lyu{(\lfloor \log n\rfloor^2 + 1)} \psi(i_0)^{\lyu{1/3}}=\cO\big((\log n)^{\lyu{2-2p/\eps}}\big)=o(1).
\end{equation}
Thus, a union bound over all $\lfloor \log n\rfloor^2 + 1$ values of $k$ shows that the event
$$\mathfrak G := \{\text{for all \lyu{non-negative} integers } k\le \lfloor (\log n)^2\rfloor, M_{i_k, n} = (1+o(1)) \mathbb E [M_{i_k, n}]\},$$
holds w.h.p.\ \lyu{with the $o(1)$ equal to $\max_{k\in [0, \lfloor \log n\rfloor^2]} \psi(i_k)^{1/3}$ (so uniform over different $k$).}

Finally, note that $(M_{i,n})_{i=1}^n$ is an increasing sequence of random variables since for all $i<j$ one has $V_i\cap V_n\subseteq V_j\cap V_n$. Thus, conditionally on the event $\mathfrak G$, for all positive integers $k\le \lfloor (\log n)^2\rfloor$ and $i\in [i_{k-1}, i_k]$,
\begin{align*}
	1-o(1) 
	&=\; (1-o(1)) \frac{\mathbb E [M_{i_{k-1},n}]}{\mathbb E [M_{i_k,n}]}\le \frac{M_{i_{k-1},n}}{\mathbb E [M_{i_k,n}]}\le\frac{M_{i,n}}{\mathbb E [M_{i,n}]}\\
	&\le\; \frac{M_{i_k,n}}{\mathbb E [M_{i_{k-1},n}]}\le (1+o(1))\frac{\mathbb E [M_{i_k,n}]}{\mathbb E [M_{i_{k-1},n}]} = 1+o(1),
\end{align*}
where the first and the last equalities follow from the fact that, for every positive integer $k\le \lfloor (\log n)^2\rfloor$, 
\[\frac{\mathbb E [M_{i_{k-1},n}]}{\mathbb E [M_{i_k,n}]} = (1+o(1)) \bigg(\frac{i_{k-1}}{i_k}\bigg)^{p/(2\eps)} = 1+o(1).\]
Hence, w.h.p.\  $M_{i,n} = (1+o(1)) \mathbb E [M_{i,n}]$ for all $i\in [i_0, n]$. 

\lyu{Now, fix $j\in [\ell, |V_n|]$. Since $\ell = (2+o(1))\mathbb E[M_{i_0,n}]$, w.h.p.\ there exists a smallest integer $i\in [i_0,n]$ such that $M_{i,n} = j$.
	As $M_{i,n} = (1+o(1))\mathbb E[M_{i,n}]$ (where the $o(1)$ is uniform over $i$ since the $o(1)$ in $\mathfrak G$ is uniform over different $k$) and, by definition, $v_j = \min\{t\in [n]: M_{t,n} = j\}$, $v_j$ satisfies that
	\begin{equation*}
		(1+o(1))\frac{2\eps v_j^{p/(2\eps)}}{n^{(1-p)/(2\eps)}} = j,
	\end{equation*}
	which yields $v_j= (1+o(1)) n^{(1-p)/p} (j/2\eps)^{2\eps/p}$ and completes the proof of the lemma.}
	\end{proof}
	
	\section{Theorem~\ref{thrm:LWC intro}: setting up the framework}\label{sec:bfs} 
	
	To start this section, we introduce some important notation. For any non-negative integer $r$, any graph $G$ and any vertex $v$ in $G$, we let $\partial B_r(G, v)$ denote the set of vertices at graph distance exactly $r$ from $v$ in $G$, that is, $\partial B_r(G, v) := V(B_r(G, v)\setminus B_{r-1}(G, v))$. 
	We recall that some graphs in this paper are naturally defined as rooted but we often omit the root from the notation for simplicity. Also, recall that, for a random variable $X$, we let $\mathfrak{L}(X)$ denote the distribution of $X$.
	
	\paragraph{The exploration process.} Recall the definition of the DRGVR model $G_n$ and the Binomial birth-death tree $\cT$ in Definitions~\ref{def:drgvr} and~\ref{def:wll}, respectively. To prove the local convergence \lyu{in distribution} of $G_n$ to $\cT$, we couple the neighbourhoods of the root 0 in $\cT$ and the neighbourhoods of a vertex $k_0$ in $G_n$ chosen uniformly at random. 
	More precisely, we couple the breadth-first search (BFS) exploration \lyu{from the vertex} $k_0$ in $G_n$ with the iterative construction of $\cT$ starting from the root vertex $0$. In this section, we provide some notation for this BFS exploration.
	
	We start by introducing the \emph{Ulam-Harris tree}, which we use to unify the notation that underpins the construction of the multi-type branching process $\cT$ and the BFS exploration of the neighbourhood of $k_0$ in $G_n$. The Ulam-Harris tree is an infinite rooted tree constructed as follows: its root vertex is denoted by $0$ and, for all $r\ge 0$, the children of any vertex $u:=(0,u_1,\ldots, u_r)$ (where $u_1,\ldots, u_r\in \N$) are given by $((u,i))_{i\ge 1}$. In this way, the vertex $u:=(0,u_1,\ldots, u_r)$ is the $u_r^{\text{th}}$ child of the $u_{r-1}^{\text{th}}$ child of $\ldots$ of the $u_1^{\text{th}}$ child of the root $0$ in the \lyu{breadth}-first order. 
	
	Furthermore, we introduce an ordering $\luh$ on the vertices in the Ulam-Harris tree. For two vertices $u:=(0,u_1, \ldots, u_r)$ and $v:=(0,v_1,\ldots, v_k)$ (with $r,k\in \N_0$ and $(u_i)_{i=1}^r \in \N^r, (v_i)_{i=1}^k\in \N^k$), we write $u\luh v$ when $u$ is smaller than $v$ in BFS order, that is, when either $r<k$ or $r=k$ and $u_j<v_j$, where $j:=\min\{i\in[r]: u_i\neq v_i\}$. For example, $(0,1,2,3)\luh (0,1,2,3,4)$ and $(0,2,5,4)\luh (0,2,6,3)$. In a similar manner, we define the ordering $\lequh$.
	
	From now on, we see the tree $\cT$ as a random sub-tree of the Ulam-Harris tree. More precisely, if a vertex $u := (0, u_1, \ldots, u_r)\in \cT$ has $\tau_u$ children, these are encoded $(u,i)_{i=1}^{\tau_u}$ in increasing order of their types. \lyu{That is, upon existence, the child of $u$ with smallest type is encoded $(u,1)$, the child with the second smallest type is encoded $(u,2)$, etc.}
	Moreover, we slightly abuse notation and use the Ulam-Harris formalism for the BFS exploration of the neighbourhood of $k_0$ in $G_n$. \lyu{(While $G_n$ is not a tree, the vertex $k_0$ is typically not contained in a short cycle, which allows a comparison between the balls with finite radii around $0$ in $\cT$ and around $k_0$ in $G_n$.)} 
	This provides more structure when keeping track of the BFS exploration \lyu{of $G_n$} and the parallel iterative construction of $\cT$.
	For a vertex \lyu{$u=(0,u_1,\ldots, u_r)$} for some $r\ge 0$ and \lyu{$u_1, \ldots, u_r\in \N$} in the sub-tree of the Ulam-Harris tree obtained from the BFS exploration of \lyu{$k_0$} in $G_n$, we let $k_{u}\in [n]$ denote its label in $G_n$ \lyu{(which coincides with its arrival time)}. In what follows, we identify the vertex $u$ in the BFS exploration of $G_n$ by its label $k_{u}$, \lyu{and let $\theta_u$ denote the number of neighbours of $v$ explored after $v$ (that is, its children)}. 
	To illustrate: \lyu{vertex $(0)$} has label $k_0$, and this is the root of the BFS exploration. The neighbours of \lyu{$(0)$} are $((0,i))_{i=1}^{\theta_0}$, where $\theta_0$ denotes the number of neighbours of $k_0$ in $G_n$, and their labels are $(k_{(0,i)})_{i=1}^{\theta_0}$. This continues throughout the BFS exploration. For convenience of notation, we skip the parentheses and write $k_{0,u_1, \ldots, u_r}$ for the label of the vertex $(0,u_1, \ldots, u_r)$, \lyu{and also skip the parentheses and write $0$ for the root $(0)$}.\\
	
	In the BFS exploration \lyu{of $G_n$}, we consider nodes to be active, probed or neutral and let $(\cA_t, \cP_t, \cN_t)_{t\in\N_0}$ denote the sets of active, probed and neutral vertices after $t$ steps of the exploration, respectively. We initialise the process by setting 
	\begin{equation}
(\cA_0,\cP_0,\cN_0)=(\{k_0\},\emptyset,V_n\backslash \{k_0\}).
\end{equation} 

\lyu{If $\cA_{t-1}$ is not empty,} we also let $k[t]$ denote the smallest vertex for the BFS order in $\cA_{t-1}$. That is, \lyu{for every integer} $t\geq1$, if $k_{u}\in \cA_{t-1}$ and $u\luh  w$
for all \lyu{$k_w\in \cA_{t-1}\backslash \{k_u\}$}, then $k[t]=k_{u}$. Let $\cD_t:=\{w\in \cN_{t-1}: \{w,k[t]\}\in E(G_n)\}$ denote the set of neutral vertices attached to $k[t]$. Then, \lyu{if $\cA_\lyu{t-1}\neq \emptyset$,} we update 
\begin{equation}
\lyu{(\cA_{t}, \cP_{t}, \cN_{t})=((\cA_{t-1}\backslash\{k[t]\})\cup \cD_t, \cP_{t-1}\cup \{k[t]\},\cN_{t-1}\,\backslash\, \cD_t)}
\end{equation} 
and, if $\cA_\lyu{t-1}=\emptyset$, we set $(\cA_{\lyu{t}}, \cP_{\lyu{t}}, \cN_{\lyu{t}})=(\cA_{\lyu{t-1}}, \cP_{\lyu{t-1}}, \cN_{\lyu{t-1}})$.

Given the definition of the DRGVR model, it is convenient to split the neighbours of each vertex in the BFS exploration into two parts: indeed, for each vertex, the probability to be adjacent to a vertex with a \emph{smaller} label is different from the probability to be adjacent to a vertex with a \emph{larger} label. 
As a result, for each vertex, we consider its neighbours to the \emph{left} (i.e.\ vertices with a smaller label) and to the \emph{right} (i.e.\ vertices with a larger label) separately. 
Moreover, we refer to these as the $L$-neighbours and $R$-neighbours of $u$ (or of $k_u$), respectively. Furthermore, we let $\theta_u^L$ and $\theta_u^R$ denote the number of $L$-neighbours and $R$-neighbours of $k_u$, respectively, and note that $\theta_u=\theta_u^L+\theta_u^R$.

In an equivalent manner, we define the $L$-children and the $R$-children of a vertex $u$ in the multi-type branching process $\cT$ as in Definition~\ref{def:wll}. Letting $\tau_u$ denote the number of children of $u$, the $L$-children (resp.\ the $R$-children) of $u$ are the children $(u,i)_{i=1}^{\tau_u}$ such that $a_{u,i}<a_u$ (resp.\ $a_{u,i}>a_u$). Let us write $\tau_u^L$ and $\tau_u^R$ for the number of $L$-children and $R$-children of $u$, respectively. As we order the children in increasing order of their types, it follows that $(u,i)_{i=1}^{\tau_u^L}$ are $u$'s $L$-children and $(u,i)_{i=\tau_u^L+1}^{\tau_u^L+\tau_u^R}$ are $u$'s $R$-children. 

\section{\texorpdfstring{Coupling the $1$-neighbourhoods of $k_0$ in $G_n$ and $0$ in $\cT$}{}}\label{sec:lwl}

In this section, we provide a coupling between the neighbours of $k_0$ in the BFS exploration of $G_n$ and the children of the root $0$ of $\cT$. The coupling should be such that:
\begin{enumerate}[(i)]
\item the number of $L$- and $R$-neighbours of $k_0$ in $G_n$ are equal to the number of $L$- and $R$-children of $0$, 
\item the rescaled labels \bas{(i.e.\ the marks)} of the neighbours of $k_0$ and the types \bas{(i.e.\ the types)} of the children of $0$ are approximately the same.
\end{enumerate}
Whilst $(ii)$ is not directly necessary for the $1$-neighbourhoods of $k_0$ and $0$ to be isomorphic, it is necessary \bas{to couple the marks of the vertices in both graphs}. It also ensures we can construct coupling such that the $2$-neighbourhoods are isomorphic.

In this section, we prove that the following events occur with high probability:  
\begin{equation}\begin{aligned}\label{eq:H1}
	\cH_{1,0}&:=\{  a_0\geq 1/\log\log n\},\\
	\cH_{1,1}&:=\{\min_{i\in[  \tau_0^L]}  a_{0,i}>1/(\log\log n)^2\},\\
	\cH_{1,2}&:=\{(B_1(G_n,k_0),k_0)\cong (B_1(\cT,0),0)\}\cap \{\forall u\in V(B_1(\cT,0)), |  a_{u}-\tfrac{k_{u}}{n}|\leq \tfrac{2}{n}\},\\
	\cH_{1,3}&:=\{ \tau_0^L+ \tau_0^R< \log\log n\}.
	\end{aligned}\end{equation}  
	In words, the events $\cH_{1,0}$ and $\cH_{1,1}$ control the types of the root $0$ and the $L$-children of the root $0$ in $\cT$, and ensure that these types are not ``too small''. Similarly, the event $\cH_{1,3}$ ensures that the number of children of the root $0$ in $\cT$ is not ``too large''. The event $\cH_{1,2}$ then states that the $1$-neighbourhoods of $k_0$ in $G_n$ and $0$ in $\cT$ are isomorphic and that the rescaled labels of the neighbours of $k_0$ in $G_n$ are very close to the types of their corresponding counterparts in $\cT$. \lyu{In what follows, similar events ensure the possibility to extend the coupling from the $r$-th to the $(r+1)$-st neighbourhoods as well}. More details related to this follow in the upcoming sections.
	
	We now state the main result of this section.
	
	\begin{lemma}\label{lemma:1couple}
Fix $\beta > 0$ and $\eps\in (0,1/2]$, and consider the DRGVR model and the Binomial birth-death tree model given in Definitions~\ref{def:drgvr} and~\ref{def:wll}, respectively. Recall the events $(\cH_{1,i})_{i=0}^3$ from~\eqref{eq:H1}. There exists a constant $C = C(\beta, \eps) > 0$ such that, for all sufficiently large $n$, 
\begin{equation}
	\mathbb P\bigg(\Big(\bigcap_{i=0}^3 \cH_{1,i}\Big)^c\bigg)\leq \frac{C}{(\log\log n)^{p/(2\eps)}}.
\end{equation} 
\end{lemma}

Note that we can bound the probability in the statement of Lemma~\ref{lemma:1couple} from above by 
\begin{equation}\label{eq:1couplebound}
\mathbb P\bigg(\Big(\bigcap_{i=0}^3 \cH_{1,i}\Big)^c\bigg)\leq \P{\cH_{1,0}^c}+\P{\cH_{1,1}^c\cap \cH_{1,0}}+\P{\cH_{1,3}^c}+\P{\cH_{1,2}^c\cap \cH_{1,0}}.
\end{equation} 
The first term on the right-hand side is readily bounded: since $a_0\sim \text{Beta}(\tfrac{p}{2\eps}, 1)$, it directly follows that 
\begin{equation}\label{eq:missed}
\P{\cH_{1,0}^c}=(\log\log n)^{-p/(2\eps)}.
\end{equation}
Let us write $\mathbf V^L_0$ for the Poisson point process on $(0,a_0)$ with \lyu{intensity} $\lambda_0^-(\dd x)$ as in~\eqref{eq:lambda}. Using the construction of the offspring of the root as in Definition~\ref{def:drgvr}, we can write
\begin{equation}
\cH_{1,1}=\{\mathbf V^L_0\cap (0,1/(\log\log n )^2)=\emptyset\}.
\end{equation} 

As a result, conditionally on $a_0$, 
\begin{equation}\begin{aligned}\label{eq:h11bound}
	\mathbb P(\cH^c_{1,1}\cap \cH_{1,0})=\P{\mathbf V^L_0\cap (0,1/(\log\log n)^2)\neq \emptyset\,|\, \cH_{1,0}}\P{\cH_{1,0}}
	\end{aligned}\end{equation}
	
	\noindent
	Conditionally on $a_0$ and the event $\cH_{1,0}$, and using that $1-\e^{-x}\leq x$ for all $x\in\R$, the probability of the event $\{\mathbf{V}_0^L\cap(0,1/(\log\log n)^2)\not=\emptyset\}$  is bounded from above by
	\begin{equation}
1-\exp\Big(-\frac{\beta}{a_0}(\log\log n)^{-p/\eps}\Big)\le 1-\exp\big(-\beta (\log\log n)^{1-p/\eps}\big)\leq \beta \big(\log\log n\big)^{-1/(2\eps)},
\end{equation} 
so that taking the \lyu{expectation} with respect to $a_0$ and using~\eqref{eq:h11bound} yields 
\begin{equation}\label{eq:h10h11bound}
\P{\cH^c_{1,1}\cap \cH_{1,0}}=\cO\Big(\big(\log\log n\big)^{-1/(2\eps)}\Big).
\end{equation} 
Furthermore, conditionally on $a_0$, $\tau^L_0+\tau^R_0$ is a Poisson random variable with rate 
\begin{equation}\begin{aligned}\label{eq:meanoffspring}
	\int_0^{a_0}\lambda_{0}^-(\dd x)+\int_{a_0}^1 \lambda^+(\dd x)&=\frac{\beta p}{2\eps}\bigg[\int_0^{a_0} \frac{x^{(1-p)/(2\eps)}}{a_0}\,\dd x+\int_{a_0}^1 x^{(1-p)/(2\eps)-1}\,\dd x\bigg]\\
	&=\beta a_0^{(1-p)/(2\eps)}+\frac{\beta p}{1-p}\big(1-a_0^{(1-p)/(2\eps)}\big)\leq \frac{\beta p}{1-p}.
	\end{aligned} \end{equation} 
	Thus, $\tau^L_0+\tau^R_0$ is stochastically dominated by $\tau'\sim \lyu{\text{Po}}(\tfrac{\beta p }{1-p})$. It follows that, for every $t > 0$,
	\begin{equation}\label{eq:taubound}
\P{\cH_{1,3}^c}\leq \P{\tau'\geq \log\log n}\leq \frac{\mathbb E[\e^{t\tau'}]}{\e^{t\log \log n}} =  \frac{\e^{(\beta p /(1-p))(\e^t-1)}}{(\log n)^t}=o((\log\log n)^{-p/(2\eps)}),
\end{equation} 
where the second inequality \lyu{is an application of Markov's inequality for the random variable $\e^{t\tau'}$}. The main goal of this section is thus to construct a coupling such that the probability of $\cH_{1,2}^c\cap \cH_{1,0}$ is $\cO(\big(\log\log n)^{-p/(2\eps)})$, which completes the proof Lemma~\ref{lemma:1couple}.

\subsection{\texorpdfstring{Coupling the roots of $G_n$ and $\cT$}{}}

We start by coupling the root vertex $0$ of $\cT$ and the uniform vertex $k_0$ from $V_n$. Due to the removal of vertices in this model, this is non-trivial. In the proofs of the local convergence \lyu{in distribution} of preferential attachment models to the P\'olya point tree, which are similar in nature, the vertex $k_0$ is uniform among $[n]$ and the root of the tree has a type $a_0\sim \text{Unif}(0,1)$. Hence, setting $k_0:=\lceil a_0n\rceil$ provides a simple coupling where $k_0/n$ and $a_0$ are sufficiently close. In our case, since $V_n$ is a random subset of $[n]$, this coupling requires more work.

Let us define
\begin{equation}\label{eq:rn} 
\cR_0:=\{a_0\leq \tfrac{k_0}{n}\leq a_0+\tfrac{1}{n}\}.
\end{equation} 
We show the following result.

\begin{lemma}\label{lemma:rootcouple}
Consider the DRGVR model and the Binomial birth-death tree model as in Definitions~\ref{def:drgvr} and~\ref{def:wll}, respectively. Let $k_0$ be a uniform vertex from $V_n$. Then, there exists a constant $c=c(\beta, \eps)>0$ and a coupling of $k_0$ and the root vertex $0$ of $\cT$ such that
\begin{equation}
	\P{\cR_0^c}\leq \frac{c}{\log n}.
\end{equation} 
\end{lemma}

\begin{proof} 
We sample an i.i.d.\ sequence $(X_j,Y_j)_{j\in\N}$ with distribution $\Gamma$ (defined in~\eqref{eq:gamma}). Let $\kappa=\inf\{j\in\N: Y_j=1\}$ and set $a_0:=X_\kappa$ as the type of $0$. By Lemma~\ref{lemma:beta}, $a_0$ has the desired distribution. We define $k_0^{(j)}:=\lceil n X_j\rceil$ as a uniform vertex in $[n]$, and write its mark as $\nu^{(j)}:=\nu_{k_0^{(j)},n}$ for brevity. We also introduce $\psi:=\inf\{j\in\N: \nu^{(j)}=1\}$ and set $k_0:=k_0^{(\psi)}$. Note that since the mark of $k_0$ equals one, $k_0$ is uniform among all vertices with mark one, so $k_0$ also has the desired distribution. Our aim is to couple $\kappa$ and $\psi$ so that $\kappa=\psi$ with probability at least $1-c/\log n$. Indeed, conditionally on the equality $\kappa=\psi$, it follows that $a_0\leq \tfrac{k_0}{n}\leq a_0+ \tfrac{1}{n}$ holds almost surely.

\lyu{For every $j$,} we couple $\nu^{(j)}$ with $Y_j\sim \text{Ber}(p(X_j)^{(1-p)/(2\eps)})$ using a standard Bernoulli coupling. Conditionally on $X_j$, this coupling fails with probability
\begin{equation}\label{eq:couplefail}
	\big|\mathbb P\big(\nu^{(j)}=1\,|\, X_j\big)-\mathbb P(Y_j=1\,|\, X_j)\big|. 
\end{equation} 
We say the coupling \lyu{fails} at step $j$ when \lyu{$\nu^{(j)}\neq Y_j$}, and that the entire coupling \lyu{fails} when \lyu{$\kappa\neq \psi$, that is, the coupling fails for some $j\in[\kappa]$}. Hence, 
\begin{equation} 
	\P{\text{Coupling fails}}=
	\E{\P{\cup_{j=1}^{\kappa}\{\text{Coupling fails at step $j$}\}\,\Big|\, \kappa}}.
\end{equation} 
\lyu{We now use the coupling between $\nu^{(j)}$ and $Y_j$, conditionally on $X_j$, with failure probability as in~\eqref{eq:couplefail}, to bound the probability of the event $\{\nu^{(j)}\neq Y_j\}$ from above, independently of the value of $Y_j$ (which is known for each $j\in[\kappa]$ when we condition on $\kappa$).}
Using a union bound and observing that the expression in~\eqref{eq:couplefail} is in fact independent of $j$, we thus obtain the upper bound
\begin{equation}\begin{aligned}
		\E{\kappa\E{\big|\mathbb P\big(\nu^{(1)}=1\,|\,X_1\big)-\P{Y_1=1\,|\, X_1}\big|}},
\end{aligned}\end{equation} 
where the inner \lyu{expectation} is with respect to $X_1$ and the outer \lyu{expectation is} with respect to $\kappa$. 
\bas{By Lemma~\ref{lemma:condsurv} applied with $R = \{\lceil nX_1\rceil\},S=\emptyset$ and $j_0=j=\lceil nX_1\rceil$ \lyu{(so that $\cW_n(S)\cap \cW_j(R)=\Omega$),}
	\be 
	\ind_{\{X_1\geq (\log n)^{-1}\}}\P{\nu^{(1)}=1\,|\, X_1}=\ind_{\{X_1\geq (\log n)^{-1}\}}p\Big(\frac{\lceil nX_1\rceil}{n}\Big)^{(1-p)/(2\eps)}(1+\cO(\lyu{\lceil nX_1\rceil^{-1/6}})).
	\ee 
}
Hence, for some positive constants $C, c>0$,
\begin{equation}\begin{aligned}
		\P{\text{Coupling fails}}
		&\le\; \mathbb P(\text{Coupling fails}\mid X_0^1\geq (\log n)^{-1}) + \mathbb P(X_0^1 < (\log n)^{-1})\\
		&\le\; \lyu{C\left(\frac{\log n}{n}\right)^{1/6}\E{\kappa_0}}+\frac{1}{\log n}\leq \frac{c}{\log n},
\end{aligned}\end{equation}
where we used that $\mathbb E[\kappa] < \infty$ since $\kappa$ has a geometric distribution with parameter $2\eps$. This concludes the proof.
\end{proof}

Having coupled the roots, we can start coupling their direct neighbours. Recall that $\theta_0^L$ and $\theta_0^R$ denote the number of $L$-neighbours and $R$-neighbours of $k_0$, respectively. The aim is to ensure that $\theta_0^L= \tau^L_0$, $\theta_0^R= \tau^R_0$, and that for any neighbour $((0,i))_{i\in[\tau^L_0+ \tau^R_0]}$, it holds that $\tfrac{k_{0,i}}{n}\approx a_{0,i}$. Observe that the latter is not required for the local convergence of the $1$-neighbourhood to hold but is needed to ensure that the $2$-neighbourhoods of the root $0$ in $\cT$ and $k_0$ in $G_n$ can be coupled in later parts of the proof of the local weak convergence.

\subsection{\texorpdfstring{The $1$-neighbourhood of the root}{}}

The children of the root $0$ in $\cT$ are distributed according to a Poisson point process, whilst the direct neighbours of the root $k_0$ in $G_n$ are characterised by a Bernoulli point process (that is, a sequence of Bernoulli random variables). 
To provide a coupling between these \lyu{stochastic processes}, we split the neighbours of $k_0$ into two groups: the $L$-neighbours are the vertices $j\in [k_0-1]$ such that $k_0\to j$ in $G_n$ while the $R$-neighbours are those $j\in\{k_0+1,\ldots, n\}$ for which $j\to k_0$ in $G_n$. As the connection probability to the $L$- and $R$-neighbours is somewhat different, we treat them separately. 

\subsubsection{\texorpdfstring{$L$-children of the root $0$ and $L$-neighbours of $k_0$ in $G_n$}{}}

The $L$-neighbours of the root $k_0$ of $G_n$ can be encoded by a Bernoulli point process $\mathbf I^L_{k_0}=(I_{k_0\to j})_{j=1}^{k_0-1}$, where $I_{k_0\to j}$ equals one if $k_0$ connects to $j$ in $G_n$. 
Similarly, the $L$-children of the root $0$ of $\cT$ can be encoded by a Poisson point process \lyu{with intensity~$\lambda^-_{a_0}$ on the interval $(0,a_0)$ as defined in~\eqref{eq:lambda}.}
We couple the two processes by discretising \lyu{the Poisson process} as follows: for all $j\in[k_0-1]$, define $V_{k_0\to j}\sim \lyu{\text{Po}}(\lambda_j)$ where 
\begin{equation}\ba \label{eq:lambdal}
\lambda_j&:=\int_{\frac{j-1}{n}}^{\frac{j}{n}} \frac{\beta p}{2\eps a_0}x^{(1-p)/(2\eps)}\,\dd x \quad&&\text{ if } j\in[k_0-2],\text{ and }\\ 
 \lambda_{k_0-1}&:=\int_{\frac{k_0-2}{n}}^{a_0} \frac{\beta p}{2\eps a_0}x^{(1-p)/(2\eps)}\,\dd x. &&
 \ea
\end{equation} 
Then, we write $\mathbf V^L_{k_0}=(V_{k_0\to j})_{j=1}^{k_0-1}$.

The following lemma states that $\mathbf I^L_{k_0}$ and $\mathbf V^L_{k_0}$ can be coupled successfully on the event $\cH_{1,0}\cap \cR_0$.

\begin{lemma}\label{lemma:Lroot}
	Consider the Bernoulli point process $\mathbf I^L_{k_0}:=(I_{k_0\to j})_{j=1}^{k_0-1}$ and the discretised Poisson point process $\mathbf V^L_{k_0}:=(V_{k_0\to j})_{j=1}^{k_0-1}$ with means defined in~\eqref{eq:lambdal}. Recall the events $\cH_{1,0}$ and $\cR_0$ from~\eqref{eq:H1} and \eqref{eq:rn}, respectively, and set $\cE_0:=\cR_0\cap \cH_{1,0}$. 
	There exists a coupling of $\mathbf I_{k_0}^L$ and $\mathbf V_{k_0}^L$ such that 
	\be 
	\P{\{ \mathbf I_{k_0}^L\neq \mathbf V_{k_0}^L\}\cap \cE_0\mid k_0}=\cO\Big(\frac{1}{\log n}\Big),
	\ee 
 \lyu{where the constant in the $\cO(\cdot)$ is absolute. Equivalently, upon the event $\cE_0$ and conditionally on $k_0$, the total variation distance between $\mathbf I_{k_0}^L$ and $\mathbf V_{k_0}^L$ is of order $\cO(1/\log n)$.}
	\end{lemma}

\begin{remark}\label{rem:rootL}
	When the coupling \lyu{of the roots in Lemma~\ref{lemma:rootcouple} and the coupling of} the Bernoulli process and the discretised Poisson point process is successful, we immediately obtain that $\theta_0^L=  \tau_0^L$, that is, the number of $L$-neighbours of the root $k_0$ in $G_n$ and the number of $L$-children of the root $0$ in $\cT$ are the same. Moreover, \lyu{the equality $\mathbf I^L_{k_0} = \mathbf V^L_{k_0}$ implies that for every $i\in [\tau_0^L]$, $|a_{0,i}-\tfrac{k_{0,i}}{n}|\leq \tfrac{1}{n}$, thus ensuring that the rescaled labels of the children of $k_0$ in $G_n$ and the types of the children of $0$ in $\mathcal T$ coincide up to an error of $\tfrac{1}{n}$.}
\end{remark}

\begin{proof}[Proof of Lemma~\ref{lemma:Lroot}]
\lyu{Based on the alternative definition of total variation distance in~\eqref{eq:dTV}, we show the stronger statement that 
\begin{equation}\label{eq:conddtvbound}
\ind_{\cE_0}\sum_{x}\big|\P{\mathbf I_{k_0}^L=x\,\big|\, k_0}-\P{\mathbf V_{k_0}^L=x\,\big|\, k_0}\big| =\cO\Big(\frac{1}{\log n}\Big),
\end{equation}
where the sum ranges over all sequences $x$ of $k_0-1$ non-negative integers.}
	For ease of writing throughout the proof, we let $\mathbb P_{k_0}$ denote the probability measure $\mathbb P_{k_0}(\cdot):=\P{\cdot\,|\, k_0}$. We start by introducing a number of quantities that we use in the proof. 
	Let $\mathbf{\wt I}_{k_0}^L=(\wt I_{k_0\to j})_{j\in[k_0-1]}$ be a sequence of independent Bernoulli random variables whose success probabilities conditionally on $k_0$ are equal to
	\be \label{eq:wtIprob}
	\Pk{\wt I_{k_0\to j}=1}=\min\Big\{\frac{\beta p}{2\eps k_0}\Big(\frac jn\Big)^{(1-p)/(2\eps)},1\Big\}, \qquad j\in [k_0-1]. 
	\ee 
	Further, let $(\wt I_{j,k_0})_{j\in[k_0-1]}$ be a sequence of i.i.d.\ Bernoulli random variables whose success probabilities conditionally on $k_0$ are equal to
	\be 
	\Pk{\wt I_{j,k_0}=1}=\min\Big\{\frac{\beta}{2\eps k_0},1\Big\}, \qquad j\in [k_0-1].
	\ee 
	Finally, let $(\wt \nu_{j,n})_{j\in[k_0-1]}$ be a sequence of i.i.d.\ Bernoulli random variables independent of all other variables with success probability 
	\be \label{eq:wtnu}
	\P{\wt{\nu}_{j,n}=1}=p\Big(\frac jn\Big)^{(1-p)/(2\eps)}. 
	\ee 
	In particular, we observe that, conditionally on $k_0$ \lyu{and when $2\eps k_0\ge \beta$}, $\wt I_{k_0\to j}\overset d=\wt I_{j,k_0}\wt \nu_{j,n}$.
	
	Recall the random vertex marks $(\nu_{j,n})_{j\in[n]}$ where $\nu_{j,n}$ equals one if a vertex is added to the graph at step $j$ and this vertex survives until step $n$, and $\nu_{j,n}$ equals zero otherwise. By the triangle inequality
	\begin{align} 
	\ind_{\cE_0}\!\!\!\!{}&\sum_{x \in \N_0^{k_0-1}}\!\!\!\!\big|\Pk{\mathbf I_{k_0}^L=x}-\Pk{\mathbf V_{k_0}^L=x}\big|\nonumber \\ 
	\leq{}& \ind_{\cE_0}\!\!\!\!\sum_{x \in \{0,1\}^{k_0-1}}\!\!\!\!\Big|\Pk{\mathbf I_{k_0}^L=x}-\Pk{(\wt I_{j,k_0}\nu_{j,n})_{j\in[k_0-1]}=x}\Big|\label{eq:dtv1}\\ 
	&+\ind_{\cE_0}\!\!\!\!\sum_{x \in \{0,1\}^{k_0-1}}\!\!\!\!\Big|\Pk{(\wt I_{j,k_0}\nu_{j,n})_{j\in[k_0-1]}=x}-\Pk{\mathbf{ \wt I}_{k_0}^L=x}\Big|\label{eq:dtv2}\\ 
	&+\ind_{\cE_0}\!\!\!\!\sum_{x \in \N_0^{k_0-1}}\!\!\!\!\Big|\Pk{\mathbf{\wt I}_{k_0}^L=x}-\Pk{\mathbf{ V}_{k_0}^L=x}\Big|.\label{eq:dtv3}
	\end{align} 
	We split the proof into three steps, each of which bounds one of the sums on the right-hand side. 
 
\paragraph{\underline{Step $1$}:} Bounding~\eqref{eq:dtv1}. Let $(I_{j,k_0})_{j\in[k_0-1]}$ denote a sequence of Bernoulli random variables \lyu{that are i.i.d.\ in the probability space conditioned on $k_0$ and $V_{k_0-1}$}, with parameter $\min\{\beta/|V_{k_0-1}|,1\}$ if $V_{k_0-1}\neq \emptyset$ and parameter $0$ otherwise, and also independent of $(\nu_{j,n})_{i\in[k_0-1]}$, conditionally on $k_0$ and $V_{k_0-1}$. Then, we can write $I_{k_0\to j}= I_{j,k_0} \nu_{j,n}$. 
	Indeed, the indicator $I_{k_0\to j}$ equals one when $k_0$ sends an edge to $j$. \lyu{For this, vertex $j$ needs to survive after $n$ steps, which occurs with probability 
		\be 
		\frac{\P{\nu_{j,n}=1,\nu_{k_0,n}=1\,|\, k_0,V_{k_0-1}}}{\P{\nu_{k_0,n}=1\,|\,k_0,V_{k_0-1}}},
		\ee
		and the prospective edge appears with probability $\min\{\beta/|V_{k_0-1}|,1\}$, if $V_{k_0-1}\neq \emptyset$, and probability $0$ otherwise.}
	Writing $I_{k_0\to j}=I_{j,k_0}\nu_{j,n}$ allows us to couple the sequences $\mathbf I_{k_0}^L$ and $(\wt I_{j,k_0}\nu_{j,n})_{j\in[k_0-1]}$ by coupling the sequences $(I_{j,k_0})_{j\in[k_0-1]}$ and $(\wt I_{j,k_0})_{j\in[k_0-1]}$. This can be done by applying a standard Bernoulli coupling and yields, by a union bound,
	\be 
	\Pk{\bigcup_{j=1}^{k_0-1}\{ I_{j,k_0}\neq \wt I_{j,k_0}\}}\leq k_0 \bigg|\min\Big\{\frac{\beta}{2\eps k_0},1\Big\}-\Ek{\ind_{\{V_{k_0-1}\neq \emptyset\}} \min\Big\{\frac{\beta}{|V_{k_0-1}|},1\Big\}}\bigg|.
	\ee 
	Now, recall the event $\cQ_n$ from~\eqref{eq:An}. Using Lemma~\ref{lemma:condconc}, for some constant $C_1>0$, we obtain
	\be \ba
	\Ek{\ind_{\cQ_{k_0-1}} \min\Big\{\frac{\beta}{|V_{k_0-1}|},1\Big\}}
	&\leq\Ek{\ind_{\{V_{k_0-1}\neq \emptyset\}} \min\Big\{\frac{\beta}{|V_{k_0-1}|},1\Big\}}\\
	&\leq \Ek{\ind_{\cQ_{k_0-1}} \min\Big\{\frac{\beta}{|V_{k_0-1}|},1\Big\}}+\Pk{\cQ_{k_0-1}^c}\\ 
	&\leq \min\Big\{\frac{\beta}{2\eps k_0},1\Big\}+C_1k_0^{-4/3}.
	\ea \ee 
	We thus arrive at 
	\be\ba  \label{eq:Icouple}
	\Pk{\exists j\in [k_0-1]: I_{k_0\to j}\neq \wt I_{j,k_0}\nu_{j,n}}&\leq \Pk{\exists j\in [k_0-1]: I_{j,k_0}\neq \wt I_{j,k_0}}\leq C_1k_0^{-1/3},
	\ea\ee 
	which completes the first part. 

    \paragraph{\underline{Step 2}:} Bounding~\eqref{eq:dtv2}. Here, we recall that we can write $\wt I_{k_0\to j}=\wt I_{j,k_0}\wt \nu_{j,n}$. As a result, it is only necessary to consider those $\nu_{j,n}$ and $\wt \nu_{j,n}$ such that $\wt I_{j,k_0}=1$. Since the number of $j\in[k_0-1]$ such that $\wt I_{j,k_0}$ equals one converges to a Poisson random variable with mean $\beta/(2\eps)$ (indeed, the sum of the $\wt I_{j,k_0}$ is a binomial random variable with mean $\beta/(2\eps)-o(1)$, conditionally on $k_0$), it follows that we need only consider very few of the $\nu_{j,n}$ and $\wt \nu_{j,n}$ in practice, allowing us to deal with the dependencies among the $\nu_{j,n}$. To do so, we fix $\delta_1,\delta_2\in(0,1)$ and define the event
	\be 
	\cB_n:=\bigg\{\min\{j\in[k_0-1]:\wt I_{j,k_0}=1\}\geq n^{\delta_1}, \sum_{j=1}^{k_0-1}\wt I_{j,k_0}\leq \delta_2\log n\bigg\}.
	\ee 
	Note that $\cB_n $ holds with high probability irrespective of the choice of $\delta_1,\delta_2$: indeed, by a union bound and Markov's inequality, we can bound 
	\be \label{eq:enbound}
	\P{\cB_n ^c\,|\, k_0}\leq \frac{\beta  n^{\delta_1}}{2\eps k_0}+\frac{\beta}{2\eps \delta_2\log n}.
	\ee 
	We then write
	\be\ba \label{eq:dtvcond}
	\ind_{\cE_0}\sum_{x\in \{0,1\}^{k_0-1}}\!\!\!\!\Big|\mathbb P_{k_0}\Big((\wt I_{j,k_0}{}&\nu_{j,n})_{j\in[k_0-1]}=x\Big)-\Pk{(\wt I_{j,k_0}\wt \nu_{j,n})_{j\in[k_0-1]}=x}\Big|\\ 
	\leq \ind_{\cE_0}\mathbb E_{k_0}\Bigg[\sum_{x\in\{0,1\}^{k_0-1} }\!\!\!\!\!\Big|{}&\P{(\wt I_{j,k_0}\nu_{j,n})_{j\in[k_0-1]}=x\,\Big|\, k_0,(\wt I_{j,k_0})_{j\in[k_0-1]}}\\ 
	&-\P{(\wt I_{j,k_0}\wt \nu_{j,n})_{j\in[k_0-1]}=x\,\Big|\, k_0,(\wt I_{j,k_0})_{j\in[k_0-1]}}\Big|\Bigg].
	\ea\ee 
	Now, define
	\be 
	M_n:=\sum_{j=1}^{k_0-1} \wt I_{j,k_0}. 
	\ee 
	Conditionally on the sequence $(\wt I_{j,k_0})_{j\in[k_0-1]}$, let $(j_i)_{i\in [M_n]}$ denote the sequence of indices in increasing order, such that $\wt I_{j_i,k_0}=1$ for all $i\in[M_n]$. We then observe that, with $x=(x_j)_{j\in[k_0-1]}$ and $x'=(x_{j_i})_{i\in[M_n]}$, the events 
	\be 
	\Big\{(\wt I_{j,k_0}\nu_{j,n})_{j\in[k_0-1]}=x\,\Big|\, k_0,(\wt I_{j,k_0})_{j\in[k_0-1]}\Big\} \quad\text{and}\quad \Big\{(\nu_{j_i,n})_{i\in[M_n]}=x'\,\Big|\, k_0,(\wt I_{j,k_0})_{j\in[k_0-1]}\Big\}
	\ee 
	have the same probability, as long as $x_\ell=0$ for all $\ell\in [k_0-1]\setminus \{j_1,\ldots, j_{M_n}\}$. If, instead, $x_\ell=1$ for some $\ell\in [k_0-1]\setminus \{j_1,\ldots, j_{M_n}\}$, then the event on the left-hand side has probability zero. The same holds when we replace $\nu_{j,n}$ by $\wt \nu_{j,n}$. As a result, using this in~\eqref{eq:dtvcond}, we obtain 
	\be \ba 
	\ind_{\cE_0}\mathbb E_{k_0}\Bigg[\sum_{x\in\{0,1\}^{M_n} }\!\!\!\!\!\Big|{}&\P{(\nu_{j_i,n})_{i\in[M_n]}=x\,\Big|\, k_0,(\wt I_{j,k_0})_{j\in[k_0-1]}}\\ 
	&-\P{(\wt \nu_{j_i,n})_{i\in[M_n]}=x\,\Big|\, k_0,(\wt I_{j,k_0})_{j\in[k_0-1]}}\Big| \Bigg]\\ 
	\leq \ind_{\cE_0}\mathbb E_{k_0}\Bigg[\ind_{\cB_n }\sum_{x\in\{0,1\}^{M_n} }\!\!\!\!\!\Big|{}&\P{(\nu_{j_i,n})_{i\in[M_n]}=x\,\Big|\, k_0,(\wt I_{j,k_0})_{j\in[k_0-1]}}\\ 
	&-\P{(\wt \nu_{j_i,n})_{i\in[M_n]}=x\,\Big|\, k_0,(\wt I_{j,k_0})_{j\in[k_0-1]}}\Big| \Bigg]+2\ind_{\cE_0}\Pk{\cB_n ^c}.
	\ea\ee 
	On the event $\cB_n $, it holds that $M_n\leq \delta_2\log n$ and that $j_i\geq n^{\delta_1}$ for all $i\in M_n$. As a result, by summing over all events $\{M_n=r\}$, with $r=1,\ldots, \lfloor \delta_2\log n\rfloor$, we can, for each such $r$, apply Corollary~\ref{cor:condsurv} (and Remark~\ref{rem:condsurv}, in particular) with $j_0=n^{\delta_1}$ to bound the absolute value from above. Together with~\eqref{eq:enbound}, this yields the upper bound 
	\be 
	\ind_{\cE_0}\Bigg[\cO\Bigg(\sum_{r=1}^{\lfloor \delta_2\log n\rfloor} \frac{4^rr}{n^{\delta_1/6}}\Bigg)+\frac{2\beta n^{\delta_1}}{2\eps k_0}+\frac{\beta}{\eps\delta_2\log n}\Bigg]=\cO\Big(\frac{1}{\log n}\Big), 
	\ee 
	where we use that  the bounds on $k_0$ in the event $\cE_0$ and choose $\delta_2$ sufficiently small. 
	
    \paragraph{\underline{Step 3}:} Bounding~\eqref{eq:dtv3}. We do this in two steps. We first couple $\mathbf{\wt I}_{k_0}^L$ to an auxiliary sequence of independent Poisson random variables $\mathbf{\wt V}_{k_0}^L=(\wt V_{k_0\to j})_{j\in[k_0-1]}$, where $\wt V_{k_0\to j}$, conditionally on $k_0$, has a mean equal to the right-hand side of~\eqref{eq:wtIprob}. Then, in a second step, we couple $\mathbf{\wt V}_{k_0}^L$ to $\mathbf V_{k_0}^L$. 
	
	For the first step, we use Lemma~\ref{lemma:poibercoupling} to obtain 
	\be 
	\ind_{\cE_0}\Pk{\wt{\mathbf I}^L_{k_0}\neq \wt{\mathbf V}^L_{k_0}}
	\leq \ind_{\cE_0}\sum_{j=1}^{k_0-1}\Big(\frac{\beta p}{2\eps k_0}\Big(\frac jn\Big)^{(1-p)/(2\eps)}\Big)^2\leq\ind_{\cE_0} \frac{\beta^2p^2}{4\eps^2 k_0^2}  \sum_{j=1}^{k_0-1} \Big(\frac{j}{k_0}\Big)^{(1-p)/\eps}.
	\ee
	As $k_0$ diverges due to the lower bound in the event $\cE_0$, we first observe that 
	\be 
	\frac{1}{k_0}  \sum_{j=1}^{k_0-1} \Big(\frac{j}{k_0}\Big)^{(1-p)/\eps}=\int_0^1 x^{(1-p)/\eps}\,\dd x+o(1),
	\ee 
	so that we arrive at the bound
	\be 
	\ind_{\cE_0}\Pk{\wt{\mathbf I}^L_{k_0}\neq \wt{\mathbf V}^L_{k_0}}=\cO\Big(\frac{\log\log n}{n}\Big),
	\ee 
	by again using the lower bound on $k_0$ in the event $\cE_0$.
	In the second step, we use Lemma~\ref{lemma:coupling poisson} to obtain, with $\lambda_j$ as in~\eqref{eq:lambdal},
	\begin{equation}\begin{aligned}
			\ind_{\cE_0}\P{\wt{\mathbf V}^L_{k_0}\neq \mathbf V^L_{k_0}\,\Big|\, a_0,k_0}&\leq \ind_{\cE_0}\sum_{j=1}^{k_0-1}\Big|\lambda_j-\frac{\beta p}{2\eps k_0}\Big(\frac jn\Big)^{(1-p)/(2\eps)}\Big|.
	\end{aligned} \end{equation} 
	Before we bound this sum, we first note that $a_0\leq k_0/n\leq a_0+1/n$ holds on the event $\cE_0$, so that 
	\be 
	\lambda_{k_0-1}=\int_{(k_0-2)/n}^{(k_0-1)/n}\frac{\beta p}{2\eps a_0}x^{(1-p)/(2\eps)}\,\dd x+\int_{(k_0-1)/n}^{a_0}\frac{\beta p}{2\eps a_0}x^{(1-p)/(2\eps)}\,\dd x.
	\ee 
	Since $a_0\leq 1$ almost surely, we can bound the second integral from above by
	\be 
	\frac{\beta p}{2\eps a_0 }\Big(a_0-\frac{k_0-1}{n}\Big)\leq \frac{\beta p}{2\eps a_0n}=\cO\Big(\frac{\log\log n}{n}\Big), 
	\ee 
	where we use the bounds on $k_0$ and $a_0$ in the event $\cE_0$ in the final two steps as well. We thus obtain 
	\be \ba 
	\ind_{\cE_0}\sum_{j=1}^{k_0-1}\Big|\lambda_j-\frac{\beta p}{2\eps k_0}\Big(\frac jn\Big)^{(1-p)/(2\eps)}\Big|\leq {}& \ind_{\cE_0}\frac{\beta p}{2\eps}\sum_{j=1}^{k_0-1}\Bigg|\int_{(j-1)/n}^{j/n} \!\!\!\!\!\!\!\!\!\!a_0^{-1}x^{(1-p)/(2\eps)}\,\dd x-\frac{1}{ k_0}\Big(\frac jn\Big)^{(1-p)/(2\eps)}\Bigg|\\ 
	&+\cO\Big(\frac{\log\log n}{n}\Big).
	\ea\ee 
	Now, for each $j\in[k_0-1]$,
	\be \ba 
	\Bigg|{}&\int_{(j-1)/n}^{j/n} \!\!\!\!\!\!a_0^{-1}x^{(1-p)/(2\eps)}\,\dd x-\frac{1}{ k_0}\Big(\frac jn\Big)^{(1-p)/(2\eps)}\Bigg|\\ 
	&\leq \Big|\frac{1}{a_0}-\frac{n}{k_0}\Big|\int_{(j-1)/n}^{j/n} x^{(1-p)/(2\eps)}\,\dd x +\frac{n}{k_0}\int_{(j-1)/n}^{j/n} \Big|\Big(\frac jn\Big)^{(1-p)/(2\eps)}-x^{(1-p)/(2\eps)}\Big|\,\dd x\\ 
	&=\Big|\frac{1}{a_0}-\frac{n}{k_0}\Big|\frac1n+\frac{1}{k_0}\Big[\Big(\frac jn\Big)^{(1-p)/(2\eps)}-\Big(\frac{j-1}{n}\Big)^{(1-p)/(2\eps)}\Big],
	\ea \ee 
	where the final step follows from the fact that the integrand in the first integral is bounded from above by $1$, and the integrand in the second integral is maximised for $x=(j-1)/n$. Summing over $j$ and using the bounds on $k_0$ and $a_0$ in the event $\cE_0$ thus yields
	\be 
	\ind_{\cE_0}\sum_{j=1}^{k_0-1}\Big|\lambda_j-\frac{\beta p}{2\eps k_0}\Big(\frac jn\Big)^{(1-p)/(2\eps)}\Big|\leq \ind_{\cE_0}\frac{\beta p}{2\eps}\Big|\frac{1}{a_0}-\frac{n}{k_0}\Big|+ \ind_{\cE_0}\frac{\beta p}{2\eps}\frac{1}{k_0}=\cO\Big(\frac{\log\log n}{n}\Big).
	\ee 
	Together with the first step, we obtain that we can couple $\mathbf{\wt I}_{k_0}^L$ with $\mathbf V_{k_0}^L$ such that
	\be 
	\ind_{\cE_0}\Pk{\mathbf{\wt I}_{k_0}^L\neq \mathbf{V}_{k_0}^L}=\cO\Big(\frac{\log\log n}{n}\Big). 
	\ee 
	Combined with the bounds in the first two parts of the proof, we obtain the desired result and conclude the proof.
\end{proof}

\subsubsection{\texorpdfstring{$R$-children of the root $0$ and $R$-neighbours of $k_0$ in $G_n$}{}}\label{sec:rootR}

The $R$-neighbours of the root $k_0$ of $G_n$ can be encoded by a Bernoulli point process $\mathbf I^R_{k_0}=(I_{j\to k_0})_{j=k_0+1}^n$, where $I_{j\to k_0}$ equals one if $j$ sends an edge to $k_0$ in $G_n$. 
Similarly, the $R$-children of the root $0$ of $\cT$ can be encoded by a Poisson point process \lyu{with intensity $\lambda^+$ on the interval $(a_0,1]$ defined in~\eqref{eq:lambda}.}
We couple the two processes by discretising \lyu{the Poisson process} as follows:
for all $j\in\{k_0+1,\ldots,n\}$, define $V_{k_0\to j}\sim \lyu{\text{Po}}(\lambda_j)$ where 
\begin{equation}\begin{aligned}\label{eq:lambdas}
	\lambda_{k_0+1}&:=\int_{a_0}^{(k_0+1)/n}\frac{\beta p}{2\eps}x^{(1-p)/(2\eps)-1}\, \mathrm{d}x \quad &&\text{ and}\\  \lambda_j&:=\int_{(j-1)/n}^{j/n}\frac{\beta p}{2\eps}x^{(1-p)/(2\eps)-1}\, \mathrm{d}x \quad &&\text{ if } j\in \{k_0+2,\ldots, n\}.
	\end{aligned}\end{equation}
	Then, we write $\mathbf V^R_{k_0}=(V_{j\to k_0})_{j=k_0+1}^n$. Moreover, recall the events \lyu{$\cH_{1,0}$ from~\eqref{eq:H1}, $\cR_0$ from~\eqref{eq:rn} and $\cE_0 = \cR_0\cap \cH_{1,0}$ from Lemma~\ref{lemma:Lroot}. Note that $\tfrac{k_0+1}{n}\geq a_0$ on the event $\cR_0$ (so $\lambda_{k_0+1}$ is non-negative). The following lemma shows that $\mathbf I^R_{k_0}$ and $\mathbf V^R_{k_0}$ can be coupled successfully as long as $\cE_0$ holds.}
	
	\begin{lemma}\label{lemma:1Rcouple}
Consider the Bernoulli point process $\mathbf{I}^R_{k_0}:=(I_{j\to k_0})_{j=k_0+1}^n$ and the discretised Poisson process $\mathbf{V}^R_{k_0}:=(V_{j\to k_0})_{j=k_0+1}^n$ with means given in~\eqref{eq:lambdas}. \bas{There exists a coupling of $\mathbf{I}^R_{k_0}$ and $\mathbf{V}^R_{k_0}$ such that 
	\be 
	\P{\{\mathbf{I}^R_{k_0}\neq \mathbf{V}^R_{k_0}\}\cap \cE_0\mid k_0}=\cO\Big(\frac{\log\log\log n}{\log n}\Big),
	\ee 
\lyu{where the constant in the $\cO(\cdot)$ is absolute. Equivalently, upon the event $\cE_0$ and conditionally on $k_0$, the total variation distance between $\mathbf I_{k_0}^R$ and $\mathbf V_{k_0}^R$ is of order $\cO(\log\log\log n/\log n)$.}}
\end{lemma} 

\begin{remark}\label{rem:rootR}
When the coupling \lyu{of the roots in Lemma~\ref{lemma:rootcouple} and the coupling} of the Bernoulli process and the (discretised) Poisson point process \lyu{in Lemma~\ref{lemma:1Rcouple}} is successful, we immediately obtain that $\theta_0^R=  \tau_0^R$, that is, the number of $R$-neighbours of the root $k_0$ in $G_n$ and the number of $R$-children of the root $0$ in $\cT$ are the same. Moreover, for each $R$-child $(0,i)$ in $\cT$ (with $i\in [\tau_0^L+1, \tau_0^L+\tau_0^R]$) we have $|a_{0,i}-\tfrac{k_{0,i}}{n}|\leq \tfrac{\lyu{2}}{n}$. 
\end{remark}

\begin{proof}[Proof of Lemma~\ref{lemma:1Rcouple}] 
\lyu{As in the proof of Lemma~\ref{lemma:Lroot}, based on the alternative definition of total variation distance in~\eqref{eq:dTV}, we show the stronger statement
\begin{equation}\label{eq:conddtvboundR}
	\ind_{\cE_0}\sum_{x}\big|\Pk{\mathbf I_{k_0}^R=x\mid k_0}-\Pk{\mathbf V_{k_0}^R=x\mid k_0}\big| =\cO\Big(\frac{\log\log\log n}{\log n}\Big),
\end{equation}
where the sum is done over all sequences $x$ of $n-k_0$ non-negative integers. Recall the measure $\mathbb P_{k_0} = \mathbb P(\cdot \mid k_0)$.}
We define $(\wt I_{j,k_0})_{j=k_0+1}^n$ to be a sequence of independent Bernoulli random variables where  $\wt I_{j,k_0}$ has success probability $\min\{\beta/(2\eps j),1\}$. Furthermore, we set $\mathbf{\wt I}_{k_0}^R=(\wt I_{j,k_0}\wt \nu_{j,n})_{j=k_0+1}^n $ where $(\wt \nu_{j,n})_{j=k_0+1}^{n}$ are Bernoulli random variables, independent of \lyu{the previous variables and between themselves}, with success probability as in~\eqref{eq:wtnu}. Then, we split the left-hand side of~\eqref{eq:conddtvboundR} into three parts as follows:
	\begin{align}
	\ind_{\cE_0}\!\!\!\!{}&\sum_{x \in \N_0^{n-k_0}}\!\!\!\!\big|\Pk{\mathbf I_{k_0}^R=x}-\P{\mathbf V_{k_0}^R=x}\big|\nonumber\\ 
	\leq{}& \ind_{\cE_0}\!\!\!\!\sum_{x \in \{0,1\}^{n-k_0}}\!\!\!\!\Big|\Pk{\mathbf I_{k_0}^R=x}-\Pk{(\wt I_{j,k_0}\nu_{j,n})_{j=k_0+1}^n=x}\Big|\label{eq:dtv1R}\\ 
	&+\ind_{\cE_0}\!\!\!\!\sum_{x \in \{0,1\}^{n-k_0}}\!\!\!\!\Big|\Pk{(\wt I_{j,k_0}\nu_{j,n})_{j=k_0+1}^n=x}-\Pk{\mathbf{ \wt I}_{k_0}^R=x}\Big|\label{eq:dtv2R}\\ 
	&+\ind_{\cE_0}\!\!\!\!\sum_{x \in \N_0^{n-k_0}}\!\!\!\!\Big|\Pk{\mathbf{\wt I}_{k_0}^R=x}-\Pk{\mathbf{ V}_{k_0}^R=x}\Big|.\label{eq:dtv3R}
	\end{align} 
    We split the proof into three steps, each of which bounds one of the above terms. 

    \paragraph{\underline{Step 1}:} Bounding~\eqref{eq:dtv1R}. We observe that $(\mathbf{I}^R_{k_0})=(I_{j\to k_0})_{j=k_0+1}^n\overset d=(I_{j,k_0}\nu_{j,n})_{j=k_0+1}^n$ where $(I_{j,k_0})_{j=k_0+1}^n$ is a sequence of Bernoulli random variables \lyu{independent conditionally on $k_0, |V_{j-1}|$ and} such that
	\be 
	\mathbb P_{k_0}(I_{j,k_0}=1\mid |V_{j-1}|)=\min\Big\{\frac{\beta}{|V_{j-1}|},1\Big\}.
	\ee 
	Then, \lyu{after using a standard Bernoulli coupling to couple $(I_{j\to k_0})_{j=k_0+1}^n$ and $(\wt I_{j\to k_0})_{j=k_0+1}^n$, we obtain that}
	\be 
	\Pk{\bigcup_{j=k_0+1}^n\{I_{j,k_0}\neq \wt I_{j,k_0}\}}\leq \sum_{j=k_0+1}^n\bigg|\min\Big\{\frac{\beta}{2\eps j},1\Big\}-\Ek{\min\Big\{\frac{\beta}{|V_{j-1}|},1\Big\}}\bigg|.
	\ee 
	By Lemma~\ref{lemma:condconc}, for some constant $C_1>0$, we have the bounds
	\be \ba
	\Ek{\ind_{\cQ_{j-1}} \min\Big\{\frac{\beta}{|V_{j-1}|},1\Big\}}
	&\leq\Ek{ \min\Big\{\frac{\beta}{|V_{j-1}|},1\Big\}}\\
	&\leq \Ek{\ind_{\cQ_{k_0-1}} \min\Big\{\frac{\beta}{|V_{j-1}|},1\Big\}}+\P{\cQ_{j-1}^c}\\ 
	&\leq \min\Big\{\frac{\beta}{2\eps j},1\Big\}+C_1j^{-4/3}.
	\ea \ee 
	Thus, for some constant $C>0$, 
	\be\ba  \label{eq:IcoupleR}
	\Pk{\exists j\in [n]\setminus[k_0]: I_{j\to k_0}\neq \wt I_{j,k_0}\nu_{j,n}}&\leq \Pk{\exists j\in [n]\setminus[k_0]: I_{j,k_0}\neq \wt I_{j,k_0}}\\ 
	&\leq \lyu{\sum_{j=k_0+1}^n C_1 j^{-4/3}\leq} Ck_0^{-1/3},
	\ea\ee 
	where the second inequality follows from a union bound over all $j\in [n]\setminus[k_0]$. This completes the bound on~\eqref{eq:dtv1R}.
	
	\paragraph{\underline{Step 2}:} Bounding~\eqref{eq:dtv2R}. For $\delta>0$, we define the event
	\be 
	\cB_n:=\bigg\{\sum_{j=k_0+1}^n\wt I_{j,k_0}\leq \delta\log n\bigg\}.
	\ee 
	Note that $\cB_n$ holds with high probability irrespective of the choice of $\delta$: indeed, by Markov's inequality,
	\be \label{eq:bnbound}
	\ind_{\cE_0}\Pk{\cB_n^c}=\ind_{\cE_0}\cO\Big(\frac{\log(n/k_0)}{\log n}\Big)=\cO\Big(\frac{\log\log\log n}{\log n}\Big).
	\ee 
	We then write
	\be\ba \label{eq:dtvcondR}
	\ind_{\cE_0}\sum_{x\in \{0,1\}^{n-k_0}}\!\!\!\!\Big|\mathbb P_{k_0}\Big((\wt I_{j,k_0}{}&\nu_{j,n})_{j=k_0+1}^n=x\Big)-\Pk{(\wt I_{j,k_0}\wt \nu_{j,n})_{j=k_0+1}^n=x}\Big|\\ 
	\leq \ind_{\cE_0}\mathbb E_{k_0}\Bigg[\sum_{x\in\{0,1\}^{n-k_0} }\!\!\!\!\!\Big|{}&\mathbb P_{k_0}((\wt I_{j,k_0}\nu_{j,n})_{j=k_0+1}^n=x\mid (\wt I_{j,k_0})_{j=k_0+1}^n)\\ 
	&-\mathbb P_{k_0}((\wt I_{j,k_0}\wt \nu_{j,n})_{j=k_0+1}^n=x\mid (\wt I_{j,k_0})_{j=k_0+1}^n)\Big| \Bigg].
	\ea\ee 
	Now, define
	\be 
	M_n:=\sum_{j=k_0+1}^n \wt I_{j,k_0}. 
	\ee 
	Conditionally on the sequence $(\wt I_{j,k_0})_{j=k_0+1}^n$, let $(j_i)_{i\in [M_n]}$ denote the sequence of indices in increasing order such that $\wt I_{j_i,k_0}=1$ for all $i\in[M_n]$. Then, given $x=(x_j)_{j=k_0+1}^n$ and $x'=(x_{j_i})_{i\in[M_n]}$, \lyu{conditionally on $k_0$ and $(\wt I_{j,k_0})_{j=k_0+1}^n$,} we observe that the events 
	\be 
	\Big\{(\wt I_{j,k_0}\nu_{j,n})_{j=k_0+1}^n=x\Big\} \quad\text{and}\quad \Big\{(\nu_{j_i,n})_{i\in[M_n]}=x'\Big\}
	\ee 
	have the same probability as long as $x_\ell=0$ for all $\ell\in \{k_0+1,\ldots, n\}\setminus \{j_1,\ldots, j_{M_n}\}$. If, instead, $x_\ell=1$ for some $\ell\in \{k_0+1,\ldots, n\}\setminus \{j_1,\ldots, j_{M_n}\}$, then the event on the left-hand side has probability zero. The same holds when we replace $\nu_{j,n}$ by $\wt \nu_{j,n}$. As a result, using this in~\eqref{eq:dtvcond}, we obtain 
	\be \ba 
	\ind_{\cE_0}\mathbb E_{k_0}\Bigg[\sum_{x\in\{0,1\}^{M_n} }\!\!\!\!\!\Big|{}&\mathbb P_{k_0}((\nu_{j_i,n})_{i\in[M_n]}=x\mid (\wt I_{j,k_0})_{j=k_0+1}^n)\\ 
	&-\mathbb P_{k_0}((\wt \nu_{j_i,n})_{i\in[M_n]}=x\mid (\wt I_{j,k_0})_{j=k_0+1}^n)\Big|\Bigg]\\ 
	\leq \ind_{\cE_0}\mathbb E_{k_0}\Bigg[\ind_{\cB_n}\sum_{x\in\{0,1\}^{M_n} }\!\!\!\!\!\Big|{}&\mathbb P_{k_0}((\nu_{j_i,n})_{i\in[M_n]}=x\mid (\wt I_{j,k_0})_{j=k_0+1}^n)\\ 
	&-\mathbb P_{k_0}((\wt \nu_{j_i,n})_{i\in[M_n]}=x\mid (\wt I_{j,k_0})_{j=k_0+1}^n)\Big|\Bigg]+\ind_{\cE_0}\Pk{\cB_n^c}.
	\ea\ee 
	On the events $\cB_n$ and $\cE_0$, it holds that $M_n\leq \delta\log n$ and that $j_i\geq n/\log\log n$ for all $i\in M_n$, respectively. As a result, by summing over the events $\{M_n=r\}$ with $r=1,\ldots, \lfloor \delta_2\log n\rfloor$, for all such $r$, we can apply Corollary~\ref{cor:condsurv} (and Remark~\ref{rem:condsurv}, in particular) with $j_0=n/\log\log n$ to bound the latter absolute values from above. Together with~\eqref{eq:bnbound}, this yields the upper bound
	\be 
	\ind_{\cE_0}\Bigg[\cO\Bigg(\sum_{r=1}^{\lfloor \delta\log n\rfloor} \frac{4^rr(\log\log n)^{1/6}}{n^{1/6}}\Bigg)+\cO\Big(\frac{\log\log\log n}{\log n}\Big)\Bigg]=\cO\Big(\frac{\log\log\log n}{\log n}\Big), 
	\ee 
	where we  choose $\delta$ sufficiently small. This concludes the second part of the proof.
	
    \paragraph{\underline{Step 3}:} Bounding~\eqref{eq:dtv3R}.	We split the bound of~\eqref{eq:dtv3R} in two further steps. Define $\mathbf{\wt V}_{k_0}^R=(\wt V_{j\to k_0})_{j=k_0+1}^n $ to be a sequence of independent Poisson random variables where $\wt V_{j\to k_0}$ has mean 
	\be \label{eq:Poimean} 
	\frac{\beta p}{2\eps j}\bigg(\frac{j}{n}\bigg)^{(1-p)/(2\eps)} \quad\text{for all}\quad j\in \{k_0+1,\ldots, n\}. 
	\ee 
	For the first step, we use Lemma~\ref{lemma:poibercoupling} to obtain 
	\be 
	\ind_{\cE_0}\Pk{\wt{\mathbf I}^R_{k_0}\neq \wt{\mathbf V}^R_{k_0}}
	\leq \ind_{\cE_0}\sum_{j=k_0+1}^n\Big(\frac{\beta p}{2\eps j}\Big(\frac jn\Big)^{(1-p)/(2\eps)}\Big)^2\leq\ind_{\cE_0} \frac{\beta^2p^2}{4\eps^2 k_0^2}  \sum_{j=k_0+1}^n \Big(\frac{j}{n}\Big)^{(1-p)/\eps}.
	\ee
	By the bound on $k_0$ in the event $\cE_0$, we know that $k_0$ diverges with $n$, so 
	\be 
	\frac{1}{k_0}\sum_{j=k_0+1}^n \Big(\frac{j}{n}\Big)^{(1-p)/\eps}\leq \log\log n\int_0^1 x^{(1-p)/\eps}\,\dd x.
	\ee 
	As a result,
	\be 
	\ind_{\cE_0}\Pk{\wt{\mathbf I}^R_{k_0}\neq \wt{\mathbf V}^R_{k_0}}= \cO((\log\log n)^2/n).
	\ee 
	In the second step, we  recall $\lambda_j$ from~\eqref{eq:lambdas} and use Lemma~\ref{lemma:coupling poisson} to obtain 
	\begin{equation}\begin{aligned}\label{eq:vneqbound}
			\ind_{\cE_0}\P{\wt{\mathbf V}^R_{k_0}\neq \mathbf V^R_{k_0}\,\Big|\, a_0,k_0}&\leq \ind_{\cE_0}\sum_{j=k_0+1}^n\Big|\lambda_j-\frac{\beta p}{2\eps j}\Big(\frac jn\Big)^{(1-p)/(2\eps)}\Big|.
	\end{aligned} \end{equation} 
	Consider any $j\in [k_0+2,n]$. To begin with, 
	\be\ba 
	\Big|\lambda_j-\frac{\beta p}{2\eps j}\Big(\frac jn\Big)^{(1-p)/(2\eps)}\Big|&\leq \frac{\beta p}{2\eps}\int_{(j-1)/n}^{j/n}\Big| x^{(1-p)/(2\eps)-1}-\Big(\frac jn\Big)^{(1-p)/(2\eps)-1}\Big|\,\dd x\\ 
	&\leq \frac{\beta p}{2\eps} \frac1n \Big|\Big(\frac{j-1}{n}\Big)^{(1-p)/(2\eps)-1}-\Big(\frac jn \Big)^{(1-p)/(2\eps)-1}\Big|,
	\ea\ee
where the latter upper bound follows from the fact that the absolute value in the integral is maximised when $x=(j-1)/n$.
\lyu{In each of the cases $\eps\ge 1/6$ and $\eps\le 1/6$ (where $\eps = 1/6$ is the root of $1-p=2\eps$), telescopic summation shows that~\eqref{eq:vneqbound} is at most
\begin{equation}\label{eq:sum_4.5}
\ind_{\cE_0}\Big(\Big|\lambda_{k_0+1} - \frac{\beta p}{2\eps (k_0+1)}\Big(\frac{k_0+1}{n}\Big)^{(1-p)/(2\eps)}\Big|+\frac{\beta p}{2\eps n}\Big(\Big(\frac{k_0+1}{n}\Big)^{(1-p)/(2\eps)-1}+1\Big)\Big).    
\end{equation}
Dominating the remaining absolute value by the sum of the two terms and using~\eqref{eq:lambdas}, we obtain the last expression (and thus~\eqref{eq:vneqbound} as well) is of order $n^{-1+o(1)}$. 
Together with the first two parts, this concludes
the proof.}
\end{proof}

\subsection{Summary}
We are now ready to prove Lemma~\ref{lemma:1couple}. 

\begin{proof}[Proof of Lemma~\ref{lemma:1couple}]
Part of the proof has already been demonstrated in~\eqref{eq:1couplebound},\eqref{eq:missed}, \eqref{eq:h10h11bound} and~\eqref{eq:taubound}. Furthermore, we observe that 
\begin{equation}
	\cH_{1,2}\supseteq\cR_0\cap \{\mathbf I^L_{k_0}=\mathbf V^L_{k_0}\}\cap \{\mathbf{I}^R_{k_0}=\mathbf{V}^R_{k_0}\},
\end{equation} 
where we recall the events $\cR_0$, $\{\mathbf I^L_{k_0}=\mathbf V^L_{k_0}\}$ and $\{\mathbf{I}^R_{k_0}=\mathbf{V}^R_{k_0}\}$ from \eqref{eq:rn} and Lemmas~\ref{lemma:Lroot} and~\ref{lemma:1Rcouple}, respectively. These events ensure that the immediate neighbourhoods of $k_0$ in $G_n$ and $0$ in $\cT$ are isomorphic, since the number of $L$- and $R$-neighbours are the same. 
Thus, from Lemmas~\ref{lemma:rootcouple},~\ref{lemma:Lroot} and~\ref{lemma:1Rcouple} we obtain that
\begin{equation}\begin{aligned}
		\P{\cH_{1,2}^c \cap \cH_{1,0}}\leq\; &\P{\cR_0^c}+\P{\{\mathbf{I}^L_{k_0}\neq \mathbf{V}^L_{k_0}\}\cap \cR_0\cap\cH_{1,0}}\\
		&+\P{\{\mathbf{I}^R_{k_0}\neq \mathbf{V}^R_{k_0}\}\cap \cR_0\cap\cH_{1,0}}=\cO\Big(\frac{\lyu{\log\log\log n}}{\log n}\Big).
\end{aligned}\end{equation}
As a result, by combining this with the steps from~\eqref{eq:1couplebound},~\eqref{eq:missed}, \eqref{eq:h10h11bound} and~\eqref{eq:taubound}, we finally obtain that
\begin{equation}\begin{aligned}
		\P{\Big(\bigcap_{i=0}^3 \cH_{1,i}\Big)^c}&\leq\P{\cH_{1,0}^c}+\P{\cH_{1,1}^c\lyu{\cap \cH_{1,0}}}+\P{\cH_{1,2}^c\cap \cH_{1,0}}+\P{\cH_{1,3}^c}\\ 
		& = \cO\Big(\frac{1}{(\log \log n)^{p/(2\eps)}}\Big),
\end{aligned} \end{equation} 
which finishes the proof.
\end{proof} 

\section{Continuing the coupling}\label{sec:lwlcont}

We continue the construction of the coupling described in the previous sections by providing a coupling of the $r$-neighbourhoods of $k_0$ in $G_n$ and $0$ in $\cT$ that is successful with high probability for any $r\geq 2$. 
Since the number of vertices in $G_n$ is ``large'' w.h.p., exploring a finite neighbourhood of $G_n$ changes the distribution of the remaining graph ``very little''. 
As a result, the coupling of the $r$-neighbourhoods follows a similar approach as the coupling of the $1$-neighbourhood in the previous section. 
Modifications are necessary, though, which makes the proofs somewhat more technical and involved.

As in~\eqref{eq:H1}, fix $r\in\N$ and for all $q\in[\bas{r}]$, define
\begin{equation}\begin{aligned}\label{eq:Hq}
	\cH_{q,1}&:=\{\forall v\in \partial B_{q-1}(\cT,0): \min_{i\in[ \tau_{v}^L]}  a_{v,i}\geq (\log\log n)^{-(q+1)}\},\\
	\cH_{q,2}&:=\{(B_q(G_n,k_0),k_0)\cong (B_q(\cT,0),0)\}\cap \{\forall v\in \partial B_q(\cT,0), |a_{v}-\tfrac{k_{v}}{n}|\leq \tfrac{q+1}{n} \}, \\
	\cH_{q,3}&:=\Big\{\forall v\in V(B_{q-1}(\cT,0)):  \tau_{v}^L+  \tau_{v}^R<(\log \log n)^{1/r}\Big\}.
	\end{aligned} \end{equation} 
\lyu{Note that the dependence of $\cH_{q,3}$ on $r$ ensures that the ball $B_r(\cT,0)$ contains at most $\log\log n$ vertices, which will turn out useful in the sequel.}
	
	Also note that for $r=1$ and \lyu{$q=1$}, these events coincide with the last three events in~\eqref{eq:H1}. When $r>1$, only minor modifications in~\eqref{eq:taubound} are required to show that, when \lyu{$q=1$}, $\cH_{1,3}$ holds with probability at least $1-\cO((\log\log n)^{-p/(2\eps)})$. We now state the main result of this section.
	
	\begin{proposition}\label{prop:qcouple}
Fix $\eps\in(0,\tfrac{1}{2}]$, $r\in \N\backslash \{1\}$ and $q\in [r-1]$.
Recall the events in~\eqref{eq:Hq} and assume that, for all $n$ sufficiently large, there is a coupling of $(G_n,k_0)$ and $(\cT,0)$ such that 
\begin{equation}
	\P{\Big(\bigcap_{i=1}^3 \cH_{q,i}\Big)^c}\leq \lyu{3\beta}(\log\log n)^{(q-1)/r-p/(2\eps)}.
\end{equation} 
Then, there is a coupling of $(G_n,k_0)$ and $(\cT,0)$ such that 
\begin{equation}
	\P{\Big(\bigcap_{i=1}^3 \cH_{q+1,i}\Big)^c}\leq \lyu{3\beta}(\log\log n)^{q/r-p/(2\eps)}.
\end{equation} 
\end{proposition}

Proposition~\ref{prop:qcouple} combined with Lemma~\ref{lemma:1couple} has the following result as an immediate corollary.
\begin{corollary}\label{cor:rcouple}
Fix $\eps\in(0,\tfrac{1}{2}]$ and set $p:=\tfrac{1}{2}+\eps$. For every fixed $r\in\N$, there exists a coupling of $(G_n,k_0)$ and $(\cT,0)$ such that
\begin{equation}
	\P{\Big(\bigcap_{i=1}^3 \cH_{r,i}\Big)^c} =  \cO\big((\log\log n)^{-(1-p)/(2\eps)-1/r}\big).
\end{equation} 
\end{corollary}
To prove Proposition~\ref{prop:qcouple}, we use that 
\begin{equation}\begin{aligned}\label{eq:inducbound}
	\P{\Big(\bigcap_{i=1}^3 \cH_{q+1,i}\Big)^c}\leq{}&\P{\Big(\bigcap_{i=1}^3 \cH_{q+1,i}\Big)^c\cap \bigcap_{i=1}^3 \cH_{q,i}}+\P{\Big(\bigcap_{i=1}^3 \cH_{q,i}\Big)^c}\\
	\leq{}& \P{\cH_{q+1,3}^c\cap \cH_{q,3}}+\P{\cH_{q+1,1}^c\cap \cH_{q+1,3}\cap \cH_{q,1}}\\
	&+\P{\cH_{q+1,2}^c\cap \cH_{q+1,1}\cap \cH_{q+1,3}\cap \bigcap_{i=1}^3 \cH_{q,i}}+\P{\Big(\bigcap_{i=1}^3 \cH_{q,i}\Big)^c}.
	\end{aligned}\end{equation} 
	We estimate the terms one by one. Let $(P_i)_{i\ge 1}$ be i.i.d.\ $\lyu{\text{Po}}(\tfrac{\beta p}{1-p})$ random variables. Recall that a vertex $v$ in $\cT$ has $\tau^L_v+\tau^R_v$ many children and, using~\eqref{eq:meanoffspring}, one can readily check that, for any type $a_v\in(0,1]$, $\tau^L_v+\tau^R_v$ is stochastically dominated by $P_1$. 
	Since for every $q\in [r-1]$, conditionally on $\cH_{q,3}$, we have that $|\partial B_{q}(\cT,0)|\leq (\log\log n)^{q/r}$, a union bound and a similar approach as in~\eqref{eq:taubound} with $t=1$ yield
	\begin{equation}\begin{aligned}\label{eq:H3c}
	\P{\cH_{q+1,3}^c\cap \cH_{q,3}}
	&\leq \P{\max_{i\in[ (\log\log n)^{q/r}]}P_i\geq (\log\log n)^{1/r}}\\
	&\le (\log\log n)^{q/r}\e^{-(\log\log n)^{1/r}}.
	\end{aligned}\end{equation}
	\lyu{Now, let us write $\mathbf{V}^L_q$ for the Poisson process with intensity $\lambda_v^-(\dd x)$ as in~\eqref{eq:lambda} where $v$ is the left-most vertex in generation $q-1$. By a similar approach as in~\eqref{eq:h11bound} and a union bound over $|\partial B_q(\cT,0)|\leq (\log\log n)^{q/r}$ vertices,} for every $q\in [r-1]$, we obtain
\begin{equation}\label{eq:H1c}
\begin{split}
\P{\cH_{q+1,1}^c\cap \cH_{q+1,3}\cap \cH_{q,1}}
&\lyu{\leq |\partial B_q(\cT,0)| \mathbb P(\mathbf{V}^L_q\cap [0,1/(\log\log n)^{q+2}]\neq\emptyset\mid \cH_{q+1,3}\cap \cH_{q,1})}\\
&\leq \beta (\log\log n)^{q/r+q-(q+1)p/(2\eps)+\lyu{1-p/(2\eps)}}\\
&\leq \beta (\log\log n)^{q/r-p/(2\eps)},     
\end{split}
\end{equation} 
since $\tfrac{p}{2\eps}\geq 1$. By the induction hypothesis, we know that the last term in~\eqref{eq:inducbound} is at most \\ $C_q(\log\log n)^{(q-1)/r-p/(2\eps)}$ for some constant $C_q>0$, so only the third term on the right-hand side of~\eqref{eq:inducbound} is left to analyse.

To extend the coupling from the $q$-neighbourhoods of $k_0$ and $0$ to their respective $(q+1)$-neighbourhoods, we can assume that a coupling of $B_q(G_n,k_0)$ and $B_q(\cT,0)$ exists such that $\cap_{i=1}^3\cH_{q,i}$ holds \lyu{with sufficiently high probability}. Then, for each $v\in \partial B_q(\cT,0)$, we want to couple its children in $\cT$ to the \lyu{unexplored} neighbours of $k_{v}\in \partial B_q(G_n,k_0)$ in $G_n$ so that the number of children and \lyu{unexplored} neighbours, respectively, are equal, and such that their types and rescaled labels are sufficiently close.

Recall \lyu{from Section~\ref{sec:bfs}} that $k[t]$ is the smallest vertex in $\cA_{t-1}$ with respect to the BFS order, and let, for ease of writing, $k_{v(t)}:=k[t]$. Suppose that $k_{v(t)}\in \partial B_q(G_n,k_0)$. We then define the event
\begin{equation}\begin{aligned}\label{eq:kappat}
	\cK_{t,q}:=\{{}& \theta^L_{v(t)}=\tau^L_{v(t)}, \theta^R_{v(t)}=\tau^R_{v(t)}\}\cap\Big\{ \forall i\in[ \tau_{v(t)}]: \Big|  a_{v(t),i}-\tfrac{k_{v(t),i}}{n}\Big|\leq \tfrac{q+2}{n}, k_{v(t),i}\not  \in \cA_{t-1}\Big\},
	\end{aligned}\end{equation} 
	\lyu{where the event $k_{v(t),i}\not\in \cA_{t-1}$ aims to ensure the absence of short cycles around $k_0$, as such would invalidate the coupling. We also define} the collection of random variables 
	\begin{equation}\label{eq:ct}
\wt \cC_t:=((k_{v},a_{v})_{k_{v}\in \cA_{t-1}\cup \cP_{t-1}}, ((\theta_{w}^L,\tau_{w}^L),(\theta_{w}^R,\tau_{w}^R))_{k_{w}\in \cP_{t-1}}).
\end{equation} 

Furthermore, for every $q\lyu{\ge 1}$, define $\rho[q]$ as
\begin{equation}\label{eq:rhoq}
\rho[q]=
\begin{cases}
	|V(B_{q-1}(G_n,k_0))|+1,&  \mbox{ if } \bas{\partial} B_q(G_n,k_0)\neq \emptyset,\\
	\infty, &\text{ otherwise}.
\end{cases}
\end{equation} 
When $\rho[q] < \infty$, it denotes the time step of the BFS exploration at which the smallest vertex $k[\rho[q]]$ (with respect to the BFS order) in $\partial B_q(G_n,k_0)$ is explored. Also observe that, for all $q\in [r-1]$,
\begin{equation}
(\cA_{\rho[q]-1},\cP_{\rho[q]-1},\cN_{\rho[q]-1})=(\partial B_q(G_n,k_0), V(B_{q-1}(G_n,k_0)), V_n\backslash V(B_q(G_n,k_0))).
\end{equation} 
We observe that 
\begin{equation}\label{eq:hq+12}
\cH_{q+1,2}=\cH_{q,2}\cap \bigcap_{t=\rho[q]}^{\rho[q+1]-1}\cK_{t,q}.
\end{equation} 
As a result, to provide an upper bound for the probability 
\begin{equation}
\P{\cH^c_{q+1,2}\cap \cH_{q+1,1}\cap \cH_{q+1,3}\cap \bigcap_{i=1}^3 \cH_{q,i}}, 
\end{equation} 
we first bound
\begin{equation}
\P{\cK^c_{t,q}\cap \cH_{q+1,1}\cap \cH_{q+1,3}\cap \bigcap_{i=1}^3 \cH_{q,i}}
\end{equation} 
for every $t\in \{\rho[q],\ldots, \rho[q+1]-1\}$. 
As for a fixed $q$ and any $t\in \{\rho[q],\ldots, \rho[q+1]-1\}$, these bounds are \lyu{achieved in the same way}, we focus on the case $t=\rho[q]$. That is, we consider the vertex $k_{v[q]}:=k[\rho[q]]$ \lyu{(with $v[q]$ serving as a convenient shorthand notation for the vertex $v(\rho[q])$ in $\cT$)} and couple the \lyu{unexplored} neighbours of $k_{v[q]}$ in $G_n$ to the children of $v[q]$ in $\cT$. For ease of writing, we also set
\begin{equation}\begin{aligned}\label{eq:kapparhoq}
	&\cK_{\rho[q]}:=\cK_{\rho[q],q}\\
	=\;
	&\{\theta_{v[q]}^L= \tau_{v[q]}^L, \theta_{v[q]}^R= \tau^R_{v[q]}\}\cap \Big\{ \forall i\in[ \tau_{v[q]}]: \Big|  a_{v[q],i}-\tfrac{k_{v[q],i}}{n}\Big|\leq \tfrac{q+2}{n}, k_{\barv[q],i}\not\in \cA_{\rho[q]-1}\Big\},
	\end{aligned}\end{equation}
	and 
	\begin{equation}\begin{aligned}\label{eq:Cq}
	\cC_q:=\wt \cC_{\rho[q]}={}&((k_{v},a_{v})_{k_{v}\in \cA_{\rho[q]-1}\cup \cP_{\rho[q]-1}}, ((\theta_{w}^L,\tau_{w}^L),(\theta_{w}^R,\tau_{w}^R))_{k_{w}\in \cP_{\rho[q]-1}})\\
	={}&((k_{v},a_{v})_{k_{v}\in V(B_q(G_n,k_0))}, ((\theta_{w}^L,\tau_{w}^L),(\theta_{w}^R,\tau_{w}^R))_{k_{w}\in V(B_{q-1}(G_n,k_0))}).
	\end{aligned}\end{equation} 
	\lyu{For simplicity of the exposition, we provide the details for the coupling of the neighbourhoods of $v[q]$ in  $B_{q+1}(\cT,0)$ and of $k_{v[q]}$ in $B_{q+1}(G_n,k_0)$ and only briefly touch upon the coupling for the remaining vertices in $\partial B_q(G_n,k_0)$. 
    Indeed, the coupling for these other vertices hides no additional difficulties but requires heavier notation that we decided to spare. 
    The only meaningful difference with the particular case we consider is that we need to ensure that forward edges are never sent to the already processed vertices in the same level. 
    The probability of this bad event is conveniently bounded from above using the event $\cQ_i$ from~\eqref{eq:An} (where $i$ is the label of the processed vertex), Lemma~\ref{lemma:condconc} and the event $\cH_{q,1}$ (ensuring that $i$ is suitably large).}\\
	
	To couple the neighbours of $k_{v[q]}$ in $G_n$ and $v[q]$ in $\cT$, we assume that the random variables in $\cC_q$ have been coupled such that $\cE_q:=\cap_{i=1}^3 \cH_{q,i}$ holds \lyu{with high probability (and more precisely, satisfies the assumption of Proposition~\eqref{prop:qcouple})}. In the upcoming sections we prove the following result.
	
	\begin{lemma}\label{lemma:kappacouple}
Fix $\eps\in(0,\tfrac{1}{2}]$, $r\in \N \setminus \{1\}$, $q\in [r-1]$, and set $p:=\tfrac{1}{2}+\eps$ and $\alpha := \tfrac{1-p}{p}$. 
There exists a coupling of $(G_n,k_0)$ and $(\cT,0)$ such that 
\begin{equation}
	\P{\cK_{\rho[q]}^c\cap \cE_q}= \cO\Big(\frac{1}{\log n}\Big).
\end{equation} 
\end{lemma}

\subsection{\texorpdfstring{Coupling the $L$-children of $v[q]$ in $\cT$ with the $L$-neighbours of $k_{\barv[q]}$ in $G_n$}{}}\label{sec:Lqcouple}

To prove Lemma~\ref{lemma:kappacouple}, we first couple the $L$-neighbours of $k_{v[q]}$ and the $L$-children of $v[q]$ in this section.  We define 
\begin{equation}\begin{aligned}\label{eq:Lq}
	\cL_{\rho[q]}:=\{{}&\theta_{v[q]}^L=\tau_{v[q]}^L\}\cap\{ \forall i\in[\tau_{v[q]}^L]: |a_{v[q],i}-\tfrac{k_{v[q],i}}{n}|\leq \tfrac{q+2}{n}, k_{\barv[q],i}\in \cN_{\rho[q]-1}\}.
	\end{aligned} \end{equation}
	to be the event that the coupling of the $L$-neighbours of $k_{v[q]}$ and the $L$-children of $v[q]$ is successful. In words, $\cL_{\rho[q]}$ says that the number of $L$-neighbours of $k_{\barv[q]}$ is equal to the number of $L$-children of $\barv[q]$, that their respective rescaled labels and types are close, and that the exploration of the $L$-neighbours of $k_{v[q]}$ does not lead to cycles.
	
	We now define an encoding of the $L$-neighbours of $k_{v[q]}$ in $G_n$ and the $L$-children of $v[q]$ in $\cT$ that allows us to construct a coupling such that $\cL_{\rho[q]}$ holds with high probability. Unlike in the proof of Lemma~\ref{lemma:Lroot}, we cannot ensure that $\tfrac{k_{\barv[q]}-1}{n}\leq a_{\barv[q]}$ holds on the event $\cE_q=\cap_{i=1}^3 \cH_{q,i}$. Instead, we let $N_{v[q]}:=\lceil a_{\barv[q]}n\rceil$ and $k^*_q:=\max\{N_{v[q]},k_{\barv[q]}-1\}$, and use $k^*_q$ to construct \lyu{a discretised Poisson process similar to $\mathbf{V}_{k_0}^L$ defined in~\eqref{eq:lambdal}, which is key in our coupling. More precisely, if}
	$N_{v[q]}\geq k_{\barv[q]}-1$, conditionally on $a_{v[q]}$, we define
	\begin{equation}\begin{aligned}\label{eq:lambdaqL}
	\lambda^{\rho[q]}_{N_{v[q]}}&:=\int_{\bas{(N_{v[q]}-1)}/n}^{a_{v[q]}}\frac{\beta p}{2\eps a_{v[q]}}x^{(1-p)/(2\eps)}\,\dd x,
	\qquad&\text{and,}&\\ \lambda^{\rho[q]}_j&:=\int_{\tfrac{j-1}{n}}^{\tfrac{j}{n}}\frac{\beta p}{2\eps a_{v[q]}}x^{(1-p)/(2\eps)}\,\dd x,\qquad  &\text{for }j\in[ N_{v[q]}-1].&
	\end{aligned}\end{equation}
	If, instead, $N_{v[q]}\leq k_{\barv[q]}-2$, we additionally set $\lambda^{\rho[q]}_j:=0$ for all integers $j\in [N_{v[q]}+1,  k^*_q]$. This defines $\lambda^{\rho[q]}_j$ for all $j\in \lyu{[k^*_q]}$, and allows us to define the discretised Poisson point process $\mathbf V^L_{k_{\barv[q]}}:=(V^L_{k_{v[q]}\to j})_{\lyu{j=1}}^{k^*_q}$ where $V^L_{k_{v[q]}\to j}$ is a Poisson random variable with mean $\lambda^{\rho[q]}_j$ (where a Poisson variable with mean zero is equal to zero almost surely). 
	\lyu{More precisely, $\mathbf V^L_{k_{v[q]}}$ must be seen as} a discretisation of the Poisson point process on $(0,a_{v[q]})$ with \lyu{intensity} $\lambda^-_{v[q]}(\dd x)$ that determines the $L$-children of $v[q]$ (see Definition~\ref{def:wll}).
	
	Similarly, we define the  Bernoulli point process $\mathbf I^L_{k_{\barv[q]}}:=(I^L_{k_{v[q]}\to j})_{j=1}^{k^*_q}$ as follows. For all $j\in [k_{\barv[q]}-1]$, $I^L_{k_{v[q]}\to j}$ equals one if $k_{v[q]}$ connects to $j$ by an edge in $G_n$ and $j\in \cN_{\rho[q]-1}\cup \cA_{\rho[q]-1}$, and equals zero otherwise. For $j\in [k_{\barv[q]},k^*_q]$ (that is, when $N_{v[q]}> k_{\barv[q]}-1$), we additionally set $I^L_{k_{v[q]}\to j}:=0$.
	
	Our aim is to couple the Bernoulli point process $\mathbf I^L_{k_{v[q]}}$ and the discretised Poisson point process $\mathbf V^L_{k_{v[q]}}$. To do this, we show the following lemma, which is similar in nature to Lemma~\ref{lemma:Lroot}.
	
	\begin{lemma}\label{lemma:qLcouple}
Recall that $\cE_q=\cap_{i=1}^3\cH_{q,i}$, where the events $(\cH_{q,i})_{i\in[3]}$ are defined in~\eqref{eq:Hq}. Consider the processes $\mathbf I^L_{k_{\barv[q]}}$ and $\mathbf V^L_{k_{\barv[q]}}$. There exists a coupling such that
\begin{equation}
	\P{\big(\{\mathbf I^L_{k_{v[q]}}= \mathbf V^L_{k_{v[q]}}\}\cap \{\forall j\in \cA_{\rho[q]-1}\cap [k_q^*]: I^L_{k_{v[q]}\to j}=0\}\big)^c\cap \cE_q\, \Big|\, \mathcal C_q} = \cO\Big(\frac{1}{\log n}\Big), 
\end{equation} 
\lyu{where the constant in the $\cO(\cdot)$ is absolute. In particular, upon the event $\cE_q$ and conditionally on the random variables in $\cC_q$, the total variation distance between $\mathbf I^L_{k_{v[q]}}$ and $\mathbf V^L_{k_{v[q]}}$ is of order $\cO(1/\log n)$.}
\end{lemma} 

\begin{remark}\label{rem:qL}
When the coupling is successful, in the sense that the event 
\begin{equation}
	\{\mathbf I^L_{k_{v[q]}}= \mathbf V^L_{k_{v[q]}}\}\cap \{\forall j\in \cA_{\rho[q]-1}\cap [k_q^*]: I^L_{k_{v[q]}\to j}=0\}
\end{equation} 
holds, it directly follows that the event $\cL_{\rho[q]}$ holds as well. As a result, 
\begin{equation}
	\P{\cL_{\rho[q]}^c\cap \cE_q}=\cO\left(\frac{1}{\log n}\right).
\end{equation} 
\end{remark}

\begin{proof}[Proof of Lemma~\ref{lemma:qLcouple}]
\lyu{As before, we show the more general statements
\be\label{eq:5.4.1}
	\ind_{\cE_q}\sum_{x}\Big|\P{\mathbf I_{k_{v[q]}}^L=x\,\Big|\,  \cC_q}-\P{\mathbf V_{k_{v[q]}}^L=x\,\Big|\, \cC_q}\Big| =\cO\Big(\frac{1}{\log n}\Big),
	\ee
where the sum ranges over all sequences $x$ of $k_q^*$ non-negative integers, and
	\be\label{eq:5.4.2} 
	\ind_{\cE_q}\P{\{\exists j\in \cA_{\rho[q]-1}\cap [k^*_q]: I^L_{k_{v[q]}\to j}=1\}\cap \{\mathbf I^L_{k_{v[q]}}=\mathbf V^L_{k_{v[q]}}\}\cap \cE_q\,\Big|\, \cC_q}=\cO\Big(\frac{1}{\log n}\Big).
	\ee}
\lyu{We start by showing~\eqref{eq:5.4.1}. Recall the probability measure $\Pq{\cdot}:=\P{\cdot\,|\, \cC_q}$ conditioned on the set of random variables $\cC_q$ and the associated expectation $\Eq{\cdot}$.} Let $(\wt I_{j,k_{v[q]}})_{j\in[k^*_q]}$ be a sequence of independent Bernoulli random variables, where  $\wt I_{j,k_{v[q]}}$ has success probability $\min\{\beta/(2\eps k_{v[q]}),1\}$. Furthermore, $\mathbf{\wt I}_{k_{v[q]}}^L=(\wt I_{j,k_{v[q]}}\wt \nu_{j,n})_{j\in[k^*_q]}$ where the $\wt \nu_{j,n}$ are Bernoulli random variables, independent of everything else, with success probability as in~\eqref{eq:wtnu}.
	\lyu{To estimate the left hand side of~\eqref{eq:5.4.1}, we divide it into three parts:} 
	\begin{align}
	\ind_{\cE_q}{}&\sum_{x \in \N_0^{k^*_q}}\!\!\Big|\Pq{\mathbf I_{k_{v[q]}}^L=x}-\Pq{\mathbf V_{k_{v[q]}}^L=x}\Big|\nonumber \\ 
	\leq{}& \ind_{\cE_q}\!\!\!\!\sum_{x \in \{0,1\}^{k^*_q}}\!\!\Big|\Pq{\mathbf I_{k_{v[q]}}^L=x}-\Pq{(\wt I_{j,k_{v[q]}}\nu_{j,n})_{j\in[k^*_q]}=x}\Big|\label{eq:dtv1q}\\ 
	&+\ind_{\cE_q}\!\!\!\!\sum_{x \in \{0,1\}^{k^*_q}}\!\!\Big|\Pq{(\wt I_{j,k_0}\nu_{j,n})_{j\in[k^*_q]}=x}-\Pq{\mathbf{ \wt I}_{k_{v[q]}}^L=x}\Big|\label{eq:dtv2q}\\ 
	&+\ind_{\cE_q}\!\!\!\!\sum_{x \in \N_0^{k^*_q}}\!\!\Big|\Pq{\mathbf{\wt I}_{k_{v[q]}}^L=x}-\Pq{\mathbf{ V}_{k_{v[q]}}^L=x}\Big|.\label{eq:dtv3q}
	\end{align}
    We split the proof into three steps, each of which bounds one of the above terms. 

    \paragraph{\underline{Step 1}:} Bounding~\eqref{eq:dtv1q}. Let $(I_{j,k_{v[q]}})_{j\in[k^*_q]}$ denote a sequence of Bernoulli random variables \bas{that are independent conditionally on $k^*_q$ and $V_{k_{v[q]}-1}$} and all have parameter $\min\{\beta/|V_{k_{v[q]}-1}|,1\}$, if $V_{k_{v[q]}-1}\neq \emptyset$, and parameter $0$ otherwise.
	\lyu{Moreover, $(I_{j,k_{v[q]}})_{j\in[k^*_q]}$ is} independent of $(\nu_{j,n})_{j\in[k^*_q]}$ conditionally on $k^*_q$ and $V_{k_{v[q]}-1}$.  
	Then, we can write $I_{k_{v[q]}\to j}= I_{j,k_{v[q]}} \nu_{j,n}$. Indeed, the indicator $I_{k_{v[q]}\to j}$ equals one when $k_{v[q]}$ connects to $j$ by an edge, which occurs with probability $\min\{\beta/|V_{k_{v[q]}-1}|,1\}$ if $V_{k_{v[q]}-1}\neq \emptyset$ and $0$ otherwise, 
	and if $j$ survives after $n$ steps, which occurs with probability 
	\be 
	\frac{\P{\nu_{j,n}=1\text{ and }\nu_{i,n}=1\text{ for all }i\in \cA_{\rho[q]-1}\lyu{\cup \cP_{\rho[q]-1}}\,|\, \cC_q,V_{k_0-1}}}{\P{\nu_{i,n}=1\text{ for all }i\in \cA_{\rho[q]-1}\lyu{\cup \cP_{\rho[q]-1}}\,|\,\cC_q,V_{k_0-1}}}.
	\ee
	Writing $I_{k_{v[q]}\to j}=I_{j,k_{v[q]}}\nu_{j,n}$ allows us to couple the sequences $\mathbf I_{k_{v[q]}}^L$ and $(\wt I_{j,k_{v[q]}}\nu_{j,n})_{j\in[k^*_q]}$ by coupling the sequences $(I_{j,k_{v[q]}})_{j\in[k^*_q]}$ and $(\wt I_{j,k_{v[q]}})_{j\in[k^*_q]}$. This can be done by applying a standard Bernoulli coupling and yields, by a union bound,
	\be \ba 
	\mathbb P_q\Bigg({}&\bigcup_{j=1}^{k^*_q}\{I_{j,k_{v[q]}}\neq \wt I_{j,k_{v[q]}}\}\Bigg)\leq k^*_q \Bigg|\min\Big\{\frac{\beta}{2\eps k_{v[q]}},1\Big\}-\mathbb E_q\bigg[\ind_{\{V_{k_{v[q]}-1}\neq \emptyset\}} \min\Big\{\frac{\beta}{|V_{k_{v[q]}-1}|},1\Big\}\bigg]\Bigg|.
	\ea \ee 
	Recalling the event $\cQ_n$ from~\eqref{eq:An} and using Lemma~\ref{lemma:condconc}, for some constant $C_1>0$, we have the bounds
	\be \ba
	\Eq{\ind_{\cQ_{k_{v[q]}-1}} \min\Big\{\frac{\beta}{|V_{k_{v[q]}-1}|},1\Big\}}
	&\leq\Eq{\ind_{\{V_{k_{v[q]}-1}\neq \emptyset\}} \min\Big\{\frac{\beta}{|V_{k_{v[q]}-1}|},1\Big\}}\\
	&\leq \Eq{\ind_{\cQ_{k_{v[q]}-1}} \min\Big\{\frac{\beta}{|V_{k_{v[q]}-1}|},1\Big\}}+\Pq{\cQ_{k_{v[q]}-1}^c}\\ 
	&\leq \min\Big\{\frac{\beta}{2\eps k_{v[q]}},1\Big\}+C_1k_{v[q]}^{-4/3}.
	\ea \ee 
	Thus, we arrive at 
	\be\ba  \label{eq:Icouple2}
	\Pq{\exists j\in [k^*_q]: I_{k_{v[q]}\to j}\neq \wt I_{j,k_{v[q]}}\nu_{j,n}}&\leq \Pq{\exists j\in [k^*_q]: I_{j,k_{v[q]}}\neq \wt I_{j,k_{v[q]}}}\\ 
	&\leq C_1k_{v[q]}^{-1/3},
	\ea\ee 
	which completes the first part. 
	
	\paragraph{\underline{Step 2}:} Bounding~\eqref{eq:dtv2q}. We fix $\delta_1,\delta_2\in(0,1)$ and define the event
	\be 
	\cB_n:=\bigg\{\min\{j\in[k^*_q]:\wt I_{j,k_{v[q]}}=1\}\geq n^{\delta_1}, \sum_{j=1}^{k^*_q}\wt I_{j,k_{v[q]}}\leq \delta_2\log n\bigg\}.
	\ee 
	Note that $\cB_n$ holds with high probability irrespective of the choice of $\delta_1,\delta_2$: indeed, by a union bound and Markov's inequality, we can bound 
	\be \label{eq:bnqlbound}
	\mathbb P_q(\cB_n^c)\leq \frac{\beta  n^{\delta_1}}{2\eps k_{v[q]}}+\frac{\beta}{2\eps \delta_2\log n}.
	\ee 
	We then write
	\be\ba \label{eq:dtvcondL}
	\ind_{\cE_q}\sum_{x\in \{0,1\}^{k^*_q}}\!\!\!\!\Big|\mathbb P_q\Big((\wt I_{j,k_{v[q]}}{}&\nu_{j,n})_{j\in[k^*_q]}=x\Big)-\Pq{(\wt I_{j,k_{v[q]}}\wt \nu_{j,n})_{j\in[k^*_q]}=x}\Big|\\ 
	\leq \ind_{\cE_q}\mathbb E_q\Bigg[\sum_{x\in\{0,1\}^{k^*_q} }\!\!\!\!\!\Big|{}&\P{(\wt I_{j,k_{v[q]}}\nu_{j,n})_{j\in[k^*_q]}=x\,\Big|\, \cC_q,(\wt I_{j,k_{v[q]}})_{j\in[k^*_q]}}\\ 
	&-\P{(\wt I_{j,k_{v[q]}}\wt \nu_{j,n})_{j\in[k^*_q]}=x\,\Big|\, \cC_q,(\wt I_{j,k_{v[q]}})_{j\in[k^*_q]}}\Big|\Bigg].
	\ea\ee 
	Now, define
	\be 
	M_n:=\sum_{j=1}^{k^*_q} \wt I_{j,k_{v[q]}}. 
	\ee 
	Conditionally on the sequence $(\wt I_{j,k_{v[q]}})_{j\in[k^*_q]}$, let $(j_i)_{i\in [M_n]}$ denote the sequence of indices in increasing order such that $\wt I_{j_i,k_{v[q]}}=1$ for all $i\in[M_n]$. We then observe that, with $x=(x_j)_{j\in[k^*_q]}$ and $x'=(x_{j_i})_{i\in[M_n]}$, the events 
	\be 
	\Big\{(\wt I_{j,k_{v[q]}}\nu_{j,n})_{j\in[k^*_q]}=x\,\Big|\, \cC_q,(\wt I_{j,k_{v[q]}})_{j\in[k^*_q]}\Big\} \quad\text{and}\quad \Big\{(\nu_{j_i,n})_{i\in[M_n]}=x'\,\Big|\, \cC_q,(\wt I_{j,k_{v[q]}})_{j\in[k^*_q]}\Big\}
	\ee 
	have the same probability as long as $x_\ell=0$ for all $\ell\in [k^*_q]\setminus \{j_1,\ldots, j_{M_n}\}$. If, instead, $x_\ell=1$ for some $\ell\in [k^*_q]\setminus \{j_1,\ldots, j_{M_n}\}$, then the event on the left-hand side has probability 0. 
	The same holds when we replace $\nu_{j,n}$ by $\wt \nu_{j,n}$. As a result, \lyu{the right hand side of~\eqref{eq:dtvcondL} can be rewritten as} 
	\be \ba 
	\ind_{\cE_q}\mathbb E_q\Bigg[\sum_{x\in\{0,1\}^{M_n} }\!\!\!\!\!\Big|{}&\P{(\nu_{j_i,n})_{i\in[M_n]}=x\,\Big|\, \cC_q,(\wt I_{j,k_{v[q]}})_{j\in[k^*_q]}}\\ 
	&-\P{(\wt \nu_{j_i,n})_{i\in[M_n]}=x\,\Big|\, \cC_q,(\wt I_{j,k_{v[q]}})_{j\in[k^*_q]}}\Big| \Bigg]\\ 
	\leq \ind_{\cE_q}\mathbb E_q\Bigg[\ind_{\cB_n}\sum_{x\in\{0,1\}^{M_n} }\!\!\!\!\!\Big|{}&\P{(\nu_{j_i,n})_{i\in[M_n]}=x\,\Big|\, \cC_q,(\wt I_{j,k_{v[q]}})_{j\in[k^*_q]}}\\ 
	&-\P{(\wt \nu_{j_i,n})_{i\in[M_n]}=x\,\Big|\, \cC_q,(\wt I_{j,k_{v[q]}})_{j\in[k^*_q]}}\Big| \Bigg]+2\cdot \ind_{\cE_q}\Pq{\cB_n^c}.
	\ea\ee 
	On the event $\cB_n$, it holds that $M_n\leq \delta_2\log n$ and that $j_i\geq n^{\delta_1}$ for all $i\in [M_n]$. As a result, for every $r=1,\ldots, \lfloor \delta_2\log n\rfloor$, conditionally on the event $\{M_n = r\}$,
	we can apply Corollary~\ref{cor:condsurv} \bas{(and Remark~\ref{rem:condsurv})} with $j_0=n^{\delta_1}$ to bound the absolute value \lyu{in the last display} from above. Together with~\eqref{eq:bnqlbound} and \lyu{by choosing $\delta_2$ suitably small}, this yields the upper bound
	\be \label{eq:IItilda}
	\ind_{\cE_q}\Bigg[\cO\Bigg(\sum_{r=1}^{\lfloor \delta_2\log n\rfloor} \frac{4^rr}{n^{\delta_1/6}}\Bigg)+\frac{\beta n^{\delta_1}}{\eps k_{v[q]}}+\frac{\beta}{\eps\delta_2\log n}\Bigg]=\cO\Big(\frac{1}{\log n}\Big), 
	\ee 
	where we used the bound on $k_{v[q]}$ from the event $\cE_q$. 
	This concludes the second part. 
	
	\paragraph{\underline{Step 3}:} Bounding~\eqref{eq:dtv3q}. We bound~\eqref{eq:dtv3q} in two further steps. First, we couple $\mathbf{\wt I}_{k_{v[q]}}^L$ to an auxiliary sequence of Poisson random variables $\mathbf{\wt V}_{k_{v[q]}}^L=(\wt V_{k_{v[q]}\to j})_{j\in[k^*_q]}$ \lyu{independent under $\mathbb P_q$ and} where $\wt V_{k_{v[q]}\to j}$
	has mean equal to $\frac{\beta p}{2\eps k_{v[q]}}(j/n)^{(1-p)/(2\eps)}$. Then, in a second step, we couple $\mathbf{\wt V}_{k_{v[q]}}^L$ to $\mathbf V_{k_{v[q]}}^L$. 
	
	For the first step, we use Lemma~\ref{lemma:poibercoupling} to obtain 
	\be 
	\ind_{\cE_q}\Pq{\wt{\mathbf I}^L_{k_{v[q]}}\neq \wt{\mathbf V}^L_{k_{v[q]}}}
	\leq \ind_{\cE_q}\sum_{j=1}^{k^*_q}\Big(\frac{\beta p}{2\eps k_{v[q]}}\Big(\frac jn\Big)^{(1-p)/(2\eps)}\Big)^2\leq\ind_{\cE_q} \frac{\beta^2p^2}{4\eps^2 k_{v[q]}^2}  \sum_{j=1}^{k^*_q} \Big(\frac{j}{k_{v[q]}}\Big)^{(1-p)/\eps}.
	\ee
	From the event $\cE_q$, we know that $k_{v[q]}$ diverges and hence
	\be 
	\frac{1}{k_{v[q]}}\sum_{j=1}^{k^*_q} \Big(\frac{j}{k_{v[q]}}\Big)^{(1-p)/\eps}=\int_0^1 x^{(1-p)/\eps}\, \dd x+o(1). 
	\ee 
	Again, by using the lower bound on $k_{v[q]}$, we arrive at 
	\be 
	\ind_{\cE_q}\Pq{\wt{\mathbf I}^L_{k_{v[q]}}\neq \wt{\mathbf V}^L_{k_{v[q]}}}=\cO\Big(\frac{(\log\log n)^{q+1}}{n}\Big).
	\ee 
	For the second step, recall $\lambda_j^{\rho[q]}$ from~\eqref{eq:lambdaqL}. By Lemma~\ref{lemma:coupling poisson}, we obtain that
	\begin{equation}\begin{aligned}\label{eq:vneqboundql}
			\ind_{\cE_q}\P{\wt{\mathbf V}^L_{k_{v[q]}}\neq \mathbf V^L_{k_{v[q]}}\,\Big|\, a_{v[q]},k_{v[q]}}&\leq \ind_{\cE_q}\sum_{j=1}^{k^*_q}\Big|\lambda_j^{\rho[q]}-\frac{\beta p}{2\eps k_{v[q]}}\Big(\frac jn\Big)^{(1-p)/(2\eps)}\Big|.
	\end{aligned} \end{equation} 
	Moreover, for every \lyu{$j\le j_q := \min\{N_{v[q]},k_{v[q]}-1\}-1$}, the triangle inequality implies that
	\begin{align} 
	\Bigg|{}&\int_{(j-1)/n}^{j/n} \frac{1}{a_{v[q]}} x^{(1-p)/(2\eps)}\,\dd x-\frac{1}{ k_{v[q]}}\Big(\frac jn\Big)^{(1-p)/(2\eps)}\Bigg|\nonumber\\ 
	&\leq \Big|\frac{1}{a_{v[q]}}-\frac{n}{k_{v[q]}}\Big|\int_{(j-1)/n}^{j/n} x^{(1-p)/(2\eps)}\,\dd x +\frac{n}{k_{v[q]}}\int_{(j-1)/n}^{j/n} \Big|\Big(\frac jn\Big)^{(1-p)/(2\eps)}-x^{(1-p)/(2\eps)}\Big|\,\dd x\nonumber\\ 
	&\le \Big|\frac{1}{a_{v[q]}}-\frac{n}{k_{v[q]}}\Big|\frac1n+\frac{1}{k_{v[q]}}\Big|\Big(\frac jn\Big)^{(1-p)/(2\eps)}-\Big(\frac{j-1}{n}\Big)^{(1-p)/(2\eps)}\Big|\nonumber\\
    &\le \lyu{\frac{q+1}{k_{v[q]}\cdot a_{v[q]}n}\frac1n+\frac{1}{k_{v[q]}}\Big|\Big(\frac jn\Big)^{(1-p)/(2\eps)}-\Big(\frac{j-1}{n}\Big)^{(1-p)/(2\eps)}\Big|}\nonumber\\ 
    &= \lyu{n^{-3+o(1)}+\frac{1}{k_{v[q]}}\Big|\Big(\frac jn\Big)^{(1-p)/(2\eps)}-\Big(\frac{j-1}{n}\Big)^{(1-p)/(2\eps)}\Big|,}\label{eq:telescopic_2}
	\end{align}
 \lyu{where the penultimate inequality follows from the second part of the event $\cH_{q,2}$ from~\eqref{eq:Hq}, and the last equality follows from the event $\cH_{q,1}$. Finally, similarly to~\eqref{eq:sum_4.5}, by bounding the absolute values in~\eqref{eq:vneqboundql} for $j\in [\min\{N_{v[q]},k_{v[q]}-1\}]$ by the sum of the two terms, using~\eqref{eq:lambdaqL} and a telescopic summation of~\eqref{eq:telescopic_2}, we obtain that~\eqref{eq:vneqboundql} is at most
	\begin{equation} 
	\ind_{\cE_q}\bigg(n^{-1+o(1)} + n\cdot n^{-3+o(1)} + \frac{1}{k_{v[q]}}\Big(\frac{j_q}{n}\Big)^{(1-p)/(2\eps)}\bigg)=n^{-1+o(1)}.\label{eq:bd 5.15}
	\end{equation}}
	\lyu{Thus, by combining~\eqref{eq:vneqboundql} and~\eqref{eq:bd 5.15}}, we deduce that there exists a coupling of $\mathbf{\wt I}_{k_{v[q]}}^L$ and $\mathbf V_{k_{v[q]}}^L$ such that
	\be 
	\ind_{\cE_q}\Pq{\mathbf{\wt I}_{k_{v[q]}}^L\neq \mathbf{V}_{k_{v[q]}}^L}=\lyu{n^{-1+o(1)}}. 
	\ee 
	Combining this with~\eqref{eq:Icouple2} and~\eqref{eq:IItilda}, we obtain the desired bound for~\eqref{eq:5.4.1}. 
	
	It remains to deal with~\eqref{eq:5.4.2}. By a union bound and Markov's inequality, we obtain 
	\begin{equation}\begin{aligned}\label{eq:cyclebound1}
			\mathbb P{}&\Big(\{\exists j\in \cA_{\rho[q]-1}\cap [k_q^*]: I^L_{\lyu{k_{\barv[q]}\to j}}= 1\}\cap \{\mathbf I^L_{k_{\barv[q]}}=\mathbf V^L_{k_{\barv[q]}}\}\cap \cE_q\Big)\\
			&\leq \E{\ind_{\cE_q}\Pq{\exists j\in \cA_{\rho[q]-1}\cap [k_q^*]: V^L_{\lyu{k_{\barv[q]}\to j}}\geq 1}}\\
			&\leq \E{\ind_{\cE_q}|\cA_{\rho[q]-1}|\max_{j\in[k^*_q]}\lambda_j^{\rho[q]}}.
	\end{aligned} \end{equation} 
	We also observe that, for all $j\in[k^*_q]$,
	\begin{equation}
		\lambda_j^{\rho[q]}\leq \frac{\beta p}{2\eps a_{v[q]}n}\Big(\frac{j+1}{n}\Big)^{(1-p)/(2\eps)}\leq \frac{\beta p}{2\eps a_{v[q]}n}\Big(\frac{\lyu{2}k_{v[q]}}{n}\Big)^{(1-p)/(2\eps)}\leq \frac{\beta p}{\eps n }\lyu{(4}a_{v[q]})^{(1-p)/(2\eps)-1},
	\end{equation} 
	where the final inequality holds on the event $\cE_q$ for all sufficiently large $n$. 
	Regardless of the value of $\eps\in(0,1/2]$, on the event $\cE_q$, we can bound $\lyu{(4}a_{v[q]})^{(1-p)/(2\eps)-1}$ from above by $(\lyu{4}\log\log n)^{q+1}$. 
	Combining this estimate, the bound $|\cA_{\rho[q]-1}|=\cO\big( (\log\log n)^{q/r}\big)$ on the event $\cE_q$ and~\eqref{eq:cyclebound1}, this finally leads to 
	\begin{equation}
		\mathbb P\Big(\{\exists j\in \cA_{\rho[q]-1}\cap [k_q^*]: I^L_{\lyu{k_{\barv[q]}\to j}}= 1\}\cap \{\mathbf I^L_{k_{\barv[q]}}=\mathbf V^L_{k_{\barv[q]}}\}\cap \cE_q\Big)=o\Big(\frac{1}{\log n}\Big), 
	\end{equation} 
	 which concludes the proof.
\end{proof} 

\subsection{\texorpdfstring{Coupling the $R$-children of $v[q]$ in $\cT$ with the $R$-neighbours of $k_{\barv[q]}$ in $G_n$}{}}

In this section, we couple the $R$-neighbours of $k_{v[q]}$ and the $R$-children of $v[q]$. We define
\begin{equation}\label{eq:Rq}
\cR_{\rho[q]}:=\{\theta_{v[q]}^R=\tau_{v[q]}^R\}\cap\{ \forall i\in[\tau_{v[q]}^R]: |a_{v[q],i\bas{+\tau^L_{v[q]}}}-\tfrac{k_{v[q],i\bas{+\tau^L_{v[q]}}}}{n}|\leq \tfrac{q+2}{n}, k_{\barv[q],i\bas{+\tau^L_{v[q]}}}\in \cN_{\rho[q]-1}\}
\end{equation}
to be the event that the coupling of the $R$-neighbours of $k_{v[q]}$ and the $R$-children of $v[q]$ is successful. In words, $\cR_{\rho[q]}$ says that the number of $R$-neighbours of $k_{\barv[q]}$ is equal to the number of $R$-children of $\barv[q]$, that their respective rescaled labels and types are close, and that the exploration of the $R$-neighbours of $k_{v[q]}$ does not lead to cycles. \bas{Note that we implicitly assume that the coupling of the $L$-neighbours of $k_{v[q]}$ and the $L$-children of $v[q]$ is successful, since otherwise $k_{v[q],i+\tau_{v[q]}^L}$ may not be well-defined. As we are interested in the event $\cL_{\rho[q]}\cap cR_{\rho[q]}$, we do not explicitly state this requirement in the event $\cR_{\rho[q]}$.} 

We now define an encoding of the $R$-neighbours of $k_{v[q]}$ in $G_n$ and the $R$-children of $v[q]$ in $\cT$ that allows us to construct a coupling such that $\cR_{\rho[q]}$ holds with high probability. As in Section~\ref{sec:Lqcouple}, we cannot ensure that $\tfrac{k_{\barv[q]}+1}{n}\geq a_{\barv[q]}$ holds on the event $\cE_q=\cap_{i=1}^3 \cH_{q,i}$. Instead, we let $N_{v[q]}:=\lceil a_{\barv[q]}n\rceil$ and $k^*_q:=\min\{N_{v[q]},k_{\barv[q]}+1\}$ and use $k^*_q$ \lyu{to construct a discretised Poisson process similar to $V^R_{k_0}$ defined in~\eqref{eq:lambdas}, which
is key in our coupling. More precisely, if}
$N_{v[q]}\leq k_{\barv[q]}+1$, conditionally on $a_{v[q]}$, we define
\begin{equation}\begin{aligned}\label{eq:lambdaqR}
	&\lambda_{N_{v[q]}}^{\rho[q]}:=\int_{a_{\barv[q]}}^{N_{v[q]}/n} \frac{\beta p}{2\eps}x^{(1-p)/(2\eps)-1}\,\dd x, \text{ and}\\
	&\lambda^{\rho[q]}_j:=\int_{\tfrac{j-1}{n}}^{\tfrac jn}\frac{\beta p}{2\eps}x^{(1-p)/(2\eps)-1}\,\dd x\text{ for all }j\in \{N_{v[q]}+1,\ldots, n\}.
	\end{aligned}\end{equation} 
	If instead $N_{v[q]}\geq k_{\barv[q]}+2$, we additionally set $\lambda^{\rho[q]}_j:=0$ for all integers $j\in [k_q^*, N_{v[q]}-1]$, which defines $\lambda^{\rho[q]}_j$ for all $j\in [k_q^*, n]$. 
	Moreover, this allows us to define a discretised Poisson point process $\mathbf V^R_{k_{\barv[q]}}:=(V^R_{j\to k_{\barv[q]}})_{j=k^*_q}^n$ where $V^R_{j\to k_{\barv[q]}}$ is a Poisson random variable with mean $\lambda^{\rho[q]}_j$ (again, a Poisson variable with mean zero is equal to zero almost surely). 
	\lyu{More precisely, $\mathbf V^R_{k_{v[q]}}$ must be seen as} a discretisation of the Poisson point process on $(a_{v[q]},1)$ with \lyu{intensity} $\lambda^+(\dd x)$ that determines the $R$-children of $v[q]$ (see Definition~\ref{def:wll}).
	
	Similarly, we define the  Bernoulli point process $\mathbf I^R_{k_{\barv[q]}}:=(I^R_{j\to k_{v\lyu{[q]}}})_{j=k^*_q}^n$ as follows. For all $j\in [k_{\barv[q]}+1, n]$, $I^R_{j\to k_{v[q]}}$ equals one if $j$ connects to $k_{v[q]}$ by an edge in $G_n$ and $j\in \cN_{\rho[q]-1}\cup \cA_{\rho[q]-1}$, and equals zero otherwise. If $N_{v[q]}<k_{\barv[q]}+1$, for all $j\in [k^*_q, k_{\barv[q]}]$, we additionally set $I^R_{j\to k_{\barv[q]}}:=0$.
	
	Our aim is to couple the Bernoulli point process $\mathbf I^R_{k_{v[q]}}$ and the discretised Poisson point process $\mathbf V^R_{k_{v[q]}}$. To do this, we show the following lemma, which is similar in nature to Lemma~\ref{lemma:1Rcouple}.
	
	\begin{lemma}\label{lemma:qR}
Recall the event $\cE_q:=\cap_{i=1}^3\cH_{q,i}$ where the events $(\cH_{q,i})_{i\in[3]}$ were defined in~\eqref{eq:Hq}, and consider the processes $\mathbf I^R_{k_{\barv[q]}}$ and $\mathbf V^R_{k_{\barv[q]}}$. There exists a coupling such that
\begin{equation}
	\P{\big(\{\mathbf I^R_{k_{v[q]}}= \mathbf V^R_{k_{v[q]}}\}\cap \{\forall j\in \cA_{\rho[q]-1}\cap [k_q^*, n]: \bas{I}^R_{j\to k_{\barv[q]}}=0\}\big)^c\cap  \cE_q\,\Big|\,\lyu{\cC_q}} = \bas{\cO\Big(\frac{\log\log\log n}{\log n}\Big)},
\end{equation}
\lyu{where the constant in the $\cO(\cdot)$ is absolute. In particular, upon the event $\cE_q$ and conditionally on the random variables in $\cC_q$, the total variation distance between $\mathbf I^R_{k_{v[q]}}$ and $\mathbf V^R_{k_{v[q]}}$ is of order $\cO(\log\log\log n/\log n)$.}
\end{lemma}

\begin{remark}\label{rem:qR}
When the coupling is successful, in the sense that the event 
\begin{equation}
	\{\mathbf I^R_{k_{v[q]}}= \mathbf V^R_{k_{v[q]}}\}\cap \{\forall j\in \cA_{\rho[q]-1}\cap [k_q^*, n]: \bas{I}^R_{j\to k_{\barv[q]}}=0\}
\end{equation} 
holds, it directly follows that the event $\cR_{\rho[q]}$ holds as well. As a result, 
\begin{equation}
	\P{\cL_{\rho[q]}\cap \cR_{\rho[q]}^c\cap \cE_q}=\bas{\cO\Big(\frac{\log\log\log n}{\log n}\Big)},
\end{equation} 
where the event $\cL_{\rho[q]}$ ensures that the event $\cR_{\rho[q]}$ is well-defined.
\end{remark}

\begin{proof}[Proof of Lemma~\ref{lemma:qR}]
\lyu{As usual, we show the more general statements
\be\label{eq:5.6.1}
	\ind_{\cE_q}\sum_{x \in \N_0^{n-k^*_q}}\Big|\Pq{\mathbf I_{k_{v[q]}}^R=x}-\Pq{\mathbf V_{k_{v[q]}}^R=x}\Big| =\cO\Big(\frac{\log\log\log n}{\log n}\Big),
\ee
where the sum ranges over all sequences $x$ of $n-k_q^*$ non-negative integers, and
\be\label{eq:5.6.2} 
\ind_{\cE_q}\Pq{\{\exists j\in \cA_{\rho[q]-1}\cap [k^*_q,n]: I^R_{k_{v[q]}\to j}=1\}\cap \{\mathbf I^R_{k_{v[q]}}=\mathbf V^R_{k_{v[q]}}\}\cap \cE_q}=\cO\Big(\frac{\log\log\log n}{\log n}\Big).
\ee}
\lyu{We start by showing~\eqref{eq:5.6.1}. Recall the probability measure $\Pq{\cdot}:=\P{\cdot\,|\, \cC_q}$ conditioned on the set of random variables $\cC_q$ and the associated expectation $\Eq{\cdot}$.} 
Let $(\wt I_{j,k_0})_{j=k^*_q}^n$ be a sequence of independent Bernoulli random variables, where  $\wt I_{j,k_{v[q]}}$ has success probability $\min\{\beta/(2\eps j),1\}$. Furthermore, $\mathbf{\wt I}_{k_{v[q]}}^R=(\wt I_{j,k_{v[q]}}\wt \nu_{j,n})_{j=k^*_q}^n $, where the $\wt \nu_{j,n}$ are Bernoulli random variables, independent of everything else, with success probability as in~\eqref{eq:wtnu}. \lyu{To estimate the left hand side of~\eqref{eq:5.6.1}, we divide it into three parts:} 
	\begin{align}
	\ind_{\cE_q}\!\!\!\!{}&\sum_{x \in \N_0^{n-k^*_q}}\!\!\Big|\Pq{\mathbf I_{k_{v[q]}}^R=x}-\Pq{\mathbf V_{k_{v[q]}}^R=x}\Big|\nonumber \\ 
	\leq{}& \ind_{\cE_q}\!\!\!\!\sum_{x \in \{0,1\}^{n-k^*_q}}\!\!\Big|\Pq{\mathbf I_{k_{v[q]}}^R=x}-\Pq{(\wt I_{j,k_{v[q]}}\nu_{j,n})_{j=k^*_q}^n=x}\Big|\label{eq:dtv1qR}\\ 
	&+\ind_{\cE_q}\!\!\!\!\sum_{x \in \{0,1\}^{n-k^*_q}}\!\!\Big|\Pq{(\wt I_{j,k_{v[q]}}\nu_{j,n})_{j=k^*_q}^n=x}-\Pq{\mathbf{ \wt I}_{k_{v[q]}}^R=x}\Big|\label{eq:dtv2qR}\\ 
	&+\ind_{\cE_q}\!\!\!\!\sum_{x \in \N_0^{n-k^*_q}}\!\!\Big|\Pq{\mathbf{\wt I}_{k_{v[q]}}^R=x}-\Pq{\mathbf{ V}_{k_{v[q]}}^R=x}\Big|.\label{eq:dtv3qR}
	\end{align}
    We split the proof into three steps, each of which bounds one of the above terms. 

    \paragraph{\underline{Step 1}:} Bounding~\eqref{eq:dtv1qR}. Let $(I_{j,k_{v[q]}})_{j=k^*_q}^n$ denote a sequence of i.i.d.\ Bernoulli random variables, \bas{conditionally on $k^*_q$ and $(V_{j-1})_{j=k^*_q}^n$},  where $I_{j,k_{v[q]}}$ has parameter $\min\{\beta/|V_{j-1}|,1\}$ if $V_{j-1}\neq \emptyset$ and parameter $0$ otherwise, 
 also independent of $(\nu_{j,n})_{j=k^*_q}^n$, conditionally on $k^*_q$ and $V_{j-1}$.  Then, we can write $I_{j\to k_{v[q]}}= I_{j,k_{v[q]}} \nu_{j,n}$. Indeed, the indicator $I_{j\to k_{v[q]}}$ equals one when $j$ makes a connection with $k_{v[q]}$, which occurs with probability $\min\{\beta/|V_{j-1}|,1\}$ if, independently for each $j$, and if $j$ survives after $n$ steps, which occurs with probability    
	\be 
	\frac{\P{\nu_{j,n}=1\text{ and }\nu_{i,n}=1\text{ for all }i\in \cA_{\rho[q]-1}\,|\, \cC_q,V_{k_0-1}}}{\P{\nu_{i,n}=1\text{ for all }i\in \cA_{\rho[q]-1}\,|\,\cC_q,V_{k_0-1}}}.
	\ee
	Writing $I_{j\to k_{v[q]}}=I_{j,k_{v[q]}}\nu_{j,n}$ allows us to couple the sequences $\mathbf I_{k_{v[q]}}^R$ and $(\wt I_{j,k_{v[q]}}\nu_{j,n})_{j=k^*_q}^n$ by coupling the sequences $(I_{j,k_{v[q]}})_{j=k^*_q}^n$ and $(\wt I_{j,k_{v[q]}})_{j=k^*_q}^n$. This can be done by applying a standard Bernoulli coupling and yields, by a union bound,
	\be \ba 
	\mathbb P_q\bigg({}&\bigcup_{jk^*_q}^n: I_{j,k_{v[q]}}\neq \wt I_{j,k_{v[q]}}\bigg)\leq \sum_{j=k^*_q}^n \Bigg|\min\Big\{\frac{\beta}{2\eps j},1\Big\}-\Eq{ \min\Big\{\frac{\beta}{|V_{j-1}|},1\Big\}}\Bigg|.
	\ea \ee 
	By Lemma~\ref{lemma:condconc}, for some constant $C_1>0$, we obtain
	\be \ba
	\Eq{\ind_{\cQ_{j-1}} \min\Big\{\frac{\beta}{|V_{j-1}|},1\Big\}}
	&\leq\Eq{\ind_{\{V_{j-1}\neq \emptyset\}} \min\Big\{\frac{\beta}{|V_{j-1}|},1\Big\}}\\
	&\leq \Eq{\ind_{\cQ_{j-1}} \min\Big\{\frac{\beta}{|V_{j-1}|},1\Big\}}+\Pq{\cQ_{j-1}^c}\\ 
	&\leq \min\Big\{\frac{\beta}{2\eps j},1\Big\}+C_1j^{-4/3}.
	\ea \ee 
	We thus arrive at 
	\be
	\Pq{\exists j\in [k^*_q,n]: I_{j\to k_{v[q]}}\neq \wt I_{j,k_{v[q]}}\nu_{j,n}}\leq \Pq{\exists j\in [k^*_q,n]: I_{j,k_{v[q]}}\neq \wt I_{j,k_{v[q]}}}\leq C_1k_{v[q]}^{-1/3},
	\ee 
	which completes the first part. 
	
	\paragraph{\underline{Step 2}:} Bounding~\eqref{eq:dtv2qR}. We fix $\delta\in(0,1)$ and define the event
	\be 
	\cB_n:=\bigg\{ \sum_{j=k_{v[q]}+1}^n\wt I_{j,k_{v[q]}}\leq \delta\log n\bigg\}.
	\ee 
	Note that $\cB_n$ holds with high probability irrespective of the choice of $\delta$: indeed, by a union bound and Markov's inequality, we can bound 
	\be \label{eq:bnqrbound}
	\ind_{\cE_q}\P{\cB_n^c\,|\, \cC_q}=\ind_{\cE_q}\cO\Big(\frac{\log(n/k_{v[q]})}{\log n}\Big)=\cO\Big(\frac{\log\log\log n}{\log n}\Big).
	\ee 
	We then write
	\be\ba \label{eq:dtvcondqR}
	\ind_{\cE_q}\sum_{x\in \{0,1\}^{n-k^*_q}}\!\!\!\!\Big|\mathbb P_q\Big((\wt I_{j,k_{v[q]}}{}&\nu_{j,n})_{j=k^*_q}^n=x\Big)-\Pq{(\wt I_{j,k_{v[q]}}\wt \nu_{j,n})_{j=k^*_q}^n=x}\Big|\\ 
	\leq \ind_{\cE_q}\mathbb E_q\Bigg[\sum_{x\in\{0,1\}^{n-k^*_q} }\!\!\!\!\!\Big|{}&\P{(\wt I_{j,k_{v[q]}}\nu_{j,n})_{j=k^*_q}^n=x, \cC_q,(\wt I_{j,k_{v[q]}})_{j=k^*_q}^n}\\ 
	&-\P{(\wt I_{j,k_{v[q]}}\wt \nu_{j,n})_{j=k^*_q}^n=x\,\Big|\, \cC_q,(\wt I_{j,k_{v[q]}})_{j=k^*_q}^n}\Big|\Bigg].
	\ea\ee 
	Now, define
	\be 
	M_n:=\sum_{j=k_{v[q]}+1}^n \wt I_{j,k_{v[q]}}. 
	\ee 
	Conditionally on the sequence $(\wt I_{j,k_{v[q]}})_{j=k^*_q}^n$, let $(j_i)_{i\in [M_n]}$ denote the sequence of indices in increasing order, such that $\wt I_{j_i,k_{v[q]}}=1$ for all $i\in[M_n]$. We then observe that, with $x=(x_j)_{j=k^*_q}^n$ and $x'=(x_{j_i})_{i\in[M_n]}$, the events 
	\be 
	\Big\{(\wt I_{j,k_{v[q]}}\nu_{j,n})_{j=k^*_q}^n=x\,\Big|\, \cC_q,(\wt I_{j,k_{v[q]}})_{j=k^*_q}^n\Big\} \quad\text{and}\quad \Big\{(\nu_{j_i,n})_{i\in[M_n]}=x'\,\Big|\, \cC_q,(\wt I_{j,k_{v[q]}})_{j=k^*_q}^n\Big\}
	\ee 
	have the same probability, as long as $x_\ell=0$ for all $\ell\in [k^*_q,n]\setminus \{j_1,\ldots, j_{M_n}\}$. If, instead, $x_\ell=1$ for some $\ell\in [k^*_q,n]\setminus \{j_1,\ldots, j_{M_n}\}$, then the event on the left-hand side has probability zero. The same holds when we replace $\nu_{j,n}$ by $\wt \nu_{j,n}$. As a result, using this in~\eqref{eq:dtvcondqR}, we obtain 
	\be \ba 
	\ind_{\cE_q}\mathbb E_q\Bigg[\sum_{x\in\{0,1\}^{M_n} }\!\!\!\!\!\Big|{}&\P{(\nu_{j_i,n})_{i\in[M_n]}=x\,\Big|\, \cC_q,(\wt I_{j,k_{v[q]}})_{j=k^*_q}^n}\\ 
	&-\P{(\wt \nu_{j_i,n})_{i\in[M_n]}=x\,\Big|\, \cC_q,(\wt I_{j,k_{v[q]}})_{j=k^*_q}^n}\Big| \Bigg]\\ 
	\leq \ind_{\cE_q}\mathbb E_q\Bigg[\ind_{\cB_n}\sum_{x\in\{0,1\}^{M_n} }\!\!\!\!\!\Big|{}&\P{(\nu_{j_i,n})_{i\in[M_n]}=x\,\Big|\, \cC_q,(\wt I_{j,k_{v[q]}})_{j=k^*_q}^n}\\ 
	&-\P{(\wt \nu_{j_i,n})_{i\in[M_n]}=x\,\Big|\, \cC_q,(\wt I_{j,k_{v[q]}})_{j=k^*_q}^n}\Big| \,\Bigg|\, \cC_q \Bigg]+2\ind_{\cE_q}\Pq{\cB_n^c}.
	\ea\ee 
	On the event $\cB_n$, it holds that $M_n\leq \delta\log n$. As a result, by summing over all events $\{M_n=r\}$, with $r=1,\ldots, \lfloor \delta\log n\rfloor$, we can, for each such $r$, apply Corollary~\ref{cor:condsurv} \bas{ (in particular Remark~\ref{rem:condsurv})} with $j_0=k_{v[q]}$  (which is at least $n(\log\log n)^{-(q+1)}$ on the event $\cC_q$) to bound the absolute value from above. Together with~\eqref{eq:bnqrbound}, this yields the upper bound
	\be 
	\ind_{\cE_q}\Bigg[\cO\Bigg(\sum_{r=1}^{\lfloor \delta\log n\rfloor} \frac{4^rr(\log\log n)^{(q+1)/6}}{n^{1/6}}\Bigg)+\cO\Big(\frac{\log\log\log n}{\log n}\Big)\Bigg]=\cO\Big(\frac{\log\log\log n}{\log n}\Big), 
	\ee
	where we use the bounds on $k_{v[q]}$ in the event $\cE_q$ and choose $\delta$ sufficiently small, which concludes the second part. 
	
	\paragraph{\underline{Step 3}:} Bounding~\eqref{eq:dtv3qR}. We bound~\eqref{eq:dtv3qR} in two further steps. We first couple $\mathbf{\wt I}_{k_{v[q]}}^R$ to an auxiliary sequence of independent Poisson random variables $\mathbf{\wt V}_{k_{v[q]}}^R=(\wt V_{j\to k_{v[q]}})_{j=k^*_q}^n$, where $\wt V_{j\to k_{v[q]}}$ has a mean equal to the right-hand side of~\eqref{eq:Poimean}. Then, in a second step, we couple $\mathbf{\wt V}_{k_{v[q]}}^R$ to $\mathbf V_{k_{v[q]}}^R$. 
	
	For the first step, we use Lemma~\ref{lemma:poibercoupling} to obtain 
	\be 
	\ind_{\cE_q}\Pq{\wt{\mathbf I}^R_{k_{v[q]}}\neq \wt{\mathbf V}^R_{k_{v[q]}}}
	\leq \ind_{\cE_q}\sum_{j=k^*_q}^n\Big(\frac{\beta p}{2\eps j}\Big(\frac jn\Big)^{(1-p)/(2\eps)}\Big)^2\leq\ind_{\cE_q} \frac{\beta^2p^2}{4\eps^2 k_{v[q]}^2}  \sum_{j=k^*_q}^n \Big(\frac{j}{n}\Big)^{(1-p)/\eps}.
	\ee
	\bas{ We bound
		\be 
		\frac1n \sum_{j=k^*_q}^n \Big(\frac{j}{n}\Big)^{(1-p)/\eps} \leq \int_0^1 x^{(1-p)/\eps}\,\dd x+o(1).
		\ee 
		Hence, by the bound on $k_{v[q]}$ in the event $\cE_q$, we obtain
		\be 
		\ind_{\cE_q}\Pq{\wt{\mathbf I}^R_{k_{v[q]}}\neq \wt{\mathbf V}^R_{k_{v[q]}}}=\cO\Big(\frac{(\log\log n)^{2(q+1)}}{n}\Big).
		\ee 
	}
	In the second step, we use Lemma~\ref{lemma:coupling poisson} to obtain 
	\begin{equation}\begin{aligned}\label{eq:vneqboundqr}
			\ind_{\cE_q}\Pq{\wt{\mathbf V}^R_{k_{v[q]}}\neq \mathbf V^R_{k_{v[q]}}}&\leq \ind_{\cE_q}\sum_{j=k^*_q}^n\Big|\lambda_j^{\rho[q]}-\frac{\beta p}{2\eps j}\Big(\frac jn\Big)^{(1-p)/(2\eps)}\Big|.
	\end{aligned} \end{equation} 
\lyu{Then, on the one hand, for each $j\in[\max\{N_{v[q]},k_{\barv[q]}+1\}+1,n]$,
	\begin{align*} 
	\Big|\lambda^{\rho[q]}_j-\frac{\beta p}{2\eps j}\Big(\frac jn\Big)^{(1-p)/(2\eps)}\Big|
 &\leq\frac{\beta p}{2\eps}\int_{(j-1)/n}^{j/n} \Big| x^{(1-p)/(2\eps)-1}-\Big(\frac jn\Big)^{(1-p)/(2\eps)-1}\Big|\,\dd x\\
 &\leq\frac{\beta p}{2\eps n} \Big| \Big(\frac{j-1}{n}\Big)^{(1-p)/(2\eps)-1}-\Big(\frac jn\Big)^{(1-p)/(2\eps)-1}\Big|.
	\end{align*} 
By bounding each of the remaining absolute values in~\eqref{eq:vneqboundqr} using $\cH_{1,q}\cap \cH_{2,q}$ (see~\eqref{eq:Hq}), using~\eqref{eq:lambdaqR} and a telescopic summation of~\eqref{eq:vneqboundqr} (whose sign only depends on $\eps$), similarly to~\eqref{eq:sum_4.5} and~\eqref{eq:bd 5.15}, we obtain that~\eqref{eq:dtv3qR} is at most
\begin{equation}\label{eq:step3qR}
\ind_{\cE_q}\Big(n^{-1+o(1)} + \frac{\beta p}{2\eps n} \Big| \Big(\frac{k^*_q-1}{n}\Big)^{(1-p)/(2\eps)-1}-1\Big|\Big) = n^{-1+o(1)}.
\end{equation}} 
We deduce that we can couple $\mathbf{\wt I}_{k_{v[q]}}^R$ with $\mathbf V_{k_{v[q]}}^R$ such that
	\be 
	\ind_{\cE_q}\Pq{\mathbf{\wt I}_{k_{v[q]}}^R\neq \mathbf{V}_{k_{v[q]}}^R\,\Big|\, \cC_q}=\cO\Big(\frac{\log\log\log n}{\log n}\Big). 
	\ee 
Combined with the bounds on~\eqref{eq:dtv1qR} and~\eqref{eq:dtv2qR}, we obtain the desired bound on~\eqref{eq:5.6.1}.
	
It remains to bound~\eqref{eq:5.6.2}. Via a union bound and Markov's inequality, we obtain 
	\begin{equation}\begin{aligned}\label{eq:cyclebound}
			\mathbb P{}&\Big(\{\exists j\in \cA_{\rho[q]-1}\cap [k_q^*,n]: I^R_{j\to k_{\barv[q]}}= 1\}\cap \{\mathbf I^R_{k_{\barv[q]}}=\mathbf V^R_{k_{\barv[q]}}\}\cap \cE_q\Big)\\
			&\leq \E{\ind_{\cE_q}\P{\exists j\in \cA_{\rho[q]-1}\cap [k_q^*,n]: V^R_{j\to k_{\barv[q]}}\geq 1\,\Big|\, \cC_q}}\\
			&\leq \E{\ind_{\cE_q}|\cA_{\rho[q]-1}|\max_{j\in[k^*_q,n]}\lambda_j^{\rho[q]}}.
	\end{aligned} \end{equation} 
	We then observe that for all $j\in[ k^*_q,n]$,
	\begin{equation}
		\lambda_j^{\rho[q]}\leq \frac{\beta p}{2\eps}\Big(\frac{j+1}{n}\Big)^{(1-p)/(2\eps)-1}\leq \frac{\beta p}{2\eps a_{v[q]}n}\Big(\frac{k_{v[q]}}{n}\Big)^{(1-p)/(2\eps)}\leq \frac{\beta p}{\eps n }a_{v[q]}^{(1-p)/(2\eps)-1},
	\end{equation} 
	where the final inequality holds on the event $\cE_q$ for all sufficiently large $n$. Regardless of the value of $\eps\in(0,1/2)$, we can bound $a_{v[q]}^{(1-p)/(2\eps)-1}$ from above by $(\log\log n)^{q+1}$ on the event $\cE_q$. Using this upper bound in~\eqref{eq:cyclebound} and bounding $|\cA_{\rho[q]-1}|=\cO\big( (\log\log n)^{q/r}\big)$ on $\cE_q$, this finally leads to 
	\begin{equation}
		\mathbb P\Big(\{\exists j\in \cA_{\rho[q]-1}\cap [k_q^*, n]: I^R_{j\to k_{\barv[q]}}= 1\}\cap \{\mathbf I^R_{k_{\barv[q]}}=\mathbf V^R_{k_{\barv[q]}}\}\cap \cE_q\Big)=\cO\Big(\frac{(\log\log n)^{q/r+q+1}}{n}\Big), 
	\end{equation}
which finishes the proof.
\end{proof}

\subsection{Proof of Lemma~\ref{lemma:kappacouple}}

In this section, we prove Lemma~\ref{lemma:kappacouple} using the coupling of the immediate neighbours of $k_{\barv[q]}$ in $G_n$ and $\barv[q]$ in $\cT$ constructed in the previous sections.

\begin{proof}[Proof of Lemma~\ref{lemma:kappacouple}]
Recall $\cL_{\rho[q]}, \cR_{\rho[q]}$ from~\eqref{eq:Lq} and~\eqref{eq:Rq}, respectively, and observe that $\cK_{\rho[q]}$ is implied by $\cL_{\rho[q]}\cap \cR_{\rho[q]}$. 
It thus follows from Lemmas~\ref{lemma:qLcouple} and~\ref{lemma:qR} that
\begin{equation}
	\P{\cK_{\rho[q]}^c\cap \cE_q}\leq \P{\cL_{\rho[q]}^c\cap\cE_q}+\P{\cL_{\rho[q]}\cap \cR_{\rho[q]}^c\cap\cE_q} = \lyu{\cO\Big(\frac{\log\log\log n}{\log n}\Big)}
\end{equation} 
\lyu{with an absolute constant in the $\cO(\cdot)$,} as desired.
\end{proof} 

\paragraph{\texorpdfstring{Coupling the remaining vertices in $B_q(G_n,k_0)$ and $B_q(\cT,0)$.}{}}
In the previous sections, we covered the exploration of the neighbours of $k_{\barv[q]}$, the smallest vertex in $\partial B_q(G_n,k_0)$ with respect to the BFS ordering, and coupled this process with the construction of the children of $\barv[q]$ in $\cT$. 
As discussed prior to this coupling, the same proofs with only minor modifications can be used to couple \emph{all} the vertices in  $\partial B_q(G_n,k_0)$ and in $\partial B_q(\cT,0)$. \lyu{The changes to be done consist og:} 
\begin{itemize}
    \item \lyu{Estimating the probability that an edge is sent to an already processed vertex in the same level. As already pointed out, the probability of this bad event is readily bounded from above in parallel with the above proof using the event $\cQ_i$ from~\eqref{eq:An} (where $i$ is the label of the processed vertex), Lemma~\ref{lemma:condconc} and the event $\cH_{q,1}$ (ensuring that $i$ is suitably large).}
    \item Estimating the number of \emph{probed} and \emph{active} vertices in the BFS exploration of $G_n$. 
    However, the bounds used in the proofs of the lemmas in this section \lyu{are still valid in the general case.}
\end{itemize}


\subsection{Summary}\label{sec:summary}
We are now ready to prove Proposition~\ref{prop:qcouple}.

\begin{proof}[Proof of Proposition~\ref{prop:qcouple}]
We start by recalling the bounds in~\eqref{eq:inducbound}, \eqref{eq:H3c} and~\eqref{eq:H1c} together with the hypothesis in Proposition~\ref{prop:qcouple}, which yields
\begin{equation}\begin{aligned}\label{eq:beforelast}
		\P{\Big(\bigcap_{i=1}^3 \cH_{q+1,i}\Big)^c}\leq \P{\cH_{q+1,2}^c\cap \cH_{q+1,1}\cap \cH_{q+1,3}\cap \bigcap_{i=1}^3 \cH_{q,i}}+2\beta (\log\log n)^{q/r-p/(2\eps)}
\end{aligned} \end{equation} 
for all sufficiently large $n$. Recall $\cK_{t,q}$ from~\eqref{eq:kappat}, $\cK_{\rho[q]}=\cK_{\rho[q],q}$ from~\eqref{eq:kapparhoq} and that, for all $t \in [\rho[q] ,\rho[q+1]-1]$ (with $\rho[q]$ defined in~\eqref{eq:rhoq}), we have $k_{v(t)}=k[t]$. 
(Otherwise said, $k_{v(t)}$ is the vertex in $\cA_{t-1}$ with smallest BFS order whose immediate neighbours are to be explored at step $t$ of the BFS exploration.) 
Recalling that $\cE_q=\cap_{i=1}^3\cH_{q,i}$, it follows from~\eqref{eq:hq+12} that
\begin{equation}
	\P{\cH^c_{q+1,2}\cap \cE_q}=\P{\bigg(\bigcap_{t=\rho[q]}^{\rho[q+1]-1}\cK_{t,q}\bigg)^c\cap \cE_q}\lyu{\le}\E{\sum_{t=\rho[q]}^{\rho[q+1]-1}\ind_{\cE_q}\ind_{ \cap_{s=\rho[q]}^{t-1} \cK_{s,q}}\
		\P{\cK_{t,q}^c\,\Big|\, \wt \cC_t}}.
\end{equation} 
\lyu{As discussed in the paragraph before Section~\ref{sec:summary}, all bounds in Section~\ref{sec:lwlcont} done for $k[q]$ remain valid for all $t\in [\rho[q],\rho[q+1]-1]$.}
Moreover, conditionally on $\cE_q \cap (\cap_{s=\rho[q]}^{t-1} \cK_{s,q})$, it follows that $\rho[q+1]-\rho[q]\leq (\log\log n)^{q/r}$, 
so we obtain the upper bound 
\begin{equation}
	\P{\cH^c_{q+1,2}\cap \bigcap_{i=1}^3\cH_{q,i}} = \cO\Big(\frac{(\log\log n)^{1+q/r}}{\log n}\Big).
\end{equation} 
Using this bound in~\eqref{eq:beforelast} finally yields 
\begin{equation}
	\P{\Big(\bigcap_{i=1}^3 \cH_{q+1,i}\Big)^c} \lyu{\le 3\beta}(\log\log n)^{q/r-p/(2\eps)}, 
\end{equation} 
which concludes the proof.
\end{proof}

With Proposition~\ref{prop:qcouple} and Corollary~\ref{cor:rcouple} at hand, Theorem~\ref{thrm:LWC intro} readily follows.

\begin{proof}[Proof of Theorem~\ref{thrm:LWC intro}]
Fix any $r\in \N$. Under the coupling of $B_r(G_n,k_0)$ and $B_r(\cT,0)$ from Corollary~\ref{cor:rcouple}, we have that
\begin{equation}\begin{aligned}
		d_{\mathrm{TV}}(\mathfrak{L}(B_r(G_n,k_0)),\mathfrak{L}(B_r(\cT,0)))
		&\leq \P{B_r(G_n,k_0)\not \cong B_r(\cT,0)}\\
		&= \cO((\log\log n)^{-(1-p)/(2\eps)-1/r}),
\end{aligned}\end{equation} 
and hence $(\cT,0)$ is the local limit of $(G_n,k_0)$.
\end{proof}

\section{The DRGVR model as an inhomogeneous random graph}\label{sec inhom E-R}

In this section, we compare our model with an inhomogeneous version of the Erd\H{o}s-R\'enyi random graph, \lyu{which allows us to transfer some established results for the latter to the DRGVR}. 
Such a comparison method was previously introduced and used by Chung and Lu in~\cite{CL04} to compare preferential attachment models with vertex and edge removal to inhomogeneous random graphs. 
Here, we are able to provide a more precise comparison thanks to the more tractable dynamics of the DRGVR model.

Recall the notations $V_n$, $\beta$, $\eps$ and $p$ from Definitions~\ref{def:drgvr} and~\ref{def:wll}. We define the following inhomogeneous random graph model.

\begin{definition}[Birth-death inhomogeneous Erd\H{o}s-R\'enyi graph]\label{def BDIER}
Fix a constant $\Delta > 0$ and an integer $n$. The \emph{Birth-death inhomogeneous Erd\H{o}s-R\'enyi graph}, denoted $G_{BD}(\Delta, n)$, has vertex set $V_n$ and, for all $1\leq i<j \le |V_n|$, the edge between the $i^{\text{th}}$ and the $j^{\text{th}}$ vertex in $V_n$ is sampled independently of other edges with probability
\begin{equation}\label{eq def p_ij}
	p_{ij} = p_{ij}(\Delta) = \Delta \beta (2\eps )^{-(1-p)/p}\frac1n \Big(\frac jn\Big)^{-2\eps/p}.
\end{equation}
\end{definition}

Observe that we condition on $V_n$ in the above definition. On the event $\cQ_n = \{||V_n| - \mathbb E |V_n|| \le n^{2/3}\}$ (which holds w.h.p.\ by Lemma~\ref{lemma:condconc}),~\eqref{eq def p_ij} rewrites as 
\begin{equation}\label{eq p_ij under cA_n}
p_{ij}=(1+o(1))\frac{\Delta \beta}{|V_n|}\Big(\frac{j}{|V_n|}\Big)^{-2\eps/p}.
\end{equation}
The reason for introducing this Erd\H{o}s-R\'enyi-type model is that w.h.p.\ one may ``sandwich'' $G_n$ between $G_{BD}(1-\hat\delta_n, n)$ and $G_{BD}(1+\hat\delta_n, n)$ where $(\hat\delta_n)_{n\ge 1}$ is a sequence satisfying \lyu{$\hat\delta_n = o(1)$.} 
\lyu{One significant advantage of this alternative Erd\H{o}s-R\'enyi model is that its} vertices``carry no randomness'' in the sense that, as opposed to the DRGVR, these do not have random arrival times and the birth-death process serving to define $(G_n)_{n\ge 1}$ is independent of the states of the edges in $G_{BD}(\Delta, n)$. \lyu{The next lemma shows} the coupling claimed above.

\begin{lemma}\label{lemma sandwich}
There is a sequence $(\hat\delta_n)_{n\ge 1}$ tending to $0$ and \lyu{such that $G_{BD}(1-\hat\delta_n, n)$, $G_n$ and $G_{BD}(1+\hat\delta_n, n)$ can be coupled so that} w.h.p.
\begin{equation*}
	G_{BD}(1-\hat\delta_n, n)\subseteq G_n\subseteq G_{BD}(1+\hat\delta_n, n).
\end{equation*}
\end{lemma}
\begin{proof}
Recall the set $\{v_1, \ldots, v_{|V_n|}\} = V_n$, the integer $\ell = \ell(n)$ and the sequence $(\delta_i)_{i=1}^{\infty}$ from Part~\eqref{part ii} of Lemma~\ref{lemma:oldvertgen}, and the event $\cQ_i = \{||V_i|-2\eps i|\le i^{2/3}\}$. We prove the lemma for $\hat\delta_n = \delta_n + \max\{\delta_n^{1/2}, (\log n)^{-1}\}$. Also, for any $j\in [|V_n|]$, set 
\begin{equation*}
	r_j = \left\lfloor (1-\delta_n)n^{(1-p)/p} \left(\frac{j}{2\eps}\right)^{2\eps/p}\right\rfloor \text{  and  } R_j = \left\lceil (1+\delta_n)n^{(1-p)/p} \left(\frac{j}{2\eps}\right)^{2\eps/p} \right\rceil,
\end{equation*}
and define the events 
\begin{align*}
	\cW_1& = \{v_1\ge n^{(1-p)/(2p)}\}, \quad &&\cW_2 = \bigcap_{j = \ell}^{|V_n|}\; \{v_j\in \left[r_j, R_j\right]\},\\
	\cC &= \bigcap_{i = \lceil n^{(1-p)/(2p)}\rceil}^{n}\; \cQ_i\quad&&\text{and}\quad\cW = \cW_1\cap \cW_2\cap \cC.
\end{align*}
Then, $\cW_1$ and $\cW_2$ both hold w.h.p.\ by Lemma~\ref{lemma:oldvertgen}, and $\cC$ \lyu{does as well by Lemma~\ref{lemma:condconc} and a union bound.}
Moreover, note that each of $\cW_1$, $\cW_2$ and $\cC$ is measurable with respect to the process $(V_i)_{i=1}^n$ (and hence $\cW$ as well). Then, for all $j\ge \ell+1$ and $i < j$, 
\begin{equation}\label{eq coupling 3 graphs}
	\begin{split}
		p_{ij}(1-\hat\delta_n) &=\; (1-\hat \delta_n) \beta (2\eps)^{-(1-p)/p} \frac{1}{n} \left(\frac{j}{n}\right)^{-2\eps/p} \le \frac{\beta}{2\eps R_j + R_j^{2/3}}\\
		&\le\; \ind_{\cW}\frac{\beta}{\max_{t\in [r_j, R_j]} |V_t|}\le \P{\{v_i, v_j\}\in G_n\mid \cW} \le \ind_{\cW}\frac{\beta}{\min_{t\in [r_j, R_j]} |V_t|}\\
		&\le\; \frac{\beta}{2\eps r_j - r_j^{2/3}}\le (1+\hat \delta_n) \beta (2\eps)^{-(1-p)/p} \frac{1}{n} \left(\frac{j}{n}\right)^{-2\eps/p} = p_{ij}(1+\hat\delta_n),
	\end{split}
\end{equation}
where for the first inequality we used that $$(1-\hat\delta_n)(1+\delta_n)(1+R_j^{-1/3}/(2\eps))\le (1+\delta_n-\hat\delta_n)(1+n^{-\Omega(1)})\le (1-(\log n)^{-1})(1+n^{-\Omega(1)})\le 1,$$ 
and for the last inequality we used that
\begin{align*}
	(1+\hat\delta_n)(1-\delta_n)(1-r_j^{-1/3}/(2\eps))
	&\ge\; (1+\max\{\delta_n^{1/2}, (\log n)^{-1}\} - \delta_n)(1-n^{\Omega(1)})\\
	&\ge\; (1+\max\{\delta_n^{1/2}, (\log n)^{-1}\}/2)(1-n^{\Omega(1)})\ge 1.
\end{align*}
As usual, for every pair of vertices $\{v_i, v_j\}$, we couple the states of the edge $\{v_i, v_j\}$ in $G_{BD}(1-\hat\delta_n, n)$, $G_n$ and $G_{BD}(1+\hat\delta_n, n)$ by sampling a random variable $U_{i,j}\sim \mathrm{Unif}[0,1]$ and setting
\be\ba 
	\{v_i, v_j\}&\in G_{BD}(1-\hat\delta_n, n)\ \ && \iff \ \ U_{i,j}\le p_{ij}(1-\hat\delta_n);\\
	\{v_i, v_j\}&\in G_n\quad  &&\iff \ \ U_{i,j}\le \mathbb P(\{v_i, v_j\}\in E(G_n)\mid \cW);\\
	\{v_i, v_j\}&\in G_{BD}(1+\hat\delta_n, n) \quad &&\iff \ \  U_{i,j}\le p_{ij}(1+\hat\delta_n).
\ea\ee

By~\eqref{eq coupling 3 graphs}, this coupling ensures that, under the event $\cW$, all edges with an endvertex among $(v_j)_{j=\ell+1}^{|V_n|}$ satisfy 
$$\{\{v_i, v_j\}\in G_{BD}(1-\hat\delta_n, n)\}\subseteq \{\{v_i, v_j\}\in G_n\}\subseteq \{\{v_i, v_j\}\in G_{BD}(1+\hat\delta_n, n)\}.$$ 
\lyu{It remains to deal with the $\tbinom{\ell}{2} = (\log n)^{\cO(1)}$ remaining vertex pairs.} 
However, on the event $\cW$ (and, in particular, $\cW_1$), each of them appears with probability $n^{-\Omega(1)}$. Since $\ell^2 n^{-\Omega(1)} = o(1)$, a union bound implies that w.h.p.\ none of these edges appears in any of $G_{BD}(1-\hat\delta_n, n)$, $G_n$ and $G_{BD}(1+\hat\delta_n, n)$. Hence, on the event $\cW$, we constructed a coupling of the three graphs ensuring that w.h.p.\ $G_{BD}(1-\hat\delta_n, n)\subseteq G_n\subseteq G_{BD}(1+\hat\delta_n, n)$. The fact that $\cW$ itself holds w.h.p.\ finishes the proof of the lemma.
\end{proof}

\subsection{\texorpdfstring{Applications of Lemma~\ref{lemma sandwich}: proofs of Theorems~\ref{thm components intro} and~\ref{thm approximate intro} and Proposition~\ref{prop beta_c intro}}{}}

In this section, we combine Lemma~\ref{lemma sandwich} with results from~\cite{BolJanRio05} to derive Theorems~\ref{thm components intro} and~\ref{thm approximate intro} and Proposition~\ref{prop beta_c intro}. We start with a brief summary of the notation and the theorems from~\cite{BolJanRio05} that we need in the sequel.

\paragraph{A closer look at the paper of Bollob\'as, Janson, and Riordan.} To start, we introduce some notation from~\cite{BolJanRio05}. A \emph{ground space} is a pair $(\cS, \mu)$ where $\cS$ is a separable metric space and $\mu$ is a Borel probability measure on $\cS$. Also, a \emph{vertex space} $\cV$ is a triplet $(\cS, \mu, (\mathbf{x}_n)_{n\ge 1})$ where $(\cS, \mu)$ is a ground space and $\mathbf x_n = (x_i^{n})_{i=1}^n$ is a random sequence of $n$ points of $\cS$ such that, for every measurable set $A\subseteq \cS$ with boundary $\partial A$ satisfying $\mu(\partial A) = 0$, 
\begin{equation}\label{eq def mu}
\frac{1}{n} \sum_{i=1}^n \ind_{x_i^{n}\in A} \xrightarrow{\mathbb P} \mu(A).
\end{equation}

A \emph{kernel} $\kappa$ on a ground space $(\cS, \mu)$ is a symmetric non-negative measurable function on $\cS\times \cS$. By a kernel on a vertex space $(\cS, \mu, (\mathbf{x}_n)_{n\ge 1})$ we mean a kernel on $(\cS, \mu)$. Also, given a vertex space $\cV$, a kernel $\kappa$ on $\cV$ and a sequence $\mathbf{x}_n$, we let $G^{\cV}(n, \kappa)$ denote the random graph with vertex set $\mathbf{x}_n$ where the edge between $x_i^n$ and $x_j^n$ appears independently of other edges with probability $\min\{\kappa(x_i^n, x_j^n)/n, 1\}$.

A kernel $\kappa$ on a vertex space $\cV = (\cS,\mu,(\mathbf{x}_n)_{n\ge 1})$ is called \emph{graphical} if the following conditions hold simultaneously:
\begin{enumerate}
\item\label{gr i1} $\kappa$ is continuous almost everywhere on $\cS\times \cS$;
\item\label{gr i2} $\kappa$ is integrable over $\cS\times \cS$ with respect to the product measure $\mu\times \mu$,
\item It holds that
\begin{equation}\label{eq BJR graphic}
	\frac{1}{n} \mathbb E[|E(G^{\cV}(n,\kappa))|] \xrightarrow[n\to \infty]{} \frac{1}{2} \int \int_{\cS\times \cS} \kappa(x,y) \mathrm{d}\mu(x)\mathrm{d}\mu(y).
\end{equation}
\end{enumerate}
Moreover, for a kernel $\kappa$ and a sequence of kernels $(\kappa_n)_{n\ge 1}$ on $\cV$, we say that $(\kappa_n)_{n\ge 1}$ is \emph{graphical on $\cV$ with limit $\kappa$} if, \lyu{first,} for almost every $(x,y)\in \cS\times \cS$,
\begin{equation}\label{eq sequences kappa}
x_n\xrightarrow[n\to \infty]{} x \text{ and } y_n\xrightarrow[n\to \infty]{} y \text{ in } \cS \text{ imply that } \kappa_n(x_n,y_n)\xrightarrow[n\to \infty]{} \kappa(x,y),
\end{equation}
\lyu{second,} $\kappa$ satisfies conditions~\eqref{gr i1} and~\eqref{gr i2} above, and \lyu{third},
\begin{equation}\label{eq conv kappa}
\frac{1}{n} \mathbb E[|E(G^{\cV}(n,\kappa_n))|] \xrightarrow[n\to \infty]{} \frac{1}{2} \int \int_{\cS\times \cS} \kappa(x,y) \mathrm{d}\mu(x)\mathrm{d}\mu(y).
\end{equation}
For a kernel $\kappa$ on $(\cS, \mu)$, define $T_{\kappa}$ as the integral operator defined by
\begin{equation}\label{eq T_kappa}
(T_{\kappa} f)(x) := \int_{\cS} \kappa(x,y) f(y) \mathrm{d}\mu(y),
\end{equation}
where $f:\cS\to \R$ is any measurable function for which the integral is defined (finite or $+\infty$) for almost every $x$, and let also
\begin{equation*}
||T_{\kappa}|| := \sup\{||T_{\kappa} f||_2: f\ge 0, ||f||_2\le 1\}\le \infty.
\end{equation*}
Then, a kernel $\kappa$ is called \emph{subcritical} if $||T_\kappa|| < 1$ and \emph{supercritical} if $||T_\kappa|| > 1$. Moreover, a kernel is \emph{irreducible} if the fact that $\kappa$ is almost everywhere 0 on $A\times (\cS\setminus A)$ for some measurable set $A$ implies that $\mu(A) = 0$ or $\mu(\cS\setminus A) = 0$.

Finally, for a kernel $\kappa$ on a ground space $(\cS, \mu)$ and a point $x\in \cS$, consider the following multi-type Galton-Watson process $\mathfrak{X}_x$: Initially, we mark a single point $x\in \cS$. Then, at every step and every point $y\in \cS$ marked at that step, unmark $y$ and mark new points according to a Poisson process with intensity $\kappa(y,z) \mathrm{d}\mu(z)$. Observe that the Poisson processes on $\cS$ for different points $y$ are independent, and the point configuration at the next step is the union of the Poisson processes associated to the points $y\in \cS$ marked at the current step. 
Finally, let $\rho(\kappa, x)$ denote the probability that the process $\mathfrak{X}_x$ survives eternally, and denote 
\begin{equation}\label{eq:rho}
\rho(\kappa) := \int_{\cS} \rho(\kappa, x) \mathrm{d}\mu(x).
\end{equation}
\lyu{The meaning of the multi-type process $\mathfrak{X}_x$ is that it coincides with the local limit of $G^{\mathcal V}(n,\kappa)$ for all graphical kernels $\kappa$.}

One of the main results from~\cite{BolJanRio05} characterises the component structure of the graph $G^{\cV}(n, \kappa)$. Denote by $\cC_1 = \cC_1(\cV, \kappa_n)$ and $\cC_2 = \cC_2(\cV, \kappa_n)$ the largest and the second-largest component in $G^{\cV}(n, \kappa_n)$.

\begin{theorem}[see Theorems 3.1 and 3.12 in~\cite{BolJanRio05}]\label{thm BJR components}
Let $(\kappa_n)$ be a graphical sequence of kernels on a vertex space $\cV$ with limit $\kappa$.
\begin{enumerate}
	\item If $||T_{\kappa}|| \le 1$, then $|\cC_1|/n$ converges in probability to $0$. Moreover, if $\kappa$ is subcritical, that is, $||T_{\kappa}|| < 1$, and $\sup_{x,y,n}\kappa_n(x,y) < \infty$, then w.h.p.\ $|\cC_1| = \cO(\log n)$.
	\item If $||T_{\kappa}|| > 1$, $\kappa$ is irreducible and either $\inf_{x,y,n} \kappa_n(x,y) > 0$ or $\sup_{x,y,n} \kappa_n(x,y) < \infty$, then $|\cC_1|/n$ converges in probability to $\rho(\kappa)$, and moreover, w.h.p.\ $|\cC_2| = \cO(\log n)$.
\end{enumerate}
\end{theorem}

Another main result of~\cite{BolJanRio05} concerns the graph distance between two uniformly chosen vertices in $G^{\cV}(n, \kappa_n)$.

\begin{theorem}[Parts of Theorem 3.14 in~\cite{BolJanRio05}]\label{thm BJR typical distance}
Let $(\kappa_n)_{n\ge 1}$ be a graphical sequence of kernels on a
vertex space $\cV$ with limit $\kappa$, where $||T_{\kappa}|| > 1$. Fix any constant $\eps\in (0,1)$. Denote by $d$ the graph distance in $G^{\cV}(n, \kappa_n)$ (where $n$ and $\kappa_n$ are spared in the notation for convenience). If $\kappa$ is irreducible, then
\begin{equation*}
	\frac{1}{n^2} |\{\{v,w\}: d(v,w) < \infty\}| \xrightarrow{\mathbb P} \frac{\rho(\kappa)^2}{2}.
\end{equation*}
If moreover $||T_{\kappa}|| < \infty$, then
\begin{equation*}
	\frac{1}{n^2} \left|\left\{\{v,w\}: d(v,w) < (1+\eps)\frac{\log n}{\log ||T_{\kappa}||}\right\}\right| \xrightarrow{\mathbb P} \frac{\rho(\kappa)^2}{2},
\end{equation*}
and if $\sup_{x,y,n}\kappa_n(x,y) < \infty$, then also
\begin{equation*}
	\frac{1}{n^2} \left|\left\{\{v,w\}: d(v,w) < (1-\eps)\frac{\log n}{\log ||T_{\kappa}||}\right\}\right| \xrightarrow{\mathbb P} 0.
\end{equation*}
\end{theorem}

\paragraph{Adaptation to our setting.} Our next aim is to show that, conditionally on $V_n$, the random graph $G_{BD}(1,n)$ is a particular instance of the general setting described above. 
We work on the event $\cQ_n = \{||V_n|-2\eps n|\leq n^{2/3}\}$, which allows us to use the more convenient expression~\eqref{eq p_ij under cA_n} for $p_{ij}$. 
\lyu{Apart from} the fact that this choice leads to a certain rescaling (the sequence $\mathbf{x}_n$ has only $|V_n|$ terms), no substantial modifications are needed to apply the theorems from~\cite{BolJanRio05}.

First of all, for any $i < j \le |V_n|$, $p_{ij}$ can be rewritten as
\begin{equation*}
p_{ij} = (1+o(1)) \frac{\beta}{|V_n|}\left(\frac{\max\{i,j\}}{|V_n|}\right)^{-2\eps/p} = (1+o(1)) \frac{\beta}{|V_n|}\max\left\{\frac{i}{|V_n|},\frac{j}{|V_n|}\right\}^{-2\eps/p}.
\end{equation*}
In accordance with the above expression, we consider the vertex space $\cV := (\cS, \mu, (\mathbf{x}_n)_{n\ge 1})$ where $\cS := (0,1]$, $\mu := \mathrm{Unif}(\cS)$ and, for every $n\ge 1$, $\mathbf{x}_n = (x_i^n)_{i=1}^{|V_n|}$ is defined so that, for all $i\in [|V_n|]$, $x_i^n\sim \mathrm{Unif}((\tfrac{i-1}{|V_n|}, \tfrac{i}{|V_n|}])$. Let us show that this vertex space is well-defined 
\lyu{by verifying}~\eqref{eq def mu}. Indeed, for any measurable set $A\subseteq (0,1]$, the fact that $\max_{i\in [|V_n|]} \mu(A\cap ((i-1)/|V_n|, i/|V_n|]) \to 0$ as $n\to \infty$ and the weak law of large numbers imply that
\begin{equation*}
\frac{1}{|V_n|} \sum_{i=1}^{|V_n|} \ind_{x_i^n\in A} = \frac{1}{|V_n|} \sum_{i=1}^{|V_n|} \ind_{x_i^n\in A\cap ((i-1)/|V_n|, i/|V_n|]} \xrightarrow{\mathbb P} \mu(A).
\end{equation*}
Furthermore, for every $n\ge 1$, consider a kernel
\begin{equation*}
\kappa_n: (x,y)\in (0,1]\times (0,1]\mapsto (1+o(1))\beta \max\left\{\frac{\lceil x|V_n|\rceil}{|V_n|},\frac{\lceil y|V_n|\rceil}{|V_n|}\right\}^{-2\eps/p}.
\end{equation*}
Then, up to an appropriate choice of the error term, the kernel $\kappa_n$ defines the probabilities $(p_{ij})_{i,j\in [|V_n|]}$. 
From the expression of $\kappa_n$, one could guess that $(\kappa_n)_{n\ge 1}$ is a graphical sequence on $\cV$ with limit
\begin{equation*}
\kappa: (x,y)\in (0,1]\times (0,1]\mapsto \beta \max\{x,y\}^{-2\eps/p}.
\end{equation*}
The next result shows this formally. 

\begin{lemma}\label{lem conv kappa}
The sequence $(\kappa_n)_{n\ge 1}$ is graphical on $\cV$ with limit $\kappa$.
\end{lemma}
\begin{proof}
First, let $x,y\in (0,1]$ and let $(x_n)_{n\ge 1}$ and $(y_n)_{n\ge 1}$ be two sequences converging to $x$ and $y$, respectively. We have 
\begin{equation*}
	\left|\frac{\lceil x_n |V_n|\rceil}{|V_n|} - x\right| = \left|\frac{\lceil x_n |V_n|\rceil}{|V_n|} - x_n + (x_n - x)\right|\le \frac{1}{|V_n|} + |x_n - x|\xrightarrow[n\to \infty]{} 0,
\end{equation*}
and the same holds for $(y_n)_{n\ge 1}$ and $y$. Distinguishing the cases $x \neq y$ and $x = y$ and using the continuity of the functions $t\in (0,1]\mapsto t^{-2\eps/p}$ and $\max\{\cdot, \cdot\}$ shows that $(\kappa_n)_{n\ge 1}$ and $\kappa$ satisfy~\eqref{eq sequences kappa}.

At the same time, $\kappa$ is continuous on $(0,1]\times (0,1]$ and 
\begin{align}
	\frac{1}{2}\int_{(0,1]\times (0,1]} \kappa(x,y) \mathrm{d}\mu(x)\mathrm{d}\mu(y) 
	&= \beta \int_{0< x < y < 1} y^{-2\eps/p} \mathrm{d}y \mathrm{d}x\nonumber\\
	&= \beta\int_{x=0}^1 \frac{1-x^{1-2\eps/p}}{1-2\eps/p} \mathrm{d}x = \beta p < \infty.\label{eq:check3}
\end{align}
Moreover, conditionally on $|V_n|$, the expected number of edges in $G_{BD}(1,n)$ is
\begin{equation*}
	\sum_{j=1}^{|V_n|} \sum_{i=1}^{j-1} \frac{\beta\lyu{+o(1)}}{|V_n|} \left(\frac{j}{|V_n|}\right)^{-2\eps/p} = \frac{\beta\lyu{+o(1)}}{|V_n|} \sum_{j=1}^{|V_n|} (j-1)\left(\frac{j}{|V_n|}\right)^{-2\eps/p}=  (1+o(1))\beta p \lyu{|V_n|},
\end{equation*}
which \lyu{together with~\eqref{eq:check3}} shows that~\eqref{eq conv kappa} is verified and finishes the proof of the lemma.
\end{proof}

\lyu{We are ready to prove Theorem~\ref{thm components intro} by combining Theorem~\ref{thm BJR components} and Lemma~\ref{lem conv kappa}.} 

\begin{proof}[Proof of Theorem~\ref{thm components intro}]
\lyu{Recall that in our setting $\cS = (0,1]$, $\dd \mu(z) = \dd z$ and $\kappa_n, \kappa$ are as defined before Lemma~\ref{lem conv kappa} (and, in particular, uniformly bounded from above).}
Thus, by Theorem~\ref{thm BJR components}, it is sufficient to prove that $\rho(\kappa) = \gamma$. 
On the one hand, recall that the multi-type Galton-Watson process $\mathfrak X_U$, where $U\sim \mu$ is the uniformly chosen root of the branching process, describes the local limit of the graph sequence $((G^\cV(n,\kappa_n),o_n))_{n\ge 1}$ where $o_n$ is chosen uniformly at random from $(x_i^n)_{i=1}^{|V_n|}$. 
Note that, in our case, both $(G_{BD}(1-\hat\delta_n, n))_{n\ge 1}$ and $(G_{BD}(1+\hat\delta_n, n))_{n\ge 1}$ are graph sequences on the same vertex sets and whose kernels both converge to $\kappa$. 
Therefore, each of these sequences converges locally to $\mathfrak X_U$. 
On the other hand, by Theorem~\ref{thrm:LWC intro}, $(G_n)_{n\ge 1}$ converges locally to $\cT$ and, by Lemma~\ref{lemma sandwich}, $G_{BD}(1-\hat\delta_n, n)\subseteq G_n\subseteq G_{BD}(1+\hat\delta_n, n)$ w.h.p.\ for a suitable coupling of the three graphs. Hence, $\mathfrak{L}(\mathfrak X) = \mathfrak{L}(\cT)$ and $\rho(\kappa) = \gamma$ is the survival probability of this branching process.
\end{proof}

\begin{remark}
One may construct the process $\cT$ from $\mathfrak X_U$ by applying the mapping $x\mapsto x^{2\eps/p}$ to the types of all particles in $\mathfrak X_U$. 
This yields $a_0 = U^{2\eps/p}$ for the type of the root in $\cT$. 
Also, consider the kernel $\wt \kappa(y,z):=\tfrac{\beta p}{2\eps}z^{(1-p)/(2\eps)}\max\{y,z\}^{-1}$ and observe that a particle $v$ of $\cT$ with type $a_v$ produces offspring independently of all other particles according to a Poisson point process on $(0,1]$ with intensity $\wt \kappa(a_v,z)\dd z$ (see Definition~\ref{def:wll}). 
Then, for any $0\leq a\leq b \leq 1$,
\begin{equation}\begin{aligned}
	\int_{a^{p/(2\eps)}}^{b^{p/(2\eps)}}\!\!\!\!\!\kappa(U,w)\dd \mu(w)
	&=\int_{a^{p/(2\eps)}}^{b^{p/(2\eps)}}\!\!\frac{\beta}{\max\big\{U^{2\eps/p},w^{2\eps/p}\big\}}\dd w\\
	&=\int_a^b\frac{\beta p}{2\eps} \frac{z^{(1-p)/(2\eps)}}{\max\{a_0,z\}}\dd z=\int_a^b\wt \kappa(a_0,z)\dd z.
	\end{aligned}\end{equation}
	Hence, the Poisson point process on $(0,1]$ with intensity $\kappa(U,w)\dd \mu(w)$ can be identified with the Poisson point process on $(0,1]$ with intensity $\wt\kappa(a_0,z)\dd z$ by applying the mapping $x\mapsto x^{2\eps/p}$ to each particle in the former point process. Applying this recursively for each particle in $\mathfrak X_U$, we obtain $\cT$.
\end{remark}

The irreducibility of $\kappa$ follows from the definition;
however, we need to work a bit to distinguish the subcritical regime from the supercritical. Concerning the operator $T_{\kappa}$, for all $f:(0,1]\to \mathbb R$ and $x\in (0,1]$, we have
\begin{align*}
(T_{\kappa} f)(x) 
&= \int_{y\in (0,1]} \kappa(x,y) f(y) \mathrm{d}\mu(y)\\
&= (1+o(1))\beta \left(\int_{y\in (0,x]} x^{-2\eps/p} f(y) \mathrm{d}y + \int_{y\in (x,1]} y^{-2\eps/p} f(y) \mathrm{d}y \right).
\end{align*}
Unfortunately, computing its norm directly leads to an optimisation problem that we are unable to solve explicitly. At best, we can provide bounds for the operator norm, as presented in the next lemma. Note that this result directly implies Proposition~\ref{prop beta_c intro}.

\begin{lemma}\label{lem bounds norm}
For every $p\in (1/2,1]$, we have 
\begin{equation}\label{eq:opineq}
\beta \sqrt{\frac{p(1+4p)}{2-p}}\le \beta \sup_{t\in (-1/2, \infty)}\sqrt{\frac{(1+2t)(2t^2+7t+4+1/p)}{(1+t)^2(t+1/p)(2t+2/p-1)}} \le ||T_{\kappa}|| \le \beta\sqrt{\frac{p}{1-p}}.
\end{equation} 
Moreover, the critical parameter $\beta_c = \beta_c(p)$ is a non-increasing continuous function of $p$ and satisfies
\begin{equation}\label{eq:betac}
\max\left\{\sqrt{\tfrac{1-p}{p}}, \tfrac{1}{4}\right\} \le \beta_c(p)\le \inf_{t\in (-1/2, \infty)}\left(\tfrac{(1+2t)(2t^2+7t+4+1/p)}{(1+t)^2(t+1/p)(2t+2/p-1)}\right)^{-1/2} \le\sqrt{\tfrac{2-p}{p(1+4p)}}.
\end{equation} 
\end{lemma}

\begin{proof}
First of all, given any $f: (0,1]\to \mathbb R^+$ satisfying $||f||_2\le 1$, the Cauchy-Schwarz inequality implies 
\begin{align}
||T_{\kappa} f||_2 
&=\; \sqrt{\int_{(0,1]} (T_{\kappa} f)(x)^2 \mathrm{d}x}\\
&=\; \sqrt{\int_{(0,1]} \left(\int_{(0,1]} \kappa(x,y) f(y) \mathrm{d}y\right)^2 \mathrm{d}x}\nonumber\le\; \sqrt{\int_{(0,1]^2} \kappa(x,y)^2 \mathrm{d}y \mathrm{d}x}.
\end{align}
For  $\eps \in (0,1/2]\setminus \{1/6\}$, the right-hand side can then be written as
\begin{equation}\begin{aligned}
	\beta\sqrt{2\int_{0<x<y<1} \max\{x,y\}^{-4\eps/p} \mathrm{d}x\mathrm{d}y}
	&\le\; \beta\sqrt{2\int_{(0,1]} \frac{1 - x^{1-4\eps/p}}{1-4\eps/p} \mathrm{d}x}\\
	&=\; \beta\sqrt{\frac{2}{1-4\eps/p} - \frac{2}{(1-4\eps/p)(2-4\eps/p)}},
	\end{aligned}\end{equation}
	which equals $\beta\sqrt{p/(1-p)}$. Similar computations for $\eps = 1/6$ result in $\beta \sqrt{2}$, and hence $||T_{\kappa} f||_2 \le \beta \sqrt{p/(1-p)}$ for all $\eps\in (0,1/2]$.
	
	Now, note that the first inequality in~\eqref{eq:opineq} comes from the fact that the left-most expression is equal to $\sqrt{\tfrac{(1+2t)(t^2+7t+4+1/p)}{(1+t)^2(t+1/p)(2t+2/p-1)}}$ when $t = 0$. Therefore, we only prove the more precise lower bound. Denote $f_t:x\in (0,1]\mapsto \sqrt{1+2t}\, x^t\in \mathbb R$ and note that $||f||_2 = 1$. We get
	\begin{align}
||T_{\kappa} f||_2 
&=\; \beta \sqrt{(1+2t)\int_{(0,1]} \left(\int_{(0,1]} \max\{x,y\}^{-2\eps/p} y^t\, \mathrm{d}y\right)^2\, \mathrm{d}x},
\end{align}
so that splitting the range of the inner integral into $(0,x]$ and $(x,1]$ yields
\begin{equation}\begin{aligned}
	\beta{}& \sqrt{(1+2t)\int_{(0,1]} \left(\int_{(0,x]} x^{-2\eps/p} y^t\, \mathrm{d}y + \int_{(x,1]} y^{t-2\eps/p}\, \mathrm{d}y\right)^2 \,\mathrm{d}x}\\
	&= \beta \sqrt{(1+2t)\int_{(0,1]} \left(\frac{x^{1+t-2\eps/p}}{1+t} + \frac{1 - x^{1+t-2\eps/p}}{1+t-2\eps/p}\right)^2\, \mathrm{d}x}.
	\end{aligned} \end{equation} 
	Expanding the squared term and integrating with respect to $x$ then finally yields
	\begin{equation}
\beta\sqrt{\frac{1+2t}{(1+t-2\eps/p)^2} \left(\left(\frac{2\eps/p}{1+t}\right)^2 \frac{1}{3+2t-4\eps/p} - \frac{4\eps/p}{(1+t)(2+t-2\eps/p)} + 1\right)}.
\end{equation} 

An immediate computation implies that
\begin{equation*}
\left(\tfrac{2\eps/p}{1+t}\right)^2 \tfrac{1}{3+2t-4\eps/p} - \tfrac{4\eps/p}{(1+t)(2+t-2\eps/p)} + 1 = \tfrac{(2t^2+7t+6-2\eps/p)(1+t-2\eps/p)^2}{(1+t)^2(2+t-2\eps/p)(3+2t-4\eps/p)},
\end{equation*}
which together with the fact that $2\eps/p = 2 - 1/p$ implies that
\begin{equation*}
||T_{\kappa} f||_2 = \beta\sqrt{\frac{(1+2t)(2t^2+7t+4+1/p)}{(1+t)^2(t+1/p)(2t+2/p-1)}}.
\end{equation*}

Optimising over $t\in (-1/2, \infty)$ proves~\eqref{eq:opineq}, and since $\beta_c$ is the value for which $||T_{\kappa}|| = 1$,~\eqref{eq:betac} follows immediately.

Furthermore, notice that, for any $p_1, p_2\in (1/2, 1)$ satisfying $p_1 < p_2$, the kernel for $p_2$ dominates the kernel for $p_1$ as a function over $(0,1]\times (0,1]$. 
From this, we conclude that, for every measurable non-negative function $f$ on $(0,1]$, we have that $||T_{\kappa}f||_2$ is a non-decreasing function of $p$, and hence $||T_{\kappa}||$ is also a non-decreasing function of $p$ by definition. 
This implies that $\beta_c$ is a non-increasing function of $p$, and thus $\beta_c(p)\ge \beta_c(1) = 1/4$ for all $p\in (1/2, 1]$, finishing the lower bound in~\eqref{eq:betac}.

\lyu{Finally, we show that $\beta_c$ is continuous.
Denote by $T_{\kappa, 1}$ and $T_{\kappa, \beta_c}$ the operators for $\beta = 1$ and $\beta = \beta_c$, respectively.
Then, Theorem~\ref{thm BJR components} implies that $\beta_c ||T_{\kappa,1}|| = ||T_{\kappa,\beta_c}|| = 1$
so it remains to show the continuity of the norm of the operator $T_{\kappa, 1}$ as a function of $p$.} Indeed, let $(p_n)_{n\ge 1}$ and $p$ be real numbers in $(1/2, 1]$ such that $p_n\to p$ as $n\to \infty$. Moreover, let $\cP = \{p_n\}_{n\ge 1}\cup \{p\}$ and let $f$ be a non-negative function in $L_2((0,1])$. Then, to show that the norm of $||T_{\kappa, 1} f||$ for $p$ is the limit of the norms of $||T_{\kappa, 1} f||$ for $p_n$ as $n\to \infty$, it is sufficient to show that
\begin{align*}
\int_{(0,1]} \left(\left(\int_{(0,1]} \max\{x,y\}^{-(2p-1)/p} f(y) \mathrm{d}y\right)^2 - \left(\int_{(0,1]} \max\{x,y\}^{-(2p_n-1)/p_n} f(y) \mathrm{d}y\right)^2\right) \mathrm{d}x
\end{align*}
tends to $0$ \lyu{uniformly over the choice of $f$. Observe that} the term in the integral is uniformly bounded by
\begin{equation*}
2 \max_{q\in \cP} \left(\int_{(0,1]} \max\{x,y\}^{-(2q-1)/q} f(y) \mathrm{d}y\right)^2\le 2\left(\int_{(0,1]} \max\{x,y\}^{-1} f(y) \mathrm{d}y\right)^2,
\end{equation*}
which is integrable since
\begin{equation*}
\int_{(0,1]} \left(\int_{(0,1]} \max\{x,y\}^{-1} f(y) \mathrm{d}y\right)^2 \dd x
\end{equation*}
is at most $||T_{\kappa, 1}||^2 < \infty$ for $p = 1$. By the dominated convergence theorem, it suffices to show that
\begin{align*}
\left(\int_{(0,1]} \max\{x,y\}^{-(2p-1)/p} f(y) \mathrm{d}y\right)^2 - \left(\int_{(0,1]} \max\{x,y\}^{-(2p_n-1)/p_n} f(y) \mathrm{d}y\right)^2
\end{align*}
tends to 0 as $n\to \infty$ for every $x\in (0,1]$, or equivalently that
\begin{align*}
\int_{(0,1]} (\max\{x,y\}^{-(2p-1)/p} - \max\{x,y\}^{-(2p_n-1)/p_n}) f(y) \mathrm{d}y
\end{align*}
tends to 0 as $n\to \infty$ for every $x\in (0,1]$. The above expression equals
\begin{align*}
\int_{(0,x]} (x^{-(2p-1)/p} - x^{-(2p_n-1)/p_n}) f(y) \mathrm{d}y + \int_{(x,1]} (y^{-(2p-1)/p} - y^{-(2p_n-1)/p_n}) f(y) \mathrm{d}y.
\end{align*}
While the first term in the sum is a multiple of $x^{-(2p-1)/p} - x^{-(2p_n-1)/p_n}$, which clearly tends to $0$ when $n\to \infty$ for any $x\in (0,1]$, the second term is dominated from above by
\begin{equation*}
\max_{y\in [x,1]}\{y^{-(2p-1)/p} - y^{-(2p_n-1)/p_n}\} f(y)\leq \big(x^{-(2p-1)/p}+x^{-(2p_n-1)/p_n}\big)f(y),
\end{equation*}
which is integrable over the interval $(x,1]$. Applying the dominated convergence theorem once again and using that $y^{-(2p-1)/p} - y^{-(2p_n-1)/p_n}$ tends to zero as $n\to \infty$ for every $y\in (x,1]$ finishes the proof of the lemma.
\end{proof}

Finally, note that the limit kernel $\kappa$ we obtained is unbounded and hence we could not apply Theorems~\ref{thm BJR components} and~\ref{thm BJR typical distance} in their full generality. Nevertheless, $\kappa(x,y)$ remains bounded if at least one of $x$ and $y$ is bounded away from 0. We use this observation to prove Theorem~\ref{thm approximate intro}.

\begin{proof}[Proof of Theorem~\ref{thm approximate intro}]
Denote $V^\lambda_n = V(G^\lambda_n)$, and let $G^{\lambda}_{BD}(1,n) = G_{BD}(1,n)[V^\lambda_n]$ be the graph induced from $G_{BD}(1,n)$ by the vertex set $V^\lambda_n$. Define the kernels
\begin{equation*}
\kappa^\lambda_n: (x,y)\in [\lambda, 1]\times [\lambda, 1]\mapsto \kappa_n(x,y).
\end{equation*}
Note that the only difference between $\kappa^\lambda_n$ and $\kappa_n$ is the domain of definition. 
One may readily verify that $\kappa^\lambda_n$ serves to define the graph $G^\lambda_{BD}(1,n)$ in the same way as $\kappa_n$ serves to define $G_{BD}(1,n)$. 
Also, with minor modifications, the proof of Lemma~\ref{lem conv kappa} shows that $(\kappa_n^\lambda)_{n\ge 1}$ converges to $\kappa^\lambda: (x,y)\in [\lambda, 1]\times [\lambda, 1]\mapsto \kappa(x,y)$. Since $\kappa^\lambda$ is a bounded kernel, the second point of Theorem~\ref{thm approximate intro} follows immediately from Part~(i) of Theorem~\ref{thm BJR components}. 

Now, we prove the third point of Theorem~\ref{thm approximate intro} for $\zeta = ||T_{\kappa}||^{-1}$. By Theorem~\ref{thm BJR typical distance}, the only thing we need to show is that $||T_{\kappa^{\lambda}}||$ converges to $||T_{\kappa}||$ as $\lambda$ decreases to 0. 
Fix any $\delta\in (0,1)$ and let $f$ be a square integrable positive function over $(0,1]$ such that $||T_{\kappa}||\le (1+\delta) ||T_{\kappa} f||_2$. This means that, in particular,
\begin{equation*}
||T_{\kappa} f||_2^2 = \int_{(0,1]} \left(\int_{(0,1]} \kappa(x,y) f(y) \mathrm{d}y\right)^2 \mathrm{d}x = \int_{(0,1]^3} \kappa(x,y_1)\kappa(x,y_2) f(y_1) f(y_2) \mathrm{d}y_2 \mathrm{d}y_1 \mathrm{d}x.
\end{equation*}
Now, by the monotone convergence theorem, we obtain that
\begin{align*}
||T_{\kappa^\lambda} f||_2^2 
&=\; \int_{[\lambda,1]^3} \kappa^{\lambda}(x,y_1)\kappa^{\lambda}(x,y_2) f(y_1) f(y_2) \mathrm{d}y_2 \mathrm{d}y_1 \mathrm{d}x\\ 
&=\; \int_{(0,1]^3} \ind_{x\in [\lambda, 1]}\ind_{y_1\in [\lambda, 1]}\ind_{y_2\in [\lambda, 1]} \kappa(x,y_1)\kappa(x,y_2) f(y_1) f(y_2) \mathrm{d}y_2 \mathrm{d}y_1 \mathrm{d}x
\end{align*}
converges to $||T_{\kappa} f||_2^2$ when $\lambda\to 0$. 
Hence, for every sufficiently small $\lambda > 0$, we have that $||T_{\kappa}||\le (1+2\delta)||T_{\kappa^\lambda} f||_2$. 
Moreover, for every $\lambda > 0$, we have that the kernel $\kappa^{\lambda}$ may be extended to $\hat \kappa^{\lambda}: (x,y)\in (0,1]\times (0,1]\mapsto \ind_{x\in [\lambda, 1]}\ind_{y\in [\lambda,1]} \kappa(x,y)$, which is dominated by $\kappa$ for any $\lambda\in (0,1]$, so $||T_{\kappa^{\lambda}}|| = ||T_{\hat \kappa^{\lambda}}||\le ||T_\kappa||$. Hence, $||T_{\kappa^{\lambda}}||$ converges to $||T_\kappa||$ as $\lambda\to 0$, which together with Theorem~\ref{thm BJR typical distance} proves the third point.

Finally, we prove the first point of Theorem~\ref{thm approximate intro}. First of all, note that, conditionally on the event $\{a_0 > \lambda\}$, the local limit of $G^\lambda_n$ exists and it constructed from $(\cT, 0)$ by removing all vertices whose type is at most $\lambda$ together with their descendants. 

Let us denote the limit of $(G^\lambda_n)_{n\ge 1}$ by $(\cT^{\lambda}, 0)$ and its survival probability by $\gamma^\lambda$. Then, for any fixed $r\ge 1$, \lyu{a first moment argument} shows that the probability that the root 0 of $\cT$ has a vertex with age in the interval $[0, \lambda]$ at distance at most $r$ tends to 0 as $\lambda\to 0$. Therefore, for every finite rooted tree $(T, o)$ we have that 
\begin{equation*}
\P{(\cT^{\lambda}, 0) \cong (T,o)} \to \P{(\cT, 0) \cong (T,o)} \text{ as } \lambda\to 0.
\end{equation*}
Recall $\gamma$ from~\eqref{eq:gamma} as well as that $\gamma=\rho(\kappa)$ by the proof of Theorem~\ref{thm components intro}. Also, recall that $\cG$ is the set of all finite rooted trees. We deduce that
\begin{align*}
|\gamma^\lambda - \gamma| 
&=\; \Big|\sum_{(T,o)\in \cG} \P{(\cT^{\lambda}, 0) \cong (T,o)} - \P{(\cT, 0) \cong (T,o)}\Big|\\
&\le \; \sum_{(T,o)\in\cG} \Big|\P{(\cT^{\lambda}, 0) \cong (T,o)} - \P{(\cT, 0) \cong (T,o)}\Big|.
\end{align*}
Then, by the reverse Fatou's lemma, 
\begin{equation*}
\limsup_{\lambda\to 0} |\gamma^{\lambda}-\gamma| \le \sum_{(T,o)\in\cG} \limsup_{\lambda\to 0} \Big|\P{(\cT^{\lambda}, 0) \cong (T,o)} - \P{(\cT, 0) \cong (T,o)}\Big| = 0,
\end{equation*}
and therefore $\gamma^\lambda \to \gamma=\rho(\kappa)$ as $\lambda\to 0$.

Now, recall the coupling constructed in Lemma~\ref{lemma sandwich} and the sequence $(\hat \delta_n)_{n\ge 1}$. Note that $G_{BD}^{\lambda}(1+\hat \delta_n, n)$ may be constructed from $G^\lambda_n$ by adding every edge $\{v_i,v_j\}$ satisfying $v_i, v_j\in V_n\setminus V_{\lambda n}$ to $G^\lambda_n$ with probability 
$$\cO\left(\frac{\hat \delta_n}{n} \left(\frac{\max\{i,j\}}{|V_n|}\right)^{-2\eps/p}\right),$$
while $G_{BD}^{\lambda}(1-\hat \delta_n, n)$ may be constructed from $G_n$ by removing every edge of $G_n$ with probability $\cO(\hat \delta_n)$. 
At the same time, for any fixed $r\in \mathbb N$, by Theorem~\ref{thrm:LWC intro} and the definition of $(\cT, 0)$, w.h.p.\ the ball of radius $r$ around a uniformly chosen vertex in $G_n$ contains $\cO(\hat \delta_n^{-1/2})$ edges. 
Moreover, 
w.h.p.\ the $r^{\text{th}}$ neighbourhood of a uniformly chosen vertex in $V^\lambda_n$ is the same in both $G_{BD}^{\lambda}(1-\hat \delta_n, n)$ and $G_{BD}^{\lambda}(1+\hat \delta_n, n)$. 
We conclude that each of $G_{BD}^{\lambda}(1-\hat \delta_n, n)$, $G^\lambda_n$ and $G_{BD}^{\lambda}(1+\hat \delta_n, n)$ converges locally to $(\cT^{\lambda}, 0)$. Thus, by a similar argument as in the proof of Theorem~\ref{thm components intro}, we conclude that $\rho(\kappa^\lambda)=\gamma^\lambda$. 
At the same time, for every fixed $\delta' > 0$, we have that, by choosing $\lambda$ sufficiently small, $|\rho(\kappa^{\lambda}) - \rho(\kappa)|\le \tfrac{\delta'}{4\eps}$ and therefore w.h.p.\ $|\cC_1(G^{\lambda}_n)|/|V_n^\lambda|\in [\rho(\kappa)-\tfrac{\delta'}{4\eps}, \rho(\kappa)+\tfrac{\delta'}{4\eps}]$. 
By combining Lemma~\ref{lemma:beta} and Lemma~\ref{lemma:condconc} applied with $S = \emptyset$, we conclude that $|V^\lambda_n|/n$ converges in probability to $2\eps (1-\lambda^{p/(2\eps)})$. 
This proves the first part of Theorem~\ref{thm approximate intro} since $2\eps\rho(\kappa)=2\eps\gamma$ and moreover w.h.p.\
$$2\eps \gamma  - \delta'\le \frac{|V_n^\lambda|}{n} \left(\gamma-\frac{\delta'}{4\eps}\right)\le \frac{|V_n^\lambda|}{n} \left(\gamma +\frac{\delta'}{4\eps}\right)\le 2\eps \gamma + \delta'$$
for every sufficiently small $\lambda$. 
\end{proof}

\section{The maximum degree: proof of Theorem~\ref{thrm:max intro}}\label{sec:max}

In this section, we study the maximum in-degree, out-degree and degree (that is, in-degree plus out-degree) of the graph $\vv{G}_n$, as stated in Theorem~\ref{thrm:max intro}. 
We remark that the evolution of the vertex degrees is severely influenced by our vertex deletion mechanism, as discussed after Theorem~\ref{thrm:max intro}. 

We let 
\begin{equation}
d^+_n(v_i):=\sum_{j=i+1}^n \ind_{\{v_i,v_j\}\in \vv{G}_n}, \quad \text{and}\quad  d_n^-(v_i):=\sum_{j=1}^{i-1}\ind_{\{v_j,v_i\}\in \vv{G}_n},
\end{equation} 
denote the in-degree and out-degree of vertex $v_i$ in $\vv{G}_n$, respectively, and let $d_n^s(v_i):=d^+_n(v_i)+d^-_n(v_i)$ denote the degree of vertex $v_i$. As in the previous section, Lemma~\ref{lemma sandwich} allows us to ``sandwich'' the vertex degrees of $\vv{G}_n$ between those of the graphs $\vv{G}_{1,n} := \vv{G}_{BD}(1-\hat\delta_n,n)$ and $\vv{G}_{2,n} := \vv{G}_{BD}(1+\hat\delta_n,n)$. For each of $\square\in \{s, +, -\}$, we let $d_{1,n}^{\square}(v_i)$ and $d_{2,n}^{\square}(v_i)$ denote the degree (if $\square = s$), in-degree (if $\square = +$) or out-degree (if $\square = -$) of the vertex $v_i$ in $\vv{G}_{1,n}$ and in $\vv{G}_{2,n}$, respectively.

The proof of Theorem~\ref{thrm:max intro} makes use of the following preliminary result, which is similar in spirit to Lemma~1 from~\cite{DevLu95}, combined with the ``sandwiching'' of the graphs $\vv{G}_{1,n}$, $\vv{G}_n$, and $\vv{G}_{2,n}$. The main conclusion is that, roughly speaking, the maximum in-degree (respectively out-degree) of $\vv{G}_n$ is approximately $k_n$ if the expected number of vertices of in-degree (respectively out-degree) $k_n$ is about 1.

\begin{lemma}\label{lemma:maxsum}
Let $(k_n)_{n\ge 1}$ be a sequence of positive numbers. Then, for $\square\in\{+,-\}$,
\begin{equation}
\P{\max_{i\in [|V_n|]} d_n^\square(v_i)\geq k_n\,\bigg|\, |V_n|}=
\begin{cases} 0 &\mbox{if } \lim_{n\to \infty} \sum_{i=1}^{|V_n|} \P{d^{\square}_{2,n}(v_i)\geq k_n\,\mid\, |V_n|}=0,\\
	1
	&\mbox{if } \lim_{n\to \infty} \sum_{i=1}^{|V_n|} \P{d^{\square}_{1,n}(v_i)\geq k_n\,\mid\, |V_n|}=\infty.
\end{cases}
\end{equation} 
Moreover,
\begin{equation*}
\P{\max_{i\in [|V_n|]} d_n^s(v_i)\geq k_n\,\bigg|\, |V_n|}= 0\quad \mbox{if } \lim_{n\to \infty} \sum_{i=1}^{|V_n|} \P{d^s_{2,n}(v_i)\geq k_n\mid |V_n|}=0.
\end{equation*}
\end{lemma}

\begin{proof}
In this proof, everything is done conditionally on $|V_n|$. 
First, by Lemma~\ref{lemma sandwich}, we know that w.h.p.\ $\vv{G}_{1,n}\subseteq \vv{G}_n\subseteq \vv{G}_{2,n}$. 
Hence, proving that the expected number of vertices of degree (or in-degree, or out-degree) $k_n$ in $\vv{G}_{2,n}$ is $o(1)$ implies that w.h.p.\ $\vv{G}_n$ contains no vertex of degree (or in-degree, or out-degree) $k_n$. 
Second, for both $\square\in \{+,-\}$ and every $i\in [|V_n|]$, define the event $\cD_{i,n}^\square:=\{d_{1,n}^{\square}(v_i)\geq k_n\}$. 
By independence of the states of the edges in $\vv{G}_{1,n}$ one may conclude that the events $\cD_{i,n}^\square$ are independent for different values of $i$. Thus, when 
\begin{equation}
\lim_{n\to\infty}\sum_{i=1}^{|V_n|}\P{\cD_{i,n}^\square\,\Big|\, |V_n|}=\lim_{n\to\infty}\sum_{i=1}^{|V_n|}\P{d^\square_{1,n}(v_i)\geq k_n\,\Big|\, |V_n|}=\infty
\end{equation} 
holds, it follows that
\begin{equation*}
\mathrm{Var}\Big(\sum_{i=1}^{|V_n|}\ind_{\cD_{i,n}^\square}\;\Big|\; |V_n|\Big) = \sum_{i=1}^{|V_n|}\mathrm{Var}(\ind_{\cD_{i,n}^\square}\mid |V_n|)\le \mathbb E\Big[\sum_{i=1}^{|V_n|}\ind_{\cD_{i,n}^\square}\;\Big|\; |V_n|\Big]\ll \mathbb E\Big[\sum_{i=1}^{|V_n|}\ind_{\cD_{i,n}^\square}\;\Big|\; |V_n|\Big]^2.
\end{equation*}
We conclude via Chebyshev's inequality that at least least one of $(\cD_{i,n}^\square)_{i=1}^{|V_n|}$ holds w.h.p.\ and then $\vv{G}_{1,n}$ (and thus also $\vv{G}_n$) contains a vertex of in-degree or out-degree $k_n$.
\end{proof}

\begin{proof}[Proof of Theorem~\ref{thrm:max intro}]
In this proof, everything is done conditionally on $|V_n|$ and the event $\cQ_n = \{||V_n|-2\eps n|\le n^{2/3}\}$, which holds w.h.p.\ by Lemma~\ref{lemma:condconc}. We remark that in several places in this proof $(1\pm \hat\delta_n)\beta$ is somewhat abusively replaced by $\beta$ only; note that this abuse does not influence the proof since the limits in~\eqref{eq:2ndords} and~\eqref{eq:2ndordmin} are both continuous as functions of $\beta$. Define $W_0$ as the inverse of the function $f:x\in [-1,\infty)\to x\e^x\in [-1/\e, \infty)$ ($W_0$ is also known as the main branch of the Lambert $W$ function). In particular, $\e^{W_0(\log n)}=\log n/W_0(\log n)$ and, moreover, by Theorem~2.7 in~\cite{HooHas08} (or a simple asymptotic analysis), we also have that
\begin{equation}\label{eq:w0asymp2}
W_0(\log n)=\log\log n-\log\log\log n+o(1).
\end{equation} 
This implies that
\begin{equation}\begin{aligned}\label{eq W_0 division}
	\frac{\log n}{W_0(\log n)}
	&=\;\frac{\log n}{\log\log n-\log\log\log n+o(1)}\\
	&=\;\frac{\log n}{\log\log n}+\frac{\log n\log\log\log n}{(\log\log n)^2}+o\left(\frac{\log n}{(\log\log n)^2}\right).
	\end{aligned}\end{equation} 
	Hence, to obtain~\eqref{eq:2ndords}, it suffices to prove that
	\begin{equation}\label{eq:cbetap}
\frac{\max_{i\in [|V_n|]}d_n^{\square}(v_i) -\log n/W_0(\log n)}{\log n/(\log\log n)^2}\toinp C_{\beta,p} :=  1+\log \Big(\frac{\beta p}{1-p}\Big)
\end{equation} 
for $\square\in\{s,+\}$, \lyu{and similarly, to obtain~\eqref{eq:2ndordmin}, it suffices to prove that
\begin{equation}\label{eq:cbetap1}
	\frac{\max_{i\in [|V_n|]}d_n^{-}(v_i) -\log n/W_0(\log n)}{\log n/(\log\log n)^2}\toinp \widehat C_{\beta} :=  1+\log\beta.
	\end{equation}}
	In fact, we show that, for every $\mu > 0$, we have that w.h.p.
	\begin{align}
C_{\beta, p} - \mu&\le \frac{\max_{i\in [|V_n|]}d_{1,n}^+(v_i) -\log n/W_0(\log n)}{\log n/(\log\log n)^2}\label{eq ineq lb}\\
&\le\; 
\frac{\max_{i\in [|V_n|]}d_{2,n}^s(v_i) -\log n/W_0(\log n)}{\log n/(\log\log n)^2}\le C_{\beta, p} + \mu\label{eq ineq ub},
\end{align}
\lyu{and
\begin{align}
	\widehat C_{\beta} - \mu&\le \frac{\max_{i\in [|V_n|]}d_{1,n}^-(v_i) -\log n/W_0(\log n)}{\log n/(\log\log n)^2}\label{eq ineq lb1}\\
	&\le\; 
	\frac{\max_{i\in [|V_n|]}d_{2,n}^-(v_i) -\log n/W_0(\log n)}{\log n/(\log\log n)^2}\le \widehat C_{\beta} + \mu.\label{eq ineq ub1}
	\end{align}}
	We first show \lyu{the upper bounds~\eqref{eq ineq ub} and~\eqref{eq ineq ub1}, starting with the former one}. We set 
	\[a_n:=\log n/W_0(\log n)+(C_{\beta, p}+\mu)\log n/(\log\log n)^2.\] 
	Then, for every $i\in [|V_n|]$ and $t\ge 0$,
	\begin{align*}
&\P{d_{2,n}^s(v_i)\geq a_n}\\
\leq\; 
&\E{\e^{td_{2,n}^s(v_i)}}\e^{-ta_n}\\
\le\; 
&\e^{-ta_n} \prod_{j=1}^{i-1}\Big(1+\big(\e^t-1)(2\eps)^{-(1-p)/p}\tfrac{\beta}{n} \Big(\tfrac in\Big)^{-2\eps/p}\Big)\prod_{j=i+1}^{|V_n|} \Big(1+\big(\e^t-1\big) (2\eps)^{-(1-p)/p} \tfrac{\beta}{n} \Big(\tfrac jn\Big)^{-2\eps/p}\Big).
\end{align*}
Including the term $j=i$ in the first product and using that $1+x\leq \e^x$ for all $x\in\R$ and $||V_n|-2\eps n|\le n^{2/3}$ implies that
\begin{align}
&\P{d_{2,n}^s(v_i)\geq a_n}\\
\leq\; 
&\exp\left(-ta_n + (\e^t - 1)\beta \left(\Big(\tfrac{i}{2 \eps n}\Big)^{1-2\eps/p} + \tfrac{(1+\cO(n^{-1/3}))((2\eps n)^{1-2\eps/p} - i^{1-2\eps/p})}{(1-2\eps/p) (2\eps n)^{1-2\eps/p}}\right)\right)\nonumber\\
\leq\; 
&\exp\left(-ta_n + (1+\cO(n^{-1/3})) (\e^t - 1)  \frac{\beta p}{1-p}\right).\label{eq need in (3)}
\end{align}
By setting $t = \log\left(\tfrac{(1-p)a_n}{\beta p}\right)$ and using that $\log\left(\tfrac{1-p}{\beta p}\right) = -\log\left(\tfrac{\beta p}{1-p}\right)$, we arrive at the upper bound
\begin{equation}\label{eq:chernub}
\P{d_{2,n}^s(v_i)\geq a_n}\leq \exp\Big(-a_n\log a_n +C_{\beta,p} a_n - \frac{\beta p}{1-p} + o(1)\Big).
\end{equation} 
By using~\eqref{eq:w0asymp2},~\eqref{eq W_0 division} and the equality $\e^{W_0(\log n)}=\log n/W_0(\log n)$, we deduce that
\begin{equation}\label{eq log k_n}
\begin{split}
	& a_n\log a_n\\
	=&
	\Big(\frac{\log n}{W_0(\log n)} + \frac{(C_{\beta, p}+\mu)\log n}{(\log\log n)^2}\Big)\log\Big(\frac{\log n}{W_0(\log n)} + \frac{(C_{\beta, p}+\mu)\log n}{(\log\log n)^2}\Big)\\
	=&
	\Big(1+\frac{(C_{\beta, p}+\mu+o(1))W_0(\log n)}{(\log\log n)^2}\Big) \frac{\log n}{W_0(\log n)} \Big(\log(\e^{W_0(\log n)}) + \frac{(C_{\beta, p}+\mu+o(1))W_0(\log n)}{(\log\log n)^2}\Big)\\
	=& 
	\log n + \frac{(C_{\beta, p} + \mu + o(1))\log n}{\log\log n} = \log n + (C_{\beta, p} + \mu + o(1))a_n.
\end{split}
\end{equation}
Thus, combining~\eqref{eq:chernub}, \eqref{eq log k_n} and a union bound over all $|V_n|\le n$ vertices in $\vv{G}_{2,n}$, we get that
\begin{align*}
\P{\max_{i\in [|V_n|]} d_{2,n}^s(v_i)\geq a_n}
&\leq\; n\exp\Big(-a_n\log a_n + a_n \log\Big(\frac{\beta p}{1-p}\Big) + a_n - \frac{\beta p}{1-p} + o(1)\Big)\\
&\le\; \exp(-(\mu + o(1))a_n) = o(1).
\end{align*}
\lyu{We turn to~\eqref{eq ineq ub1}. Define}
\begin{equation}\label{eq:chat}
c_n:=\log n/W_0(\log n)+(\widehat C_{\beta}+\mu)\log n/(\log\log n)^2.
\end{equation}
For every $j\in [|V_n|]$, $d^-_{2,n}(v_j)$ is a binomial random variable with distribution $\mathrm{Bin}(j-1, p_{1j})$. 
Observe also that, on the event $\cQ_n$ (defined in~\eqref{eq:An}), \bas{the mean $m_{n,j}$ of $d^-_{2,n}(v_j)$~is
\be \label{eq:dminmean}
m_{j,n}=(j-1)p_{1j} = \lyu{\Big(1-\frac{1}{j}\Big)}\beta (2\eps)^{-(1-p)/p}\Big(\frac{j}{n}\Big)^{1-2\eps/p}.
\ee 
Since $1-2\eps/p>0$ for any $\eps\in(0,1/2)$, the mean is $o(1)$ for $j=o(n)$ and of constant order for $j$ of order $n$.} \lyu{By Chernoff's inequality with $\phi(x)=(1+x)\log(1+x)-x$ (see Lemma~\ref{lem chern}(i)), we conclude} that
\begin{equation}\ba \label{eq d_2,n minus}
\P{d^-_{2,n}(v_j)\ge c_n} &\le \bas{\exp\Big(-m_{j,n}\phi\Big(\frac{c_n}{m_{j,n}}-1\Big)\Big)}\\ 
&=\exp\big(-c_n\log\big(c_n/m_{j,n}\big)+c_n-m_{j,n}\big).
\ea\end{equation}
\bas{As $m_{j,n}=\cO(1)$ uniformly in $j$, the \lyu{third} term in the exponential can be ignored. The first term is increasing in $m_{j,n}$. Furthermore, $m_{j,n}\leq m_{|V_n|,n}=\beta +o(1)$ for all $j\in [|V_n|]$ and all $n$ sufficiently large. Hence, uniformly in $j$, we can bound
\be 
\P{d^-_{2,n}(v_j)\ge c_n} =\cO\Big( \exp\big(-c_n\log(c_n/\beta)+(1+o(1))c_n\big)\Big).
\ee 
}
Now, since 
\begin{equation}
c_n=\frac{\log n}{\log\log n}+\frac{\log n \log\log\log n}{(\log\log n)^2}+(\widehat C_{\beta}+\mu+o(1))\frac{\log n}{(\log\log n)^2},
\end{equation} 
it follows that 
\be
-c_n\log(c_n/\beta)+c_n(1+o(1))= -\log n-(\mu+o(1))\frac{\log n}{\log\log n}.
\ee
It then follows from a union bound that
\begin{equation}\label{eq ub eq:2ndordmin}
\P{\max_{i\in [|V_n|]} d_{2,n}^-(v_i) \ge c_n}\le \bas{\sum_{j=1}^{|V_n|} \P{d_{2,n}^-(v_j) \ge c_n}} = o(1).
\end{equation}
Now, we \lyu{come back to the proof of the lower bound~\eqref{eq ineq lb}; note that~\eqref{eq ineq lb1} is shown in the same manner, so we only mark the necessary modifications along the way.}
Recall $C_{\beta,p}$ and $\widehat C_{\beta}$ from~\eqref{eq:cbetap} and~\eqref{eq:cbetap1}, respectively, and define
\begin{align}\label{eq:bn}
b_n^+&:=\left\lfloor\frac{\log n}{W_0(\log n)}+\frac{(C_{\beta, p}-\mu)\log n}{(\log\log n)^2}\right\rfloor,\quad  b_n^-:=\left\lfloor\frac{\log n}{W_0(\log n)}+\frac{(\widehat C_{\beta}-\mu)\log n}{(\log\log n)^2}\right\rfloor, \\  e_n&:=\left\lceil\frac{\mu\log n}{2(\log\log n)^2}\right\rceil.
\end{align}
Fix $i\in [|V_n|]$ and set $\cS_i = [i+1, |V_n|]$ in the lower bound for $\square=+$, and $\cS_i = [i-1]$ in the lower bound for $\square=-$. Then, for every $j\in \cS_i$, let $I_j$ be the indicator random variable of the event $\{v_i, v_j\}\in \vv{G}_{1,n}$. Then, by~\eqref{eq def p_ij}, we have
\begin{equation*}
\lambda_j := \P{I_j = 1} = 
\begin{cases}
	& (1+\hat\delta_n)\frac{\beta}{|V_n|}\Big(\frac{j}{|V_n|}\Big)^{-2\eps/p} \quad \text{for } \square=+,\\
	& (1+\hat\delta_n)\frac{\beta}{|V_n|}\Big(\frac{i}{|V_n|}\Big)^{-2\eps/p} \quad \text{for } \square=-.\\
\end{cases}
\end{equation*}
Let $(P_j)_{j\in \cS_i}$ be independent Poisson random variables where $P_j$ has mean $\lambda_j$. Then, by Lemma~\ref{lemma:poibercoupling}, we can couple $I_j$ and $\ind_{\{P_j\leq 1\}}P_j$ so that, for each $j\in [n]$, almost surely $I_j\geq \ind_{\{P_j\leq 1\}}P_j$. Setting
\begin{equation*}
W_i := \sum_{j\in \cS_i} P_j\quad  \text{  and  }\quad Z_i := \sum_{j\in\cS_i} \ind_{\{P_j>1\}} P_j,
\end{equation*}
for both $\square\in \{+,-\}$ we get that 
\begin{equation}\label{eq:poic}
d_{1,n}^{\square}(v_i)=\sum_{j\in \cS_i} I_j\geq \sum_{j\in \cS_i} \ind_{\{P_j\le1\}}P_j = W_i-Z_i. 
\end{equation} 
It then follows that
\begin{equation}\label{eq:sumlb}
\sum_{i=1}^{|V_n|} \P{d_{1,n}^{\square}(v_i)\geq b^\square_n}\geq \sum_{i=1}^{|V_n|}\P{W_i\geq b_n^{\square}+e_n}-\sum_{i=1}^{|V_n|}\P{Z_i\geq e_n}. 
\end{equation} 
Next, we show that
\begin{equation}\label{eq:sumclaims}
\lim_{n\to\infty} \sum_{i=1}^{|V_n|}\P{W_i\geq b_n^{\square}+e_n}=\infty, \quad\text{and}\quad   \lim_{n\to\infty}\sum_{i=1}^{|V_n|} \P{Z_i\geq e_n}=0, 
\end{equation}  
which implies that 
\begin{equation}\label{eq:resultclaims}
\lim_{n\to\infty} \sum_{i=1}^{|V_n|} \P{d_{1,n}^{\square}(v_i)\geq b_n^{\square}}=\infty.
\end{equation} 
\lyu{Then, by using Lemma~\ref{lemma:maxsum}, the lower bounds~\eqref{eq ineq lb} and~\eqref{eq ineq lb1} follow.}

We start with the left part of~\eqref{eq:sumclaims}. Recall that $W_i$ is a sum of independent Poisson random variables with means $\lambda_j$. As a result, $W_i\sim \lyu{\text{Po}}(\sum_{j=\cS_i} \lambda_j)$, so
\begin{equation}
\P{W_i\geq b_n^{\square}+e_n}\geq \P{W_i=b_n^{\square}+e_n}=\e^{-\sum_{j\in \cI_i}\lambda_j}\Big(\sum_{j\in \cS_i} \lambda_j\Big)^{b_n^{\square}+e_n}\frac{1}{(b_n^{\square}+e_n)!}. 
\end{equation} 
By Stirling's formula 
$$(b_n^{\square}+e_n)! = \cO\Big(\left(\tfrac{b_n^{\square}+e_n}{\e}\right)^{b_n^{\square}+e_n}\sqrt{2\pi(b_n^{\square}+e_n)}\Big) = \cO\Big(\left(\tfrac{b_n^{\square}+e_n}{\e^{1+o(1)}}\right)^{b_n^{\square}+e_n}\Big),$$
which yields
\begin{equation}\label{eq stirling 1}
\begin{split}
	\mathbb P{}&\Big(W_i\geq b_n^{\square}+e_n\Big) \\
	&=\Omega\Big(\Big(\frac{\e^{1+o(1)}}{b_n^{\square}+e_n}\sum_{j\in \cS_i} \lambda_j\Big)^{b_n^{\square}+e_n}\Big)\\ 
	&=\exp{}\Big(\Big(1+\log\Big(\sum_{j\in \cS_i} \lambda_j\Big)\Big)(b_n^{+}+e_n)-(b_n^{+}+e_n)\log(b_n^{+}+e_n)+o(b_n^{+}+e_n)\Big).    
\end{split}
\end{equation}

Now, in the case when $\square = +$, fix an integer $i\in [\sqrt{n}, n/\log n]$. Then,
\begin{equation}\label{eq:lambdaasymp}
\sum_{j=i+1}^{|V_n|} \lambda_j=(1+o(1))\frac{\beta p}{1-p}\Big(1-\Big(\frac{i}{|V_n|}\Big)^{(1-p)/p}\Big)=(1+o(1))\frac{\beta p}{1-p}
\end{equation} 
since $|V_n| = \omega(n/\log n)$ (recall that we work on the event $\cQ_n$). By combining this with~\eqref{eq stirling 1} and~\eqref{eq:lambdaasymp}, we obtain that
\begin{align*}
\sum_{i=1}^{|V_n|}{}&\P{W_i\geq b_n^{+}+e_n}\\
&\ge\sum_{i=\lceil \sqrt{n}\rceil}^{\lfloor n/\log n\rfloor}\P{W_i\geq b_n^{+}+e_n}\\
&\ge\frac{n}{2\log n}\exp{} \Big(\Big(1+\log\Big(\frac{\beta p}{1-p}\Big)\Big)(b_n^{+}+e_n)-(b_n^{+}+e_n)\log(b_n^{+}+e_n)+o(b_n^{+}+e_n)\Big).
\end{align*}
By a similar analysis to the one done for~\eqref{eq log k_n}, \[(b_n^{+}+e_n)\log(b_n^{+}+e_n) = \log n + (C_{\beta, p}-\mu/2+o(1))\log n/\log\log n,\] 
and combined with $(1+\log(\tfrac{\beta p}{1-p})) (b_n^{+}+e_n) = (C_{\beta, p}+o(1))\log n/\log\log n$, this shows the left part of~\eqref{eq:sumclaims} for the case $\square = +$.

If $\square = -$, fix an integer $i\in [|V_n|-n/\log n, |V_n|]$. Then, $\sum_{j=1}^{i-1} \lambda_j=(1+o(1))\beta$, which together with~\eqref{eq stirling 1} implies that
\begin{align*}
&\sum_{i=1}^{|V_n|}\P{W_i\geq b_n^{-}+e_n}\\
\ge\; 
&\sum_{i=\lfloor |V_n| - n/\log n\rfloor}^{|V_n|}\P{W_i\geq b_n^{-}+e_n}\\
\ge\;
&\frac{n}{\log n}\exp{}((1+\log(\beta))(b_n^{-}+e_n)-(b_n^{+}+e_n)\log(b_n^{-}+e_n)+o(b_n^{-}+e_n)).   
\end{align*}
By a similar analysis to the one for~\eqref{eq log k_n}, 
\[(b_n^{-}+e_n)\log(b_n^{-}+e_n) = \log n + (\widehat C_{\beta}-\mu/2+o(1))\log n/\log\log n,\] 
and combined with $(1+\log(\beta)) (b_n^{-}+e_n) = (\widehat C_{\beta}+o(1))\log n/\log\log n$, this shows the first part in~\eqref{eq:sumclaims} for the case $\square = -$.

Next, we concentrate on the second part of~\eqref{eq:sumclaims}. 
Fix $i\in [|V_n|]$. Note that 
\[\max_{i\in [|V_n|]} \lambda_i = \max_{i,j\in [|V_n|]} p_{ij} = \cO(n^{-(1-p)/p}),\] 
and a union bound implies
\begin{equation}\label{eq fst part o}
\mathbb P\big(\max_{j\in \cS_i} P_i\ge \lceil \tfrac{1}{1-p}\rceil\big) = \cO((\max_{i\in [|V_n|]} \lambda_i)^{\lceil 1/(1-p)\rceil}) = \cO(n^{-1/p}) = o(n^{-1}).
\end{equation}
At the same time, for every $j\in \cS_i$, using that $\lambda_j = o(1)$, we obtain that
\begin{equation*}
\P{P_j\ge 2} = \sum_{k=2}^\infty \e^{-\lambda_j}\frac{\lambda_j^k}{k!}\le \lambda_j^2.
\end{equation*}
Hence, since $Z_n\ge e_n$ implies that $(\max_{j\in \cS_i} P_j) \cdot |\{j\in \cS_i: P_j\ge 2\}|\ge e_n$. In particular, either some of the $P_j$-s are at least $\lceil \tfrac{1}{1-p}\rceil$, or at least $(1-p)e_n$ of the $P_j$-s are at least $2$. 
Thus, by using~\eqref{eq fst part o},
\begin{align*}
\P{Z_i\ge e_n}
&\le\; \mathbb P\big(\{\max_{j\in \cS_i} P_j\ge \lceil \tfrac{1}{1-p}\rceil\big)\}\cup \{|\{j\in \cS_i: P_j\ge 2\}|\ge (1-p) e_n\}\big)\\
&\le\; o(n^{-1}) + \P{|\{j\in \cS_i: P_j\ge 2\}|\ge (1-p) e_n}.
\end{align*}
Finally, since $|\{j\in \cS_i: P_j\ge 2\}|$ is a sum of $|\cS_i|$ independent indicators with mean at most 
$$\Lambda_i := \sum_{j\in \cS_\lyu{i}} \lambda_j^2 = \cO\bigg(\max_{i,j\in |V_n|} p_{ij} \lyu{\sum_{j\in \cS_i} \lambda_j\bigg)} = \cO(n^{-(1-p)/p}),$$
Chernoff's inequality (Lemma~\ref{lem chern}(i)) implies that
\begin{align*}
&\P{Z_i\ge (1-p)e_n}\le\; o(n^{-1})+\exp{}\Big(-\Lambda_i\phi\Big(\frac{(1-p)e_n}{\Lambda_i}\Big)\Big)\\
=\;
&o(n^{-1}) + \exp{}\Big(-\Omega(e_n \log(n^{(1-p)/p}))\Big) = o(n^{-1}) + \exp{}\Big(-\Omega\Big(\big(\tfrac{\log n}{\log\log n}\big)^2\Big)\Big) = o(n^{-1}),
\end{align*}
which finishes the proof of~\eqref{eq:sumclaims}, and the lower bounds in~\eqref{eq ineq lb} and~\eqref{eq ineq lb1} follow. \lyu{In turn, this proves~\eqref{eq:2ndords} and~\eqref{eq:2ndordmin}.}\\

Finally, we prove~\eqref{eq:loc} in Theorem~\ref{thrm:max intro}, starting with the first statement. Recall $b_n^+$ from~\eqref{eq:bn}. 
By Lemma~\ref{lemma sandwich}, it is sufficient to show that, for every $\mu > 0$, none of the first $\exp(\log n - 2\mu \log n/\log\log n)$ vertices in $\vv{G}_{2,n}$ have degree more than $b_n^+$. 
Indeed, by Lemma~\ref{lemma:maxsum} and~\eqref{eq:resultclaims}, we know that w.h.p.\ $\vv{G}_{1,n}$ has maximum in-degree larger than $b_n^+$ and $\vv{G}_{1,n}\subseteq \vv{G}_n\subseteq \vv{G}_{2,n}$. 
At the same time, by replacing $a_n$ with $b_n^+$ in ~\eqref{eq:chernub} and~\eqref{eq log k_n}, we have that
$$\max_{i\in [|V_n|]} \P{d_{2,n}^s(v_i)\ge b_n^+} \le \exp{}\Big(- \log n + \frac{(\mu+o(1))\log n}{\log\log n}\Big),$$
so a union bound over the first $\exp(\log n - 2\mu \log n/\log\log n)$ vertices in $\vv{G}_{2,n}$ shows that w.h.p.\ $\log(\min \cI_n^{\square}) = (1-o(1))\log n$ in each of the cases $\square\in \{s, +\}$.

For the second statement in~\eqref{eq:loc}, fix any $\delta\in (0, 2\eps)$. Then,  for every $i\in [\delta n, |V_n|]$ and $t>0$, similarly to~\eqref{eq need in (3)}, we have that
\begin{equation}
\P{d_{2,n}^s(v_i)\ge b_n^+}\le \exp{}\Big(-tb_n^+ + (1+\cO(n^{-1/3}))(\e^t-1)\frac{\beta p}{1-p}\Big(1 - \frac{2\eps}{p} \Big(\frac{\delta}{2\eps}\Big)^{1-2\eps/p}\Big)\Big).
\end{equation} 
Then, by choosing 
\begin{equation}
t = \log b_n^+-\log\Big(\frac{\beta p}{1-p} \Big(1- \frac{2\eps}{p} \Big(\frac{\delta}{2\eps}\Big)^{1-2\eps/p}\Big)\Big),
\end{equation} 
mimicking~\eqref{eq:chernub} and~\eqref{eq log k_n}, and using a union bound over $|V_n|\leq n$ many vertices, we arrive at
\begin{equation}
\P{\max_{i\in [\delta n,|V_n|]}d^s_{2,n}(v_i)\geq b_n^+}\leq (1-\delta)\exp\Big(b_n^+\Big(\mu+\log\Big(1-\tfrac{2\eps}{p}\Big(\tfrac{\delta}{2\eps}\Big)^{1-2\eps/p}\Big)+o(1)\Big)+\mathcal O(1)\Big).
\end{equation}
As the logarithm is strictly negative for any choice of $\eps\in(0,1/2)$ and $\delta\in(0,2\eps)$, it follows that there exists a sufficiently small $\mu$ such that the upper bound converges to zero with $n$, which implies the second statement in~\eqref{eq:loc}.

Finally, we concentrate on the third claim in~\eqref{eq:loc}. Fix any $\delta\in (0,2\eps)$. As in the proof of the first claim of~\eqref{eq:loc}, using that the events $\{\vv{G}_{1,n}\subseteq \vv{G}_n\subseteq \vv{G}_{2,n}\}$ and $\{\max_{i\in [|V_n|]} d_{1,n}^-(v_i)\ge b_n^-\}$ hold w.h.p., we obtain that
\begin{align}
&\mathbb P\Big(\cI^-_n\cap \{v_1,\ldots,v_{\lceil \delta n\rceil}\}\neq \emptyset\Big)\nonumber\\
&\leq {}
\bigg(\sum_{i=1}^{\lceil \delta n\rceil} \mathbb P\Big(d_{2,n}^-(v_i)\geq b_n^-\Big)\bigg)+\mathbb P\Big(\max_{i\in [|V_n|]}d_{1,n}^-(v_i) < b_n^-\Big)+\P{\{\vv{G}_{1,n}\subseteq \vv G_n\subseteq \vv G_{2,n}\}^c}.\hspace{3em}\label{eq 3 terms}
\end{align}
The last two terms both tend to 0 as $n\to \infty$ by~\eqref{eq:2ndordmin} and Lemma~\ref{lemma sandwich}. \bas{For the first term, we again use that $d^-_{2,n}(v_i)$ has a binomial distribution with $i-1$ trials and success probability $p_{1i}$, with mean $m_{i,n}$ as in~\eqref{eq:dminmean}. By using the same approach as in~\eqref{eq d_2,n minus}, we can show that, for every sufficiently small $\mu > 0$ and for every $i\in [\lceil \delta n\rceil]$, we have $\P{d_{2,n}^-(v_i)\geq b_n^-} = o(n^{-1})$. The  sum 
in~\eqref{eq 3 terms} thus also tends to 0 as $n\to \infty$, which proves the last statement in Theorem~\eqref{thrm:max intro} and concludes the proof.}
\end{proof}

\section{\texorpdfstring{Conditional concentration of Lipschitz-type statistics: proof of Theorem~\ref{thm concentration intro}}{}}\label{sec:lipschitz}

In this section, we prove Theorem~\ref{thm concentration intro} by applying Azuma's inequality (Lemma~\ref{lemma:chern}(ii)) to some suitable martingales with bounded differences. 
For the first part, we set up our martingale conditionally on $|V_n|$ as well as the out-degree sequence $(d_n^-(v_i))_{i=1}^{|V_n|}$ of $\vv{G}_n$.
Then, at each step, the second end of one unmatched edge is revealed. 
For the second part, conditionally on $(V_i)_{i=1}^n$ and $|E_n| = |E(\vv{G}_n)|$, we expose the edges of $\vv{G}_n$ one by one and with a suitable probability (which is different for different edges, and is determined by the process $(V_i)_{i=1}^n$). 
Our results are inspired by Theorem 2.19 in~\cite{Wor99}; however, despite the fact that the proof method is similar, we do not condition on the entire degree sequence of $G_n$, as opposed to~\cite{Wor99}.

Recall that a function is $L$-Lipschitz if for any two (directed or undirected) graphs $G_1$ and $G_2$ that differ in only one edge (that is, $|E(G_1)\setminus E(G_2)|+|E(G_2)\setminus E(G_1)|\le 1$) we have $|f(G_1) - f(G_2)|\le L$. Theorem~\ref{thm concentration intro} follows immediately from the following two propositions.

\begin{proposition}\label{prop mg conc 1}
For every $L$-Lipschitz function $f$ defined on the set of directed graphs and every $t\ge 0$,
\begin{equation*}
\P{\Big|f(\vv{G}_n) - \mathbb E\left[f(\vv{G}_n)\mid |V_n|, (d_n^-(v_i))_{i=1}^{|V_n|}\right]\Big|\ge t\;\Big|\; |V_n|, (d_n^-(v_i))_{i=1}^{|V_n|}}\le 2\exp\left(-\frac{t^2}{8 |E_n| L^2}\right).
\end{equation*}
\end{proposition}
\begin{proof}
Order the out-going half-edges of $\vv{G}_n$ in increasing order with respect to the label of their starting endvertex, ties being broken arbitrarily. Then, consecutively connect each of these out-going half-edges to a vertex with smaller label in such a way that no double edges are formed. 
Observe that, conditionally on $|V_n|$ and $(d_n^-(v_i))_{i=1}^{|V_n|}$, this stochastic algorithm outputs the graph distributed as $\vv{G}_n$.
Let $(\cF_i)_{i=1}^{|E_n|}$ be the natural filtration associated to above algorithm conditionally on $|V_n|$ and $(d_n^-(v_i))_{i=1}^{|V_n|}$.
We define the martingale $(X_i)_{i=1}^{|E_n|}$ by setting $X_i = \mathbb E[f(\vv{G}_n)\mid \cF_i]$. We prove that
\begin{equation}\label{eq claim}
\forall i\in [|E_n|-1], |X_{i+1} - X_i| \le 2L.
\end{equation}

To show~\eqref{eq claim}, fix any $i\in [|E_n|-1]$ and let $e_i^-$ be the outgoing edge from a vertex $u_i$ that is matched at step $i$ of the algorithm.
Also, let $(v_j^{i})_{j=1}^{t_i}$ be the vertices to which $e_i^-$ may be matched conditionally on $\cF_i$. 
Finally, for every $j\in [t_i]$, let $\vv{G}_n^{i}$ be the partially constructed directed graph up to step $i$ and let $\cS_j$ denote the family of graphs containing $\vv{G}_n^{i}\cup \{u_iv_j^{i}\}$. 
Then, for every pair $(j_1, j_2)$ of different positive integers among $[t_i]$, $|\cS_{j_1}| = |\cS_{j_2}|$ since at every step, the number of choices for attaching the next out-going half-edge is fixed. Moreover, the map from $\cS_{j_1}\setminus \cS_{j_2}$ to $\cS_{j_2}\setminus \cS_{j_1}$ that deletes the edge $u_i v_i^{j_1}$ and instead constructs $u_i v_i^{j_2}$ is bijective. As $f$ is $L$-Lipschitz, the average of $f$ over different classes is the same up to $2L$ (the factor 2 comes from the fact that one edge is deleted and another is constructed). We conclude that
\begin{equation*}
|X_{i+1} - X_i| \le \max_{j_1, j_2\in [t_i]} |\mathbb E[f(\vv{G}_n)\mid \vv{G}_n\in \mathcal S_{j_1}] - \mathbb E[f(\vv{G}_n)\mid \vv{G}_n\in \mathcal S_{j_2}]| \le 2L.
\end{equation*}
Applying Azuma's inequality and using~\eqref{eq claim} finishes the proof.
\end{proof}

The upper bound in Proposition~\ref{prop mg conc 1} has the disadvantage of being rather constrained in terms of the graph structure as the entire out-degree sequence is exposed. Our next proposition avoids conditioning on the structure of the graph at the price of having complete information about the birth-death process $(V_i)_{i=1}^n$.

\begin{proposition}\label{prop mg conc 2}
For any $L$-Lipschitz function $f$ defined on the set of directed graphs and every $t\ge 0$,
\begin{equation*}
\P{\Big|f(\vv{G}_n) - \mathbb E\left[f(\vv{G}_n)\mid (V_i)_{i=1}^n, |E_n|\right]\Big|\ge t\;\Big|\; (V_i)_{i=1}^n, |E_n|}\le 2\exp\left(-\frac{t^2}{8 |E_n| L^2}\right).
\end{equation*}
\end{proposition}
\begin{proof}
Let us consider the following process conditionally on $(V_i)_{i=1}^n$ and $E_n$. We start with an empty graph on $V_n$ and, at each of $|E_n|$ rounds, we add one edge to the graph. 
More precisely, for every $i\in [|E_n|]$, conditionally on the set $\cE_{i-1} := \{e_1, \ldots, e_{i-1}\}$ of already exposed edges, we add an edge $e_i = u_iv_i \notin \cE_{i-1}$ (where $u_i < v_i$) with probability proportional to $\beta/|V_{v_i}|$. 
Setting $X_i = \mathbb E[f(\vv{G}_n)\mid \cE_i]$ for all $i\in [|E_n|]$, we have that $(X_j)_{j=0}^{E_n}$ is a martingale. Moreover, by replacing the families $(\cS_j)$ from the proof of Proposition~\ref{prop mg conc 1} by the families 
$$\cS_e := \left\{\vv{G}\subseteq \{0,1\}^{\binom{V_n}{2}}: \{e_1, \ldots, e_{i-1}, e\}\subseteq E(\vv{G})\right\}$$
for all $e\notin \cE_{i-1}$, we similarly deduce that, for all $i\in [|E_n|]$, we have $|X_{i-1} - X_i|\le 2L$.
Again, the proof is completed by the use of Azuma's inequality.
\end{proof}

\section{Conclusion}
In this paper, we defined and studied in depth a new model of a dynamic random graph with vertex removal. Some questions remain open. 
\begin{itemize}
\item Perhaps the most interesting question that we did not answer is to determine the first term in the expression of the diameter \lyu{of the giant component $\cC_1$} in $G_n$. By Lemma~\ref{lemma sandwich} and Theorem~\ref{thm BJR typical distance}, we know that w.h.p. $\mathrm{diam}(\cC_1) = \Omega(\log n)$.
\item \bas{As discussed in Remark~\ref{rem:brw}, we were not able to explicitly determine the survival probability $\gamma$ of the local weak limit $(\cT,0)$, nor $\beta_c$. Due to the non-identical offspring distribution of $\cT$, these seem rather challenging questions.}
\item We do not know the expected distance between two vertices of $G_n$ chosen uniformly at random \lyu{in the giant component}. It is natural to believe that it satisfies Part~3 of Theorem~\ref{thm approximate intro} with the same constant $\zeta$.
\item Another natural question is to determine the order of the largest component in $G_n$ when $\beta < \beta_c$. The fact that \lyu{our kernel of interest} $\kappa$ is unbounded over $(0,1]^2$ does not allow us to use further results from~\cite{BolJanRio05} to establish a more precise expression for $|\cC_1(G_n)|$. We conjecture that w.h.p.\ $|\cC_1| = \Theta(\log n)$.
\item It would also be interesting to provide a better description of the sets of vertices $\cI^{\square}_n$ defined in Theorem~\ref{thrm:max intro}. Although our proof shows bounds that are a bit stronger than the ones stated in Theorem~\ref{thrm:max intro}, we do not provide any details about the size of these sets as well as the positions of the vertices in the bulk (in case $|\cI^{\square}_n| = \omega(1)$).
\end{itemize}

\paragraph{Acknowledgements.}
We are thankful to Dieter Mitsche for useful discussions and suggestions,
and to the two anonymous referees for multiple suggestions, corrections and improvements.

\paragraph{Funding.}
Josep D\'iaz is supported by grant MOTION, PID2020-112581GB-C21 from MCIN/ AEI / 10.13039/501100011033. Bas Lodewijks has been supported by grant GrHyDy ANR-20-CE40-0002.

\bibliographystyle{amsplain}
\bibliography{Bib}

\end{document}